\documentclass[12pt]{amsart}

\usepackage{enumerate}
\usepackage[foot]{amsaddr}
\usepackage[english]{babel}
\usepackage{changepage}

\allowdisplaybreaks

\usepackage[
pdftex,hyperfootnotes]{hyperref}
\hypersetup{
	colorlinks=true,
	linkcolor=NavyBlue, 
	urlcolor=RoyalPurple,
	citecolor=OliveGreen,
	pdftitle={Bidiagonal factorization of Hessenberg or Banded matrices},
	bookmarks=true,
}
\usepackage[usenames,dvipsnames,svgnames,table,x11names]{xcolor}
\usepackage{pgfplots}
\usepackage{tikz}
\usepackage{tikz-3dplot}
\usetikzlibrary{automata,quotes, chains,matrix,calc,shadows,shapes.callouts,shapes.geometric,shapes.misc,positioning,patterns,decorations.shapes,
	decorations.pathmorphing,decorations.markings,decorations.fractals,decorations.pathreplacing,shadings,fadings,arrows.meta,bending}

\usepackage{nicematrix,drawmatrix}
\usepackage[utf8]{inputenc}

\usepackage{comment}

\usepackage[textwidth=17cm,textheight=23cm,
height=
24cm,
width=18cm]{geometry}

\usepackage{amssymb,latexsym,amsmath,amsthm,bm}
\usepackage{mathrsfs}
\usepackage{mathtools,mathdots}

\usepackage{tcolorbox}

\usepackage[table]{xcolor}

\usepackage{Baskervaldx}
\usepackage[]{newtxmath}

\usepackage{tensor}
\usepackage{bigints}

\theoremstyle{plain}

\newtheorem{teo}{Theorem}[section]
\newtheorem{coro}[teo]{Corollary}
\newtheorem{lemma}[teo]{Lemma}
\newtheorem{pro}[teo]{Proposition}
\newtheorem{defi}[teo]{Definition}
\newtheorem{rem}[teo]{Remark}

\renewcommand{\d}{\operatorname{d}}

\newcommand{\B}{\mathbb{B}}

\newcommand{\N}{\mathbb{N}}
\newcommand{\R}{\mathbb{R}}

\newcommand{\sz}[1]{\left| \vec{#1} \right|}

\allowdisplaybreaks[1]

\makeatletter
\DeclareRobustCommand{\gaussk}{\DOTSB\gaussk@\slimits@}
\newcommand{\gaussk@}{\mathop{\vphantom{\sum}\mathpalette\bigcal@{K}}}

\newcommand{\bigcal@}[2]{%
	\vcenter{\m@th
		\sbox\z@{$#1\sum$}%
		\dimen@=\dimexpr\ht\z@+\dp\z@
		\hbox{\resizebox{!}{0.8\dimen@}{$\mathcal{K}$}}%
	}%
}
\newcommand{\cfracplus}{\mathbin{\cfracplus@}}
\newcommand{\cfracplus@}{%
	\sbox\z@{$\dfrac{1}{1}$}%
	\sbox\tw@{$+$}%
	\raisebox{\dimexpr\dp\tw@-\dp\z@\relax}{$+$}%
}
\newcommand{\cfracdots}{\mathord{\cfracdots@}}
\newcommand{\cfracdots@}{%
	\sbox\z@{$\dfrac{1}{1}$}%
	\sbox\tw@{$+$}%
	\raisebox{\dimexpr\dp\tw@-\dp\z@\relax}{$\cdots$}%
}
\makeatother

\makeatletter
\newcommand*{\relrelbarsep}{.386ex}
\newcommand*{\relrelbar}{%
	\mathrel{%
		\mathpalette\@relrelbar\relrelbarsep
	}%
}
\newcommand*{\@relrelbar}[2]{%
	\raise#2\hbox to 0pt{$\m@th#1\relbar$\hss}%
	\lower#2\hbox{$\m@th#1\relbar$}%
}
\providecommand*{\rightrightarrowsfill@}{%
	\arrowfill@\relrelbar\relrelbar\rightrightarrows
}
\providecommand*{\leftleftarrowsfill@}{%
	\arrowfill@\leftleftarrows\relrelbar\relrelbar
}
\providecommand*{\xrightrightarrows}[2][]{%
	\ext@arrow 0359\rightrightarrowsfill@{#1}{#2}%
}
\providecommand*{\xleftleftarrows}[2][]{%
	\ext@arrow 3095\leftleftarrowsfill@{#1}{#2}%
}
\makeatother

\newcommand*\pFqskip{8mu}
\catcode`,\active
\newcommand*\pFq{\begingroup
	\catcode`\,\active
	\def ,{\mskip\pFqskip\relax}%
	\dopFq
}
\catcode`\,12
\def\dopFq#1#2#3#4#5{%
	 \tensor*[_{#1}]{F}{_{#2}} 
	 \left[\genfrac..{0pt}{}{#3}{#4};#5\right]%
	\endgroup
}

\usepackage{appendix}

\usepackage{xcolor} 

\usepackage[normalem]{ulem}
\usepackage{soul} 

\usepackage{verbatim}

\tikzstyle{block} = [draw, rectangle, 
minimum height=3em, minimum width=2em]

\begin{document}
	\title[Bidiagonal Factorization of Banded Matrices]
	{Bidiagonal Factorization of Banded Recursion Matrices for Mixed-Type Multiple Orthogonal POlynomials}
	
	\author[A Branquinho]{Amílcar Branquinho\(^{1}\)}
	\address{\(^1\)CMUC, Departamento de Matemática,
		Universidade de Coimbra, 3001-454 Coimbra, Portugal}
	\email{\(^1\)ajplb@mat.uc.pt}
	

	\author[A Foulquié]{Ana Foulquié-Moreno\(^{3}\)}
	\address{\(^3\)CIDMA, Departamento de Matemática, Universidade de Aveiro, 3810-193 Aveiro, Portugal}
	\email{\(^3\)foulquie@ua.pt}

	\author[M Mañas]{Manuel Mañas\(^{4}\)}
	\address{\(^4\)Departamento de Física Teórica, Universidad Complutense de Madrid, Plaza Ciencias 1, 28040-Madrid, Spain 
	}
	\email{\(^4\)manuel.manas@ucm.es}
	
	
	\maketitle
\begin{abstract}
Given a banded matrix $\mathscr{T}_N$ with $p$ subdiagonals and $q$
superdiagonals arising from the Gauss--Borel factorization
$\mathscr{M}_N = \mathscr{L}_N^{-1}\mathscr{U}_N^{-1}$ of a moment
matrix, this paper constructs explicitly its bidiagonal factorization
\[
\mathscr{T}_N = L_1 \cdots L_p\, U_q \cdots U_1.
\]
Bidiagonal factorizations of this type are central to the study of
oscillatory banded matrices and to the spectral Favard theorem for
multiple orthogonal polynomials \cite{aim}.

The factorization is obtained via Christoffel transformations of the
moment matrix. Provided that the perturbed moment matrices
$\mathscr{M}_{N,(b,0)}$ and $\mathscr{M}_{N,(0,a)}$ admit a
Gauss--Borel factorization, each bidiagonal factor is a quotient of
the corresponding Gauss--Borel factors:
\[
U_b = \mathscr{U}_{N,(b,0)}^{-1}\mathscr{U}_{N,(b-1,0)},
\qquad
L_a = \mathscr{L}_{N,(0,a-1)}\mathscr{L}_{N,(0,a)}^{-1}.
\]
Explicit Christoffel-type formulas for the entries of the bidiagonal
factors are then derived in terms of certain tau-determinants evaluated at
the origin:
\[
U_{b,n} = -\frac{\tau^B_{b-1,n}\,\tau^B_{b,n+1}}
{\tau^B_{b-1,n+1}\,\tau^B_{b,n}},
\qquad
L_{a,n+1} = -\frac{\tau^A_{a-1,n+2}\,\tau^A_{a,n}}
{\tau^A_{a-1,n+1}\,\tau^A_{a,n+1}}.
\]
As an illustration, the theory is applied to the recurrence matrices
of multiple Hahn orthogonal polynomials. For two weights the
tetradiagonal case is handled via contiguous hypergeometric relations
\cite{LAA,JMAA}; for three weights, i.e. the pentadiagonal case, the direct hypergeometric
representations of \cite{BDFMW} are required. In both cases fully
explicit bidiagonal factorizations are obtained.
\end{abstract}

\section{Introduction}

The bidiagonal factorization of banded matrices plays a central role
in two related but distinct directions. The nonneg\-ative bidiagonal
factorization is a fundamental tool in the structural theory of total
positivity of matrices, providing an explicit and computable
description of totally nonneg\-ative banded matrices in terms of
elementary factors \cite{fallat-johnson,pinkus}. The positive
bidiagonal factorization, in turn, leads to a spectral Favard theorem
for banded matrices in terms of mixed-type multiple orthogonal
polynomials \cite{aim}, and has been applied to construct explicit
stochastic factorizations of finite-time Markov chains
\cite{finite-Markov,Markov JP}. Constructing bidiagonal
factorizations explicitly is therefore of direct relevance both for
the theory of totally positive matrices and for the spectral analysis
of banded operators.

The setting of this paper is the following. A rectangular $q\times p$
matrix of measures $\mathrm{d}\mu$ determines a moment matrix
$\mathscr{M}_N$ whose Gauss--Borel factorization
\[
\mathscr{M}_N = \mathscr{L}_N^{-1}\mathscr{U}_N^{-1}
\]
produces two dual families of mixed-type multiple orthogonal
polynomials $B^{[N]}$ and $A^{[N]}$, together with a banded
recurrence matrix $\mathscr{T}_N$ with $p$ subdiagonals and $q$
superdiagonals. This framework was introduced in \cite{cum1,cum2} and
further developed in \cite{Manuel-Miguel,Manuel-Miguel_2}.

The first main result (Theorem~\ref{th:Bidiagonal_Factorization})
gives the bidiagonal factorization $\mathscr{T}_N = L_1\cdots
L_p\,U_q\cdots U_1$, under the condition that the Christoffel
perturbations $\mathscr{M}_{N,(b,0)}$ and $\mathscr{M}_{N,(0,a)}$
of the moment matrix admit a Gauss--Borel factorization. The
bidiagonal factors are then quotients of the corresponding
Gauss--Borel factors:
\[
U_b = \mathscr{U}_{N,(b,0)}^{-1}\mathscr{U}_{N,(b-1,0)},
\qquad
L_a = \mathscr{L}_{N,(0,a-1)}\mathscr{L}_{N,(0,a)}^{-1}.
\]
This establishes a precise correspondence: the bidiagonal
factorization of the recurrence matrix reflects exactly the
Gauss--Borel structure of the Christoffel-transformed moment
matrices.

The second main result (Theorem~\ref{th:bidiagonal_Christoffel})
makes this fully explicit. The entries of the bidiagonal factors are
expressed as cross-ratios of tau-determinants:
\begin{equation}\label{eq:intro-formulas}
	U_{b,n} = -\frac{\tau^B_{b-1,n}\,\tau^B_{b,n+1}}
	{\tau^B_{b-1,n+1}\,\tau^B_{b,n}},
	\qquad
	L_{a,n+1} = -\frac{\tau^A_{a-1,n+2}\,\tau^A_{a,n}}
	{\tau^A_{a-1,n+1}\,\tau^A_{a,n+1}},
\end{equation}
where $\tau^B_{b,n}$ and $\tau^A_{a,n}$ are $b\times b$ and $a\times
a$ determinants of values of the mixed-type polynomials at $x=0$.
The cross-ratio structure in \eqref{eq:intro-formulas} is the matrix
counterpart of the classical Christoffel formula for scalar
orthogonal polynomials \cite{christoffel}, and the condition for the
factorization to exist --- the nonvanishing of the tau-determinants
--- coincides with the condition for the transformed orthogonality
to be well-defined \cite{Manuel-Miguel}.

Section~6 subjects these formulas to a substantial computational
test on multiple Hahn orthogonal polynomials. For two weights the
recurrence matrix is tetradiagonal, and the tau-determinants can be
evaluated using contiguous hypergeometric relations from
\cite{LAA,JMAA}; the bidiagonal factorization of tetradiagonal
matrices in this context was studied in \cite{darboux}. For three
weights this approach is no longer viable; the direct hypergeometric
representations of \cite{BDFMW} are required, and a more involved
computation delivers the explicit factorization. The hypergeometric
functions appearing throughout follow the standard notation of
\cite{slater}, and analogous representations for other classical
families are available in \cite{PAMS,{JP_friends}}.

The paper is organized as follows. Section~2 introduces moment
matrices, their Gauss--Borel factorization, and the mixed-type
multiple orthogonal polynomials. Section~3 defines the banded
recurrence matrices. Section~4 constructs the bidiagonal
factorization via Christoffel transformations. Section~5 derives the
explicit Christoffel formulas for the bidiagonal factors. Section~6
applies the theory to multiple Hahn polynomials with two and three
weights.
	\subsection{Moment matrices}
	We will study orthogonality with respect a rectangular matrix of measures, namely
\begin{defi}[The matrix of measures]
		Given two natural  numbers $p,q\in\N$, we consider a  rectangular $q\times p $ matrix of measures
	\[
	\mu=\begin{bNiceArray}{ccc}[first-row,last-col]
		\Hbrace{3}{p}\\
		\mu_{1,1}&\Cdots&\mu_{1,p}&\Vbrace{6}{q}\\
		\Vdots& &\Vdots\\\\\\\\
			\mu_{q,1}&\Cdots&\mu_{q,p}
	\end{bNiceArray}\hspace*{20pt}.
	\]
\end{defi}
First, we need to introduce moment matrices of finite size for this matrix of measures and characterize its symmetries.
We collect all the monomials in the following matrices 
\begin{defi}
	We consider the following semi-infinite matrices on monomials
	\[
X_{[r]}=\begin{bNiceArray}{c}[first-row]
	\Hbrace{1}{r}\\
	I_r\\
	xI_r\\
	\Vdots[shorten-end=2.5pt]\\\\
\end{bNiceArray}\in \R^{\infty\times r}[x].
\]
as well as its $N$ truncations, $N\in\N$, $X^{[N]}_{[r]}$. 
\end{defi}

 For $r\in\N$ we consider the Euclidean division 
 $N=N_rr+s_r$, with $s_r\in\{0,1,\dots, r-1\}$,  $N_r=\left\lfloor\frac{N}{r}\right\rfloor$ and
\[
I_{s,r}\coloneq\begin{bNiceArray}{ccccccc}[last-col,first-row,last-row]
	\Hbrace{7}{r}\\
	1&0&\Cdots&&&&0&\Vbrace{4}{s}\\
	0&1&0&\Cdots&&&0\\
	\Vdots &\Ddots&\Ddots&\Ddots&&&\Vdots\\
	0&\Cdots&0& 1&0&\Cdots&0\\
	\Hbrace{4}{s}&\Hbrace{3}{r-s}
\end{bNiceArray}\hspace*{15pt}.
\]
\begin{pro}
	The truncated matrix of monomials  are given by
	\[
	X^{[N]}_{[r]}=\begin{bNiceArray}{c}[first-row,last-col]
		\Hbrace{1}{r}\\
		I_r&\Vbrace{4}{N_r r}\\
		xI_r\\
		\Vdots\\
		x^{N_r-1} I_r\\
	x^{N_r}I_{s_r,r}	 &\Vbrace{1}{s_r}
	\end{bNiceArray}\hspace*{20pt}\in \R^{N\times r}[x].
	\]
\end{pro}
	These matrices  of monomials have important properties with respect the following matrices:
\begin{defi}[Shift matrices]
		The  shift matrix is
	\[
	\Lambda=\begin{bNiceArray}{ccccc}
		0&1&0&\Cdots[shorten-end=-5pt]&\phantom{h}\\
		0&0&1&\Ddots&\\
		\Vdots[shorten-end=-10pt]&\Ddots[shorten-end=-25pt]&\Ddots[shorten-end=-20pt]&\Ddots[shorten-end=-0pt]&\\
		\phantom{h}&\phantom{h}&\phantom{h}&\phantom{h}
		\end{bNiceArray}, 
	\]
	in terms of which we have the corresponding block shift matrix
		\[
	\Lambda_{[r]}=\begin{bNiceArray}{ccccc}
		0_{r}&I_{r}&0_r&\Cdots[shorten-end=-5pt]&\phantom{h}\\
		0_r&0_r&I_r&\Ddots&\\
		\Vdots[shorten-end=-10pt]&\Ddots[shorten-end=-15pt]&\Ddots[shorten-end=-15pt]&\Ddots[shorten-end=-0pt]&\\
		\phantom{h}&\phantom{h}&\phantom{h}&\phantom{h}
	\end{bNiceArray}=\Lambda^r.
	\]
\end{defi}
Our focus will be in truncated versions of these matrices as follows:

\begin{pro}[Truncated block shift matrices]
		Truncations $\Lambda^{[N]}_{[r]}\in\R^{N\times N}$   of this shift matrix are 	
	\begin{align*}
		\Lambda^{[N]}_{[r]}&=\begin{bNiceArray}{ccccc|c}[first-row, last-col,margin,last-row]
			\Hbrace{6}{N}\\
			0_r & I_r&0_r&\Cdots &0_r&0_{r\times s_r}&\Vbrace{6}{N}\\
			\Vdots &\Ddots&\Ddots&\Ddots &\Vdots&\Vdots\\
			&&& &0_r&0_{r\times s_r}\\
			&&&&I_r&0_{r\times s_r}\\
					&&&&0_r&I_{s_r,r}^\top\\
			0_{s_r\times r}&\Cdots &&&0_{s_r\times r}&0_{s_r}\\
			\Hbrace{5}{N_r r}&\Hbrace{1}{s_r}
		\end{bNiceArray}
		\end{align*}
\end{pro}
We will also need to border these truncated shift matrices
\begin{defi}[Bordered shift matrices]
	Bordered truncated matrices  $\Lambda^{[N,N+r]}_{[r]}\in\R^{N\times (N+r)}$ are obtained by augmenting with $r$ columns as follows
	\begin{align*}
	\Lambda^{[N,N+r]}_{[r]}&=\left[\begin{NiceArray}{ccccc|cc}
		0_r & I_r&0_r&\Cdots &0_r&0_{r\times s_r}&\Block[borders={left}]{4-1}{0_{(N-r)\times r}}\\
		\Vdots &\Ddots&\Ddots&\Ddots &\Vdots&\Vdots\\
		&&& &0_r&0_{r\times s_r}&\\
		&&&&I_r&0_{r\times s_r}&\\
		&&&&0_r&I_{s_r,r}^\top&\Block[borders={top,left}]{2-1}{I_r}\\ 
		0_{s_r\times r}&\Cdots &&&0_{s_r\times r}&0_{s_r}&
	\end{NiceArray}\right].
\end{align*}
We have the key spectral type relations
\begin{equation}\label{eq:spectralX}
			\Lambda^{[N,N+r]}_{[r]}	X^{[N+r]}_{[r]}(x)=x X^{[N]}_{[r]}(x).
\end{equation}
\end{defi}
We first give an slightly different view of these bordered truncated shift  matrices
\begin{pro}
		The bordered truncated shift matrices  $\Lambda^{[N,N+r]}_{[r]}\in\R^{N\times (N+r)}$ have the following expression
	\begin{align*}
		\Lambda^{[N,N+r]}_{[r]}&=\left[\begin{NiceArray}{ccccccw{c}{3cm}w{c}{3cm}}
			0_r & I_r&0_r&\Cdots &0_r&\Block[borders={left}]{4-2}<\Large>{0_{(N-r-s_r)\times (r+s_r)}}&\\
			\Vdots &\Ddots&\Ddots&\Ddots &\Vdots&&\\
			&&& &0_r&&\\
			&&&&I_r&&\\
			&&&&0_r&\Block[borders={top,left}]{2-2}<\large>{I_{r+s_r}}&\\ 
			0_{s_r\times r}&\Cdots &&&0_{s_r\times r}&&
		\end{NiceArray}\right].
	\end{align*}
\end{pro}
Then, we are ready to state the simple but key spectral property of these shift matrices with respect the matrix of monomials:
\begin{pro}[Monomial spectrality]
	We have the  spectral type relations:
	\begin{equation}\label{eq:spectralX}
		\Lambda^{[N,N+r]}_{[r]}	X^{[N+r]}_{[r]}(x)=x X^{[N]}_{[r]}(x).
	\end{equation}
\end{pro}
	Now, we are prepared to consider matrices of moments of the matrix of measures $\d\mu$:
\begin{defi}[Moment matrices]
		For $N,M\in\N$, we have the following matrices of moments of $\mu$:
	\[
\begin{aligned}
		\mathscr M^{[N,M]}&\coloneq\int X^{[N]}_{[q]}(x)\d\mu(x) \left(X^{[M]}_{[p]}(x)\right)^\top\\&=
		\bigints \begin{bNiceMatrix}
			I_q\\
			xI_q\\
			\Vdots\\
			x^{N_q-1} I_q\\
			x^{N_q}I_{s_q,q}
		\end{bNiceMatrix}\d\mu(x)\begin{bNiceMatrix}
		I_p&	xI_p&	\Cdots &		x^{M_p-1} I_p & x^{M_p}I_{s_p,p}^\top
		\end{bNiceMatrix}\\&=
		\begin{bNiceArray}{ccccc}[first-row,last-col,small]
			\Hbrace{5}{M}\\
			\int \d\mu(x) &\int x \d\mu(x) &\Cdots & \int x^{M_p-1}\d\mu(x)
			& \int x^{M_p}\d\mu(x)I_{s_p,q}^\top&\Vbrace{7}{N}\\
				\int x\d\mu(x) &\int x^2 \d\mu (x)&\Cdots & \int x^{M_p}\d\mu(x)	& \int x^{M_p+1}\d\mu(x)I_{s_p,p}^\top\\
			\Vdots & \Vdots&&\Vdots&\Vdots\\\\\\
				 \int x^{N_q-1}\d\mu(x)& \int x^{N_q}\d\mu(x)&\Cdots &  \int x^{N_q+M_p-2}\d\mu(x)&  \int x^{N_q+M_p-1} \d\mu(x)I_{s_p,p}^\top\\
				 \int x^{N_q}I_{s_q,q}\d\mu(x)& \int x^{N_q+1}I_{s_q,q}\d\mu(x)&\Cdots &  \int I_{s_q,q}x^{N_q+M_p-1}\d\mu(x)&  \int x^{N_q+M_p} I_{s_q,q}\d\mu(x)I_{s_p,p}^\top
		\end{bNiceArray}\hspace*{20pt}.
\end{aligned}
	\]
\end{defi}
Using the spectral properties of the matrix of monomials we  find the following block Hankel  type symmetries of these matrices of moments:
\begin{pro}[Hankel type symmetries]
	We have the following symmetries for the  moment matrices
	\[
	\Lambda^{[N,N+q]}_{[q]}\mathscr M^{[N+q,M]}=\mathscr M^{[N,M+p]} 	\left(\Lambda^{[M,M+p]}_{[p]}\right)^\top.
	\]
\end{pro}
\begin{proof}
	Is a direct consequence of Equation \eqref{eq:spectralX}. Indeed,
	\[
	\begin{aligned}
	\Lambda^{[N,N+q]}_{[q]}\mathscr M^{[N+q,M]}&=\Lambda^{[N,N+q]}_{[q]}\int X^{[N+q]}_{[q]}(x)\d\mu(x) \left(X^{[M]}_{[p]}(x)\right)^\top\\&=\int xX^{[N]}_{[q]}(x)\d\mu(x) \left(X^{[M]}_{[p]}(x)\right)^\top=\int X^{[N]}_{[q]}(x)\d\mu(x) \left(xX^{[M]}_{[p]}(x)\right)^\top\\
	&=\int X^{[N]}_{[q]}(x)\d\mu(x) \left(\Lambda^{[M,M+p]}_{[p]}X^{[M+p]}_{[p]}(x)\right)^\top=\mathscr M^{[N,M+p]} 	\left(\Lambda^{[M,M+p]}_{[p]}\right)^\top.
	\end{aligned}
	\]
\end{proof}
In particular, 

\begin{coro}
	We have  the following Hankel symmetry type relations
\begin{subequations}\label{eq:factorization}
	\begin{align}
			\label{eq:factorization.0}
		\Lambda^{[N,N+q]}_{[q]}\mathscr M^{[N+q,N]}&=\mathscr M^{[N,N+p]}	\left(\Lambda^{[N,N+p]}_{[p]}\right)^\top,\\
	\Lambda^{[N-q,N]}_{[q]}\mathscr M^{[N,N]}&=\mathscr M^{[N-q,N+p]}	\left(\Lambda^{[N,N+p]}_{[p]}\right)^\top,\\\label{eq:factorization.-+}
		\Lambda^{[N,N+q]}_{[q]}\mathscr M^{[N+q,N-p]}&=\mathscr M^{[N,N]}	\left(\Lambda^{[N-p,N]}_{[p]}\right)^\top.
\end{align}
\end{subequations}
\end{coro}

Note that for $N\neq M$ the matrix $\mathscr M^{[N,M]}$ is a rectangular matrix, and it happens to be a square matrix if and only if $N=M$. In this case we will use the notation \[\mathscr M^{[N,N]}\eqcolon\mathscr M_N.\]

\section{The Gauss-Borel factorization of the moment matrices}

We assume that the matrix of moments $\mathscr M_N$ admits a Gauss--Borel factorization
\[
\mathscr M_N=\mathscr L_N^{-1}\mathscr U_N^{-1}
\] 
in terms of nonsingular lower and upper triangular matrices  \(\mathscr L_{N}\) and \(\mathscr U_{N}\), respectively. This is equivalent to all  leading principal submatrices $\mathscr M_n$, $n\in\{1,\dots, N\}$, of $\mathscr M_{N}$ to be non singular. This is equivalent that for all these matrices submatrices there exists a Gauss--Borel factorization  
 \[
 \mathscr M_n=\mathscr L_n^{-1} \mathscr U_n^{-1},
 \]
 where $\mathscr L_n$ and $\mathscr U_n$, can be taken, for example, as the $n$-th leading principal submatrices of  $\mathscr L_N$ and $\mathscr U_N$, respectively.
 
\begin{defi}[Polynomial matrices]
	 We define the following  polynomials vectors 
 \[
 \begin{aligned}
 	B^{[N]}&\coloneq\mathscr L_{N} X_{[q]}^{[N]}, &
 	A^{[N]}&\coloneq\left( X_{[p]}^{[N]}\right)^\top\mathscr U_{N}.
 \end{aligned}
 \]
 We will use the notation
 \[
 \begin{aligned}
 	B^{[N]}&=\begin{bNiceArray}{ccc}[margin,cell-space-limits=1pt]
 		B^{(1)}_0&\Cdots &B^{(q)}_0\\
 		B^{(1)}_1&\Cdots &B^{(q)}_1\\
 		\Vdots& &\Vdots\\\\
 		B^{(1)}_{N-1}&\Cdots &B^{(q)}_{N-1}\\
 	\end{bNiceArray}, &
 	 	A^{[N]}&=\begin{bNiceArray}{ccccc}[margin,cell-space-limits=2pt]
 		A^{(1)}_0&	A^{(1)}_1&\Cdots &&A^{(1)}_{N-1}\\
 \Vdots	&\Vdots & &&\Vdots\\
 		A^{(p)}_0&	A^{(p)}_1&\Cdots &&A^{(p)}_{N-1}\\
 	\end{bNiceArray}, 
 \end{aligned}
 \]
\end{defi}
\begin{pro}
Whenever $\mathscr{M}_N$ admits a Gauss--Borel factorization the following relations hold:
 \NiceMatrixOptions{cell-space-top-limit=2pt, cell-space-bottom-limit=2pt}
 \begin{align*}
 	 \mathscr M^{[N, N+p]}&=\mathscr{L}_{N}^{-1}\begin{bNiceArray}{c|c}
 	\mathscr{U}_{N}^{-1} &U^{[N,N+p]}
 	\end{bNiceArray},\\
 	\mathscr M^{[N+q,N]}&=\left[\begin{NiceArray}{c}[margin,cell-space-limits=2pt]
 		\mathscr{L}_{N}^{-1} \\\hline L^{[N+q,N]}
 	\end{NiceArray}\right]\mathscr{U}_{N}^{-1},\\
 \mathscr M^{[N-q,N+p]}&=\mathscr{L}_{N-q}^{-1}\begin{bNiceArray}{c|c}
 \mathscr{U}_{N-q}^{-1} &U^{[N-q,N+p]}
 \end{bNiceArray},\\
 \mathscr M^{[N+q,N-p]}&=\left[\begin{NiceArray}{c}[margin,cell-space-limits=2pt]
\mathscr{L}_{N-p}^{-1} \\\hline L^{[N+q,N-p]}
 \end{NiceArray}\right]\mathscr{U}_{N-p}^{-1}
 \end{align*}
 where $U^{[N,N+p]}\in\R^{N\times p}$, $L^{[N+q,N]}\in\R^{ q\times N}$, 
 $U^{[N-q,N+p]}\in\R^{(N-q)\times (q+p)}$ and $L^{[N+q,N-p]}\in\R^{ (q+p)\times(N-p)}$.
\end{pro}
 We now introduce an important polynomial matrix:
 \begin{defi}
For $s\in\{0,1,\dots,r\}$, let us consider
 	\[
 	\mathfrak{X}_{[r,s]}(x)\coloneq\begin{bNiceArray}{w{c}{.2cm}w{c}{.5cm}w{c}{.2cm}|w{c}{.2cm}w{c}{.2cm}w{c}{.5cm}w{c}{.5cm}}[margin,cell-space-limits=2pt]
 		\Block[borders={bottom}]{6-3}<\large>{0_{ (r-s)\times s}}&&&
 		\Block[borders={bottom}]{6-4}<\large>{I_{(r-s)\times(r-s)}}&&&\\
 		\\
 		\\
 		\\\\\\
 		\Block{4-3}<\large>{x I_{s}}&&&
 		\Block{4-4}<\large>{0_{ s\times(r-s)}}&&&
 		\\
 		\\\\\\
 	\end{bNiceArray}.
 	\]
 \end{defi}
 
 This matrix enjoys some important properties:
 \begin{pro}We have
\begin{enumerate}
	\item 
 \[
 \mathfrak X_{[r,r]}=x I_r.\]
\item 
 \[
 \mathfrak{X}_{[r,s]}=\mathfrak{X}_{[r,1]}^s,
 \]
\item 
 \[
 \mathfrak{X}_{[r,s]}\mathfrak{X}_{[r,s']}=\mathfrak{X}_{[r,s']}\mathfrak{X}_{[r,s]}=\mathfrak{X}_{[r,s+s']}.
 \]
\end{enumerate}
 \end{pro}
 This matrix is useful to represent the matrix of monomials
\begin{pro}
It holds that
 \[
 X_{[r]}^{[N+r]}=\begin{bNiceMatrix}
 	 	X_{[r]}^{[N]}\\
 	 x^{N_r}	\mathfrak{X}_{[r,s_r]}
 \end{bNiceMatrix}
 \]
\end{pro}
It also allows for explicit expressions for block related to the truncations of the inverse of the triangular factors of the Gauss--Borel factorization

\begin{pro}
The following relations are fulfilled
	\begin{align}\label{eq:L}
		L^{[N+q,N]}&=\int x^{N_q} \mathfrak{X}_{[q,s_q]}(x)  \d\mu(x) A^{[N]}(x),\\	\label{eq:U}
		U^{[N,N+p]}&=\int B^{[N]}(x)\d\mu(x) \left(\mathfrak{X}_{[p,s_p]}(x)\right)^\top x^{N_p}.
	\end{align}
\end{pro}
\begin{proof}
First, we compute that
\[
\begin{aligned}
	\mathscr M^{[N+q,N]}&=\bigintss \begin{bNiceArray}{w{c}{2.5cm}}[margin,cell-space-limits=2pt]
\Block{3-1}{ X^{[N]}_{[q]}(x)}\\ \\\\\hline x^{N_q} \mathfrak{X}_{[q,s_q]}(x) 
\end{bNiceArray}\d\mu(x) \left(X^{[N]}_{[p]}(x)\right)^\top\\&=\begin{bNiceArray}{c}[margin,cell-space-limits=2pt]
\Block{3-1}{\mathscr M_N }\\ \\\\\hline x^{N_q}\left(\mathfrak{X}_{[q,s_q]}(x)\right)^\top \d\mu(x)\left(X^{[N]}_{[p]}(x)\right)^\top
\end{bNiceArray}\\&=
\begin{bNiceArray}{w{c}{5.5cm}}[margin,cell-space-limits=2pt]
\Block{3-1}{\mathscr L_N^{-1}} \\\\\\\hline\bigintssss x^{N_q} \mathfrak{X}_{[q,s_q]}(x)  \d\mu(x) A^{[N]}(x)
\end{bNiceArray}\mathscr U_N^{-1},\\
	\mathscr M^{[N,N+p]}&=\bigintssss X^{[N]}_{[q]}(x)\d\mu(x)\begin{bNiceArray}{w{c}{2.5cm}|c}
	\left(X^{[N]}(x)\right)^\top& x^{N_p} \left(\mathfrak{X}_{[p,s_p]}(x)\right)^\top
\end{bNiceArray}\\&=\begin{bNiceArray}{w{c}{2cm}|c}
	\mathscr M_N & \bigintssss X^{[N]}_{[q]}(x)\d\mu(x) x^{N_p}\left(\mathfrak{X}_{[p,s_p]}(x)\right)^\top
\end{bNiceArray}\\&=\mathscr L_N^{-1}
\begin{bNiceArray}{w{c}{2cm}|c}
	\mathscr U_N^{-1} & \bigintssss B^{[N]}(x)\d\mu(x) x^{N_p}\left(\mathfrak{X}_{[p,s_p]}(x)\right)^\top
\end{bNiceArray}.
\end{aligned}
\]

\end{proof}

 Being $\mathscr L_N$ and $\mathscr U_N$ lower and upper triangular non singular matrices, respectively, they have the block structure
 \[
 \begin{aligned}
 	\mathscr L_N&=
 	\begin{bNiceArray}{c|c}[margin,cell-space-limits=2pt]
 		\mathscr L_{N-q} & 0_{(N-q)\times q}\\\hline
 		\mathscr l^{[N,N-q]}&\mathscr L_{[N-q,N]}
 	\end{bNiceArray}, &
 	 	\mathscr U_N&=
	\begin{bNiceArray}{c|c}[margin,cell-space-limits=2pt]
 		\mathscr U_{N-p} & \mathscr u^{[N-p,N]} \\\hline
 	0_{p\times (N-p)}	&\mathscr U_{[N-p,N]}
 	\end{bNiceArray}, 
 	 \end{aligned}
 \]
 where $\mathscr L_{[N-q,N]}\in\R^{q\times q}$ and $\mathscr U_{[N-p,N]}\in\R^{p\times p}$ are lower and upper triangular non singular matrices, respectively. Hence, their inverses have the form
\begin{equation}\label{eq:inverses_direct}
	 \begin{aligned}
 	\mathscr L_N^{-1}&=
  	\begin{bNiceArray}{c|c}[margin,cell-space-limits=2pt]
 		\mathscr L_{N-q}^{-1}& 0_{(N-q)\times q}\\\hline
 	\tilde{\mathscr l}^{[N,N-q]}&\mathscr L_{[N-q,N]}^{-1}
 	\end{bNiceArray}, &
 	\mathscr U_N^{-1}&=
 	\begin{bNiceArray}{c|c}[margin,cell-space-limits=2pt]
 		\mathscr U_{N-p}^{-1} & \tilde{\mathscr u}^{[N-p,N]} \\\hline
 		0_{p\times (N-p)}	&\mathscr U_{[N-p,N]}^{-1}
 	\end{bNiceArray}.
 \end{aligned}
\end{equation}
For a better block representation of these inverses 
we write  \[\begin{aligned}
	A^{[N]}&=\begin{bNiceArray}{c|c}[]
 	A^{[N-q]} & A^{[N,q]}
 \end{bNiceArray}, &B^{[N]}&=\begin{bNiceMatrix}[margin,cell-space-limits=2pt]
 	B^{[N-p]} \\ \hline B^{[N,p]}
 \end{bNiceMatrix}, &A^{[N,q]},B^{[N,p]}&\in\R^{p\times q}[x],
\end{aligned}\]
 with
 \[
\begin{aligned}
	B^{[N,p]}&=\begin{bNiceArray}{ccc}[margin,cell-space-limits=1pt]
			B^{(1)}_{N-p}&\Cdots &B^{(q)}_{N-p}\\
		B^{(1)}_{N-p+1}&\Cdots &B^{(q)}_{N-p+1}\\
		\Vdots& &\Vdots\\\\
		B^{(1)}_{N-1}&\Cdots &B^{(q)}_{N-1}\\
	\end{bNiceArray}, &
	A^{[N,q]}&=\begin{bNiceArray}{ccccc}[margin,cell-space-limits=2pt]
	A^{(1)}_{N-q}&	A^{(1)}_{N-q+1}&\Cdots &&A^{(1)}_{N-1}\\
		\Vdots	&\Vdots & &&\Vdots\\
	A^{(p)}_{N-q}&	A^{(p)}_{N-q+1}&\Cdots &&A^{(p)}_{N-1}\\
	\end{bNiceArray}.
\end{aligned}
\]
 \begin{pro}[Block structure of triangular inverses]
The inverse matrices of the triangular factors of the Gauss--Borel factorization have the following block structure:
 \[
 \begin{aligned}
 	\mathscr L_{N}^{-1}&=\begin{bNiceArray}{w{c}{1.5cm}|w{c}{5.7cm}}[margin,cell-space-limits=2pt]
 		\mathscr L_{N-q} ^{-1}& 0_{(N-q)\times q}\\\hline
 		L^{[N,N-q]} & \bigintsss x^{N_q-1} \mathfrak{X}_{[q,s_q]}(x) \d\mu(x)A^{[N,q]}(x)
 	\end{bNiceArray}
 	,\\
 	\mathscr U_N ^{-1}&=\begin{bNiceArray}{w{c}{1.5cm}|w{c}{6.5cm}}[margin,cell-space-limits=2pt]
 	\mathscr U_{N-p} ^{-1}& U^{[N-p,N]} \\\hline
 	0_{p\times(N-q)}& \bigintsss B^{[N,p]}(x)\d\mu(x)x^{N_p-1} \left(\mathfrak{X}_{[p,s_p]}(x) \right)^\top
 		\end{bNiceArray}
 \end{aligned}.\]
 \end{pro}
\begin{proof}
	It follows from
	 \[
	\begin{aligned}
		\mathscr L_{N}^{-1}&=\mathscr M_N \mathscr U_N =\int X^{[N]}(x)\d\mu(x)\left(X^{[N]}(x)\right)^\top\mathscr U_N =
		\bigintss \begin{bNiceArray}{w{c}{3cm}}[margin,cell-space-limits=2pt]
			\Block{3-1}{ X^{[N-q]}_{[q]}(x)}\\ \\\\\hline x^{N_q-1} \mathfrak{X}_{[q,s_q]}(x) 
		\end{bNiceArray}\d\mu(x)A^{[N]}(x)\\&=\begin{bNiceArray}{w{c}{1.5cm}|w{c}{5.5cm}}[margin,cell-space-limits=2pt]
			\mathscr L_{N-q} ^{-1}& 0_{(N-q)\times q}\\\hline
			L^{[N,N-q]} & \bigintssss x^{N_q-1} \mathfrak{X}_{[q,s_q]}(x) \d\mu(x)A^{[N,q]}(x)
		\end{bNiceArray}
		,\\
		\mathscr U_N ^{-1}&=\mathscr L_N\mathscr M_N=\mathscr L_N \int X^{[N]}(x)\d\mu(x)\left(X^{[N]}(x)\right)^\top\\&=
		\bigintssss B^{[N]}(x)\d\mu(x)\begin{bNiceArray}{c|c}[margin,cell-space-limits=2pt]
			\left(X^{[N-p]}_{[p]}(x)\right)^\top&x^{N_p-1} \left(\mathfrak{X}_{[p,s_p]}(x) \right)^\top
		\end{bNiceArray}\\&=\begin{bNiceArray}{w{c}{1.5cm}|w{c}{6.5cm}}[margin,cell-space-limits=2pt]
			\mathscr U_{N-p} ^{-1}& U^{[N-p,N]} \\\hline
			0_{p\times(N-q)}& \bigintssss B^{[N,p]}(x)\d\mu(x)x^{N_p-1} \left(\mathfrak{X}_{[p,s_p]}(x) \right)^\top
		\end{bNiceArray}
	\end{aligned}.\]
\end{proof}

\section{The banded recursion matrices}
The aim of this section is to find matrices that models the recursion relations that the polynomial matrices $A^{[N]}$ and $B^{[N]}$ fulfill. We will do it at the end of the section. 
	For that aim we need some preparation. 
	
	In this paper we assume that $\mathscr M_N$ has a Gauss-Borel factorization, but $\det  \mathscr M_{N+1}=0$. Hence, $\mathscr M_{N+1}$ do not admit a Gauss-Borel factorization, and we only have a finite number of polynomials collected in $A^{[N]}$ and $B^{[N]}$ which satisfy multiple orthogonal relations of mixed type in the step-line. This is the reason the treatment we have given to the recursion relations working only with truncations for which the orthogonal polynomials do exist.
	
	First note that
\begin{pro}
	The following relations hold
	\begin{subequations}\label{eq:shift_banda}
		\begin{align}\label{eq:shift_banda_0}
			\mathscr{L}_{N}	\Lambda^{[N,N+q]}_{[q]}\left[\begin{NiceArray}{c}[margin,cell-space-limits=2pt]
				\mathscr{L}_{N}^{-1} \\\hline L^{[N+q,N]}
			\end{NiceArray}\right]&=\begin{bNiceArray}{c|c}
				\mathscr{U}_{N}^{-1} &U^{[N,N+p]}
			\end{bNiceArray}\left(\Lambda^{[N,N+p]}_{[p]}\right)^\top 	\mathscr U_{N},\\
			\label{eq:shift_banda_1}
			\mathscr{L}_{N-q}	\Lambda^{[N-q,N]}_{[q]}\mathscr L_{N}^{-1}&=\begin{bNiceArray}{c|c}
				\mathscr{U}_{N-q}^{-1} &U^{[N-q,N+p]}
			\end{bNiceArray}\left(\Lambda^{[N,N+p]}_{[p]}\right)^\top 	\mathscr U_{N},\\\label{eq:shift_banda_2}
			\mathscr L_{N}	\Lambda^{[N,N+q]}_{[q]}\left[\begin{NiceArray}{c}[margin,cell-space-limits=2pt]
				\mathscr{L}_{N-p}^{-1} \\\hline L^{[N+q,N-p]}
			\end{NiceArray}\right]&=
			\mathscr U_{N}^{-1}	\left(\Lambda^{[N-p,N]}_{[p]}\right)^\top \mathscr{U}_{N-p}.
		\end{align}
	\end{subequations}
\end{pro}
\begin{proof}
	An immediate consequence of 	the Hankel symmetry given in  Equations \eqref{eq:factorization} and the Gauss--Borel factorization.
\end{proof}
We are now ready to introduce three matrices that will represent the mentioned recursion relations:

\begin{defi}[Recursion matrices]
	We define the square  matrix $\mathscr T_N$ and the rectangular matrices $\mathscr T^{[N-q,N]}$ and $\mathscr T^{[N,N-p]}$ as follows:
\begin{subequations}\label{eq:Ts}
	\begin{align}\label{eq:T}
		\mathscr T_{N}&\coloneq		\mathscr{L}_{N}	\Lambda^{[N,N+q]}_{[q]}\left[\begin{NiceArray}{c}[margin,cell-space-limits=2pt]
			\mathscr{L}_{N}^{-1} \\\hline L^{[N+q,N]}
		\end{NiceArray}\right]=\begin{bNiceArray}{c|c}
			\mathscr{U}_{N}^{-1} &U^{[N,N+p]}
		\end{bNiceArray}\left(\Lambda^{[N,N+p]}_{[p]}\right)^\top 	\mathscr U_{N}\in\R^{N\times N},\\
		\label{eq:T1}
		\mathscr T^{[N-q,N]}&\coloneq\mathscr{L}_{N-q}	\Lambda^{[N-q,N]}_{[q]}\mathscr L_{N}^{-1}=\begin{bNiceArray}{c|c}
			\mathscr{U}_{N-q}^{-1} &U^{[N-q,N+p]}
		\end{bNiceArray}\left(\Lambda^{[N,N+p]}_{[p]}\right)^\top 	\mathscr U_{N}\in\R^{(N-q)\times N},\\\label{eq:T2}
		\mathscr T^{[N,N-p]}&\coloneq
		\mathscr U_{N}^{-1}	\left(\Lambda^{[N-p,N]}_{[p]}\right)^\top \mathscr{U}_{N-p} =\mathscr L_{N}	\Lambda^{[N,N+q]}_{[q]}\left[\begin{NiceArray}{c}[margin,cell-space-limits=2pt]
			\mathscr{L}_{N-p}^{-1} \\\hline L^{[N+q,N-p]}
		\end{NiceArray}\right]\in\R^{N\times (N-p)},
	\end{align}
\end{subequations}
\end{defi}
\begin{pro}[Banded structure of the recursion matrices]
Recursion matrices $\mathscr T_N$, $\mathscr T^{[N-q,N]}$ and $\mathscr T^{[N,N-p]}$ are banded matrices with $q$ superdiagonals and $p$ subdiagonals.
\end{pro}
\begin{proof}
It directly follows from	Equations \eqref{eq:shift_banda}  or \eqref{eq:Ts}.
\end{proof}

We now discuss the relation between these three matrices. First, we need to introduce the following:
\begin{defi}
Let us consider the following lower and upper triangular matrices, respectively:
	\begin{align*}
		T^{[N-q,N]}&\coloneq\mathscr L_{[N-2q,N-q]}\mathscr L_{[N-q,N]}^{-1}=\int  B^{[N-q,q]}(x)\d\mu(x)A^{[N,q]}(x)\in\R^{q\times q},\\ 
		T^{[N,N-p]}&\coloneq\mathscr U_{[N,N-p]}^{-1}\mathscr U_{[N-p,N-2p]}=\int  B^{[N,p]}(x)\d\mu(x)A^{[N-p,p]}(x)\in\R^{p\times p},
	\end{align*}
\end{defi}

	\begin{pro}[Relation between recursion matrices]\label{pro:structure_T}
	The following block structure of the recursion matrices holds:
	\begin{align*}
		\mathscr	 T^{[N-q,N]}&=	\begin{bNiceArray}{cw{c}{3cm}cw{c}{1.8cm}c}[margin,cell-space-bottom-limit=10pt]
			\Block[borders=right]{5-3}<\huge>{\mathscr T_{N-q}}&&&\Block[]{3-2}<\Large>{0_{(N-2q)\times q}}&
			\\
			\\\\
			&&	&\Block[borders=top]{2-2}<\large>{T^{[N-q,N]}}&\\
			\\	\end{bNiceArray}, \\
		\mathscr	 T^{[N,N-p]}&=	\begin{bNiceArray}{cw{c}{2cm}w{c}{.75cm}w{c}{.5cm}c}[margin,cell-space-bottom-limit=10pt]
			\Block[borders=bottom]{5-5}<\huge>{\mathscr T_{N-p}}&&&&
			\\
			\\
			\\\\\\
			\Block[]{2-3}<\Large>{\hspace*{-1cm}0_{p\times (N-2p)}} &&\Block[c,borders=left]{2-2}<\large>{\hspace*{.7cm}T^{[N,N-p]}}&	\\
			\\	\end{bNiceArray}, 
	\end{align*}
\end{pro}

\begin{proof}
	From Equation \eqref{eq:T} we get
\[
\begin{aligned}
	\mathscr T_{N-q}&=		\mathscr{L}_{N-q}	\Lambda^{[N-q,N]}_{[q]}\left[\begin{NiceArray}{c}[margin,cell-space-limits=2pt]
		\mathscr{L}_{N-q}^{-1} \\\hline L^{[N,N-q]}
	\end{NiceArray}\right],&
	\mathscr T_{N-p}&=		\begin{bNiceArray}{c|c}
		\mathscr{U}_{N-p}^{-1} &U^{[N-p,N]}
	\end{bNiceArray}\left(\Lambda^{[N-p,N]}_{[p]}\right)^\top 	\mathscr U_{N-p}.
\end{aligned}
\]
Using \eqref{eq:inverses_direct}, \eqref{eq:T1} and \eqref{eq:T2} we obtain
	\begin{align}
	\mathscr T^{[N-q,N]}&\coloneq\mathscr{L}_{N-q}	\Lambda^{[N-q,N]}_{[q]} \begin{bNiceArray}{c|c}[margin,cell-space-limits=2pt]
		\mathscr L_{N-q} ^{-1}& 0_{(N-q)\times q}\\\hline
		L^{[N,N-q]} &\mathscr L_{[N,N-q]}^{-1}
	\end{bNiceArray}=\begin{bNiceArray}{c|c}
	\mathscr T_{N-q} & 	\mathscr T^{[N-q,N,q]}
	\end{bNiceArray},\\
	\mathscr T^{[N,N-p]}&\coloneq
\begin{bNiceArray}{c|c}[margin,cell-space-limits=2pt]
\mathscr U_{N-p} ^{-1}& U^{[N-p,N]} \\\hline
0_{p\times(N-q)}& \mathscr U_{[N-p,N]}^{-1}
\end{bNiceArray}\left(\Lambda^{[N-p,N]}_{[p]}\right)^\top \mathscr{U}_{N-p} =\begin{bNiceArray}{c}[margin,cell-space-limits=2pt]
\mathscr T_{N-q}
\\\hline
\mathscr T^{[N,N-p,p]}
\end{bNiceArray},
\end{align}
with $\mathscr T^{[N-q,N,q]}\in\R^{(N-q)\times q}, \mathscr T^{[N,N-p,p]}\in\R^{p\times (N-p)}$ given by
\begin{align*}
	\mathscr T^{[N-q,N,q]}&= \mathscr{L}_{N-q}	\Lambda^{[N-q,N]}_{[q]}\begin{bNiceArray}{c}[margin,cell-space-limits=2pt]
		 0_{(N-q)\times q}\\\hline
	\bigintssss x^{N_q-1} \mathfrak{X}_{[q,s_q]}(x) \d\mu(x)A^{[N,q]}(x)
		\end{bNiceArray}\\&=\begin{bNiceArray}{c|c}[margin,cell-space-limits=2pt]
		\mathscr{L}_{N-2q} &0_{(N-2q)\times q}\\\hline
		\mathscr l^{[N-q,N-2q]} & \mathscr L_{[N-2q,N-q]}
		\end{bNiceArray}	\begin{bNiceArray}{c}[margin,cell-space-limits=2pt]
		0_{(N-2q)\times q}\\\hline
		\bigintssss x^{N_q-1} \mathfrak{X}_{[q,s_q]}(x) \d\mu(x)A^{[N,q]}(x)
		\end{bNiceArray}\\&=\begin{bNiceArray}{c}[margin,cell-space-limits=2pt]
		0_{(N-2q)\times q}\\\hline
		\bigintssss B^{[N-q,q]}(x)\d\mu(x)A^{[N,q]}(x)
		\end{bNiceArray},
		\\
			\mathscr T^{[N,N-p,p]}&=
			\begin{bNiceArray}{c|c}[margin,cell-space-limits=2pt]
				0_{p\times(N-p)}&\bigintsss  B^{[N,p]}(x)\d\mu(x)x^{N_p-1}\left(\mathfrak{X}_{[p,s_p]}(x)\right)^\top
			\end{bNiceArray} \left(\Lambda^{[N-p,N]}_{[p]}\right)^\top \mathscr{U}_{N-p}\\&=
			\begin{bNiceArray}{c|c}[margin,cell-space-limits=2pt]
				0_{p\times(N-2p)}&\bigintsss B^{[N,p]}(x)\d\mu(x) x^{N_p-1}\left(\mathfrak{X}_{[p,s_p]}(x)\right)^\top
			\end{bNiceArray} 	\begin{bNiceArray}{c|c}[margin,cell-space-limits=2pt]
			\mathscr U_{N-2p} & \mathscr u^{[N-2p,N-p]} \\\hline
			0_{p\times (N-2p)}	&\mathscr U_{[N-2p,N-p]}
			\end{bNiceArray}\\&=
				\begin{bNiceArray}{c|c}[margin,cell-space-limits=2pt]
				0_{p\times(N-2p)}&\bigintssss  B^{[N,p]}(x)\d\mu(x)A^{[N-p,p]}(x)
			\end{bNiceArray} .
\end{align*}
\end{proof}

\begin{pro}[Recursion relations]
	The rectangular banded matrices $ \mathscr T^{[N-q,N]}$ and $  \mathscr T^{[N,N-p]}$ model recursion relations for the  $B$ and $A$ polynomials, respectively, 
\[ \begin{aligned}
 \mathscr T^{[N-q,N]}		B^{[N]}(x)&= x	B^{[N-q]}(x),&
A^{[N]}(x)  \mathscr T^{[N,N-p]}	&= x	A^{[N-p]}(x).
 \end{aligned}\]
\end{pro}
 The square banded matrix $\mathscr T_N$ needs a correction to describe such recursion relations. For that aim we require of:
\begin{defi}
	Let us consider
	   $A^{[N]}=\begin{bNiceArray}{c|c}[small]
 	A^{[N-p]} & A^{[N,p]}
 \end{bNiceArray}$ and $B^{[N]}=\begin{bNiceMatrix}[small,cell-space-limits=1pt]
 	B^{[N-q]} \\ \hline B^{[N,q]}
 \end{bNiceMatrix}$, $A^{[N,p]}\in\R^{p\times p}[x],B^{[N,p]}\in\R^{q\times q}[x]$, with
  \[
 \begin{aligned}
 	B^{[N,q]}&=\begin{bNiceArray}{ccc}[margin,cell-space-limits=1pt]
 		B^{(1)}_{N-q}&\Cdots &B^{(q)}_{N-q}\\
 		B^{(1)}_{N-q+1}&\Cdots &B^{(q)}_{N-q+1}\\
 		\Vdots& &\Vdots\\\\
 		B^{(1)}_{N-1}&\Cdots &B^{(q)}_{N-1}\\
 	\end{bNiceArray}, &
 	A^{[N,p]}&=\begin{bNiceArray}{ccccc}[margin,cell-space-limits=2pt]
 		A^{(1)}_{N-p}&	A^{(1)}_{N-p+1}&\Cdots &&A^{(1)}_{N-1}\\
 		\Vdots	&\Vdots & &&\Vdots\\
 		A^{(p)}_{N-p}&	A^{(p)}_{N-p+1}&\Cdots &&A^{(p)}_{N-1}\\
 	\end{bNiceArray}.
 \end{aligned}
 \]
\end{defi}
 
\begin{coro}[Recursion relations]The recursion relations can be alternatively written as 
	 \[
 \begin{aligned}
\mathscr T_{N-q}B^{[N-q]}(x)+\begin{bNiceArray}{c}[margin,cell-space-limits=2pt]
	0_{(N-q)\times q }\\
	\hline
	T^{[N-q,N]}B^{[N,q]}(x)
\end{bNiceArray}&=xB^{[N-q]}(x), \\
A^{[N-p]}(x)\mathscr T_{N-p}+\begin{bNiceArray}{c|c}  0_{(N-p)\times p} &
	A^{[N,p]}(x)T^{[N,N-p]}
\end{bNiceArray}&=xA^{[N-p]}(x). 
 \end{aligned}
 \]
\end{coro}

 \section{Christoffel transformations and bidiagonal factorizations}
 
 In this section we discuss bidiagonal factorizations of the recursion matrices $\mathscr T_N$. 
\begin{defi}[Bidiagonal factorization]
	 Given bidiagonal matrices of the form
 \[
 \begin{aligned}
 		U_b&=\begin{bNiceArray}{w{c}{20pt}w{c}{20pt}cccw{c}{30pt}}[cell-space-limits=0pt,margin]
 			U_{b,0} & 1&0&\Cdots&& 0\\
 			0&U_{b,1}  & 1&0&& \\
 			\Vdots&\Ddots &\Ddots[shorten-end=-5pt]&\Ddots&\Ddots&\Vdots\\[0pt]&&&&&0\\
 			&&&&&1\\
 			0&\Cdots&&&0&U_{b,N-1} 
 		\end{bNiceArray}, & b&\in\{1,\dots,q\},\\
 		L_a&=\begin{bNiceArray}{cccccc}[,margin]
 			1&0&\Cdots&&&0\\[-5pt]
 			L_{a,1} & 1&0&&& \Vdots\\
 			0&L_{a,2}& 1&0&& \\
 			\Vdots&\Ddots &\Ddots[shorten-end=-15pt]&\Ddots&\Ddots&\\
 			&&&&&0\\
 			0&\Cdots&&0&L_{a,N-1}&1
 		\end{bNiceArray}, & a&\in\{1,\dots,p\}.
 	\end{aligned}
 	\]
 we say that a monic banded matrix $\mathscr T_N$, with $q$ superdiagonal and $p$ subdiagonals and with the highest superdiagonal normalized with $1$'s,  admits a bidiagonal factorization if it can written as 
  \[
 \mathscr T_N=L_1\cdots L_pU_q\cdots U_1.
 \]
\end{defi}
 In this section we discuss this possibility and find conditions for the existence of  such factorization, and also expressions for its factors. We will do it considering simple Christoffel transformations  of the matrix of measures $\d\mu$.
 
 For a more clear presentation, within this section we  consider that   $\mathscr L_N$ is lower unitriangular. 
 To start with, we need truncations of the basic shift matrix $\Lambda$:
\begin{defi}[Truncated basic shift matrices]
	For $i\in\{0,1,\dots,r-1\}$, let us define
 \[
 \Lambda^{[N+i,N+i+1]}\coloneq \begin{bNiceArray}{ccccc|c}[first-row, last-col,margin]
 	\Hbrace{6}{N+i+1}\\
 	0& 1&0&\Cdots &0&0&\Vbrace{5}{N+i}\\
 	\Vdots &\Ddots&\Ddots&\Ddots &\Vdots&\Vdots\\
 	&&& &0&0\\
 	&&&&1&0\\
 	0&\Cdots&&&0&1\\
 \end{bNiceArray}\hspace*{.7cm}.
 \]
\end{defi}
Similar to the block case we have a spectral property on monomials:
\begin{pro}
	 	The action of these truncated basic shift matrices on the monomial vectors is
 	\[
 	\Lambda^{[N+r-1,N+r]}X_{[q]}^{[N+r]}=X_{[r]}^{[N+r-1]}\mathfrak X_{[r,1]}(x).
 	\]
\end{pro}
 
 Importantly, we have that:
\begin{pro}\label{pro:TNr}
	  We can factor out the matrix $\Lambda^{[N,N+r]}_{[r]}$  as follows
 \[
 \Lambda^{[N,N+r]}_{[r]}=  \Lambda^{[N,N+1]}\Lambda^{[N+1,N+2]}\cdots \Lambda^{[N+r-1,N+r]}.
 \]
\end{pro}
 We now give a rewriting of the equation defining the recursion matrices:
\begin{pro}
	The recursion matrices $\mathscr T_N$ satisfy
 \begin{align}\label{eq:T_bis_1}
 		\mathscr T_{N}&=		\mathscr{L}_{N}	\Lambda^{[N,N+1]}\Lambda^{[N+1,N+2]}\cdots \Lambda^{[N+q-1,N+q]}\mathscr M^{[N+q,N]}\mathscr U_N\\\label{eq:T_bis_2}
 		&=
 		\mathscr L_N \mathscr M^{[N, N+p]}
\left( \Lambda^{[N+p-1,N+p]} \right)^\top\cdots\left(\Lambda^{[N+1,N+2]}\right)^\top	\left(\Lambda^{[N,N+1]}\right)^\top\mathscr U_{N}.
 \end{align}
\end{pro}
\begin{proof}
Rewrite \eqref{eq:T} as follows
\begin{equation}\label{eq:T_bis}
	\begin{aligned}
		\mathscr T_{N}&=		\mathscr{L}_{N}	\Lambda^{[N,N+q]}_{[q]}\mathscr M^{[N+q,N]}\mathscr U_N
		=
		\mathscr L_N \mathscr M^{[N, N+p]}
		\left(\Lambda^{[N,N+p]}_{[p]}\right)^\top 	\mathscr U_{N},
	\end{aligned}
\end{equation}
and Proposition \ref{pro:TNr}.
\end{proof}
 Now, we introduce some basic Christoffel type perturbations of the matrix of measures. Namely perturbations by left or right multiplication by the  first degree matrix polynomials $\mathfrak X_{[r,1]}$:
\begin{defi}[Christoffel perturbed matrix of measures]
	 Let us define
 \[
 \d\mu_{(n,m)}\coloneq\mathfrak X_{[q,1]}^n\d\mu \left(\mathfrak X_{[p,1]}^m\right)^\top,
 \]
 and the corresponding moment matrices
 \[
 \mathscr M^{[N,M]}_{(n,m)}=\int X^{[N]}(x)\d\mu_{(n,m)}(x)\left(X^{[M]}(x)\right)^\top.
 \]
\end{defi}

\begin{rem}
	Observe that 
\[
x\d\mu=\d\mu^{(p,0)}=\d\mu^{(0,q)},
\]
and the Christoffel transformation coincides.
\end{rem}

\begin{teo}\label{th:Bidiagonal_Factorization}
	The recursion matrix $\mathscr T_N$ admits a bidiagonal factorization 
	\[
	 \mathscr T_N=L_1\cdots L_pU_q\cdots U_1,
	\]
	if all the moment matrices $\mathscr M_N$, $\mathscr M_{N,(b,0)}$, $b\in\{1,\dots,q\}$, and $\mathscr M_{N,(0,a)}$, $a\in\{1,\dots,p\}$, admit a Gauss--Borel factorization. The bidiagonal factors are given by 
\begin{equation}\label{eq:UL}
	\begin{aligned}
		U_b&\coloneq  \mathscr U_{N,(b,0)}^{-1}\mathscr U_{N,(b-1,0)}=\begin{bNiceArray}{w{c}{20pt}w{c}{20pt}cccw{c}{30pt}}[cell-space-limits=0pt,margin]
			U_{b,0} & 1&0&\Cdots&& 0\\
			0&U_{b,1}  & 1&0&& \\
			\Vdots&\Ddots &\Ddots[shorten-end=-8pt]&\Ddots&\Ddots&\Vdots\\[0pt]&&&&&0\\
			&&&&&1\\
			0&\Cdots&&&0&U_{b,N-1} 
		\end{bNiceArray}, & b&\in\{1,\dots,q\},\\
		L_a&\coloneq  \mathscr L_{N,(0,a-1)}\mathscr L_{N,(0,a)}^{-1}
		=\begin{bNiceArray}{cccccc}[,margin]
			1&0&\Cdots&&&0\\[-5pt]
			L_{a,1} & 1&0&&& \Vdots\\
			0&L_{a,2}& 1&0&& \\
			\Vdots&\Ddots &\Ddots[shorten-end=-15pt]&\Ddots&\Ddots&\\
			&&&&&0\\
			0&\Cdots&&0&L_{a,N-1}&1
		\end{bNiceArray}, & a&\in\{1,\dots,p\}.
	\end{aligned}
\end{equation}
\end{teo}
 \begin{proof}
 	We start the proof by finding the bidiagonal factors $U_n$.
  Firstly,  we note that
 \[
 \Lambda^{[N+q-1,N+q]}\mathscr M^{[N+q,N]}=\int X^{[N+q-1]}(x)\mathfrak X_{[q]}(x)\d\mu(x)\left(X^{[N]}_{[p]}(x)\right)^\top=\mathscr M_{(1,0)}^{[N+q-1,N]}
 \]
 and 
 \begin{align}\label{eq:bidiagonal_1}
 	\Lambda^{[N+q-1,N+q]}\mathscr M^{[N+q,N]}\mathscr U_N=\Lambda^{[N+q-1,N+q]}\begin{bNiceArray}{c}[margin,cell-space-limits=2pt]
 	\mathscr L_N^{-1}\\\hline L^{[N+q,N]}
 \end{bNiceArray} =\begin{bNiceArray}{c}[margin,cell-space-limits=2pt]
 \mathscr L_{N,(1,0)}^{-1}\\\hline L^{[N+q-1,N]}_{(1,0)}
 \end{bNiceArray}\mathscr U_{N,(1,0)}^{-1}\mathscr U_N.
 \end{align}
 Now, the matrix $\Lambda^{[N+q-1,N+q]}\begin{bNiceArray}{c}[margin,small]
 	\mathscr L_N^{-1}\\\hline L^{[N+q,N]}
 \end{bNiceArray} $ is a upper Hessenberg matrix, that assuming that $\mathscr L_N$ is monic, is also monic. Therefore,  
 \[
 U_1\coloneq \mathscr U_{N,(1,0)}^{-1}\mathscr U_N
 \]
 is a monic upper bidiagonal matrix (with its first superdiagonal being monic).

 Now we consider
 \[
\begin{aligned}
	 \Lambda^{[N+q-2,N+q-1]} \Lambda^{[N+q-1,N+q]}\mathscr M^{[N+q,N]}\mathscr U_N&=  \Lambda^{[N+q-2,N+q-1]}\mathscr M_{(1,0)}^{[N+q-1,N]}\mathscr U_N\\&=
	 \Lambda^{[N+q-2,N+q-1]}\mathscr M_{(1,0)}^{[N+q-1,N]}\mathscr U_{N,(1,0)}U_1
\end{aligned}
 \]
 and using \eqref{eq:bidiagonal_1} for $\d\mu_{(1,0)}$ and $N+q\to N+q-1$ we find
  \begin{align}\label{eq:bidiagonal_2}
 	\Lambda^{[N+q-2,N+q-1]}\mathscr M^{[N+q-1,N]}_{(1,0)}\mathscr U_{N,(1,0)}=\Lambda^{[N+q-2,N+q-1]}\begin{bNiceArray}{c}[margin,cell-space-limits=2pt]
 		\mathscr L_{N,(1,0)}^{-1}\\\hline L^{[N+q-1,N]}_{(1,0)}
 	\end{bNiceArray} =\begin{bNiceArray}{c}[margin,cell-space-limits=2pt]
 		\mathscr L_{N,(2,0)}^{-1}\\\hline L^{[N+q-2,N]}_{(2,0)}
 	\end{bNiceArray}\mathscr U_{N,(2,0)}^{-1}\mathscr U_{N,(1,0)}.
 \end{align}
  Hence, 
   \[
 U_2\coloneq \mathscr U_{N,(2,0)}^{-1}\mathscr U_{N,(1,0)}
 \]
 is a monic upper bidiagonal matrix and
 \begin{align}\label{eq:U2}
 	\Lambda^{[N+q-2,N+q-1]} \Lambda^{[N+q-1,N+q]}\mathscr M^{[N+q,N]}\mathscr U_N&=  \begin{bNiceArray}{c}[margin,cell-space-limits=2pt]
 		\mathscr L_{N,(2,0)}^{-1}\\\hline L^{[N+q-2,N]}_{(2,0)}
 	\end{bNiceArray}U_2U_1.
 \end{align}

  Repeating this process $q$ times we find 
  \[
  \Lambda^{[N,N+1]}\Lambda^{[N+1,N+2]}\cdots \Lambda^{[N+q-1,N+q]}\mathscr M^{[N+q,N]}\mathscr U_N=\mathscr L_{N,(q,0)}^{-1} U_q\cdots U_1\]
   where 
  \[
  \begin{aligned}
  	U_b&\coloneq  \mathscr U_{N,(b,0)}^{-1}\mathscr U_{N,(b-1,0)}=\begin{bNiceArray}{w{c}{20pt}w{c}{20pt}cccw{c}{30pt}}[cell-space-limits=0pt,margin]
  	U_{b,0} & 1&0&\Cdots&& 0\\
  		0&U_{b,1}  & 1&0&& \\
  		\Vdots&\Ddots &\Ddots[shorten-end=-7pt]&\Ddots&\Ddots&\Vdots\\[0pt]&&&&&0\\
  			&&&&&1\\
  		0&\Cdots&&&0&U_{b,N-1} 
  	\end{bNiceArray}, & b&\in\{1,\dots,q\}.
  \end{aligned}
  \]
so that
 \[
 \mathscr{T}_N=\mathscr L_N\mathscr L_{N,(q,0)}^{-1} U_q\cdots U_1.
 \]

We now find the lower bidiagonal factors $L_m$.
 Following an similar procedure  as above we get
	\[ \mathscr M^{[N, N+p]}
 \left( \Lambda^{[N+p-1,N+p]} \right)^\top=\int X^{[N]}(x) \d\mu(x)\left(\mathfrak X_{[q,1]}(x)\right)^\top\left(X^{[N+p-1]}_{[p]}(x)\right)^\top=\mathscr M_{(0,1)}^{[N,N+p-1]} 
\]
 and also that
\begin{equation}\label{eq:ML1}
\begin{aligned}
	 \mathscr L_N \mathscr M^{[N, N+p]}
 \left( \Lambda^{[N+p-1,N+p]} \right)^\top&= \begin{bNiceArray}{c|c}[margin,cell-space-limits=2pt]
 	\mathscr U_{N}^{-1} &  U^{[N,N+p]}
 \end{bNiceArray} \left( \Lambda^{[N+p-1,N+p]} \right)^\top\\&= \mathscr L_N\mathscr L_{N,(0,1)}^{-1}\begin{bNiceArray}{c|c}[margin,cell-space-limits=2pt]
 \mathscr U_{N,(1,0)}^{-1} &  U^{[N,N+p-1]}_{(0,1)}
 \end{bNiceArray} .
\end{aligned}
\end{equation}
Thus,
 \[
 L_1\coloneq \mathscr L_N\mathscr L_{N,(0,1)}^{-1}
 \]
 is a monic lower bidiagonal matrix (with is main diagonal being monic). A iterative procedure leads to 
 \[
 \mathscr T_N= L_1\cdots L_p \mathscr  U_{N,(0,p)}^{-1}\mathscr U_N,
 \]
 with
  \[
\begin{aligned}
	 L_a&\coloneq  \mathscr L_{N,(0,a-1)}\mathscr L_{N,(0,a)}^{-1}
	 =\begin{bNiceArray}{cccccc}[,margin]
	 	1&0&\Cdots&&&0\\[-5pt]
	 	L_{a,1} & 1&0&&& \Vdots\\
	 	0&L_{a,2}& 1&0&& \\
	 	\Vdots&\Ddots &\Ddots[shorten-end=-15pt]&\Ddots&\Ddots&\\
	 	&&&&&0\\
	 	0&\Cdots&&0&L_{a,N-1}&1
	 \end{bNiceArray}, & a&\in\{1,\dots,p\}.
\end{aligned}
 \]
 Given the normalization of these bidiagonal matrices on its higher diagonal, we get that the bidiagonal factorization
 \[
 \mathscr T_N=L_1\cdots L_pU_q\cdots U_1
 \]
 is satisfied.
 \end{proof}
\begin{rem}
	 We observe that given the triangular form  these bidiagonal matrices for the $M$-th, $M\leq N$,  truncations we have
  \[
 \mathscr T_M=L_1^{[M]}\cdots L_p^{[M]}U_q^{[M]}\cdots U_1^{[M]},
 \]
 where $U_n^{[M]}, L_m^{[M]}$ are the  $M$-th truncations of the corresponding bidiagonal matrices. 
\end{rem}

\begin{rem}
	 Note that for $a\in\{1,\dots,p\}$ and $b\in\{1,\dots,q\}$ we have a the connection formulas
 \[
 \begin{aligned}
 	B^{[N]}_{(a-1,0)}&=L_a B^{[N]}_{(a,0)}, & 	A^{[N]}_{(0,b-1)}&= A^{[N]}_{(0,b)}U_b.
 \end{aligned}
 \]
\end{rem}
 
 \begin{teo}\label{teo:bidiagonal_moment}
 	The entries of the bidiagonal matrices can be expressed as follows:
\begin{equation}
	\label{eq:UL2}
\begin{aligned}
	L_{a,n}&=\left(\mathscr L_{N,(0,a-1)}\right)_{n,n-1}-\left(\mathscr L_{N,(0,a)}\right)_{n,n-1}, & a&\in\{1,\dots, p\}, & n&\in\{1,\dots,N-1\},\\
U_{b,n}&=\frac{\left(\mathscr U_{N,(b-1,0)}\right)_{n,n}}{\left(\mathscr U_{N,(b,0)}\right)_{n,n}}, & b&\in\{1,\dots, q\}, & n&\in\{0,1,\dots,N-1\}.
\end{aligned}
\end{equation}
Alternatively, we also find for $b\in\{1,\dots,q\}$ and $a\in\{1,\dots,p\}$ 
\begin{equation}
	\label{eq:UL3}
\begin{aligned}
			L_{a,n}&=\frac{\left(\mathscr U_{N,(0,a))}\right)_{n-1,n-1}}{\left(\mathscr U_{N,(0,a-1))}\right)_{n,n}}, &  n\in\{1,\dots,N-1\},\\
		U_{b,n}&=(1-\delta_{n,0})\left(\mathscr L_{N,(b,0)}\right)_{n,n-1}-\left(\mathscr L_{N,(b-1,0)}\right)_{n+1,n}, &  n\in\{0,1,\dots,N-1\}.
\end{aligned}
\end{equation}
 \end{teo}
 
 \begin{proof}
 Equations \eqref{eq:UL2}	follows from Equations \eqref{eq:UL} for the bidiagonal matrices. Equations \eqref{eq:UL3} follow from
 	  \begin{align*}
 	\Lambda^{[N+q-b,N+q-b+1]}\begin{bNiceArray}{c}[margin,cell-space-limits=2pt]
 			\mathscr L_{N,(b-1,0)}^{-1}\\\hline L^{[N+q-b+1,N]}_{(b-1,0)}
 		\end{bNiceArray} &=\begin{bNiceArray}{c}[margin,cell-space-limits=2pt]
 			\mathscr L_{N,(b,0)}^{-1}\\\hline L^{[N+q-b,N]}_{(b,0)}
 		\end{bNiceArray}U_b,& b&\in\{1,\dots,q\},\\
 		 \begin{bNiceArray}{c|c}[margin,cell-space-limits=2pt]
 				\mathscr U_{N,(0,a-1)}^{-1} &  U^{[N,N+p-a+1]}
 			\end{bNiceArray} \left( \Lambda^{[N+p-a,N+p-a+1]} \right)^\top&=L_a\begin{bNiceArray}{c|c}[margin,cell-space-limits=2pt]
 				\mathscr U_{N,(a,0)}^{-1} &  U^{[N,N+p-a]}_{(0,a)}
 			\end{bNiceArray},& a&\in\{1,\dots,p\} .
 		\end{align*}
 \end{proof}
 \begin{coro}
 	In terms of the sub-leading coefficients $\beta^{(s+1)}_{n,(0,b)}$ of the polynomials $B^{(s+1)}_{n,(0,b)}(x)$  and the leading coefficients $\alpha^{(s+1)}_{n,(0,b)}$ of the polynomials $A^{(s+1)}_{n,(0,a)}(x)$, $b\in\{1,\dots, q\}$, $a\in\{1,\dots, p\}$, $n\in\{1,\dots,N-1\}$,
we have
\[
\begin{aligned}
	L_{a,n}&=\beta^{(s_{n}+1)}_{n+1,(0,a-1)}-\beta^{(s_n+1)}_{n,(0,a)}=\frac{\alpha^{(s_{n-1}+1)}_{n-1,(0,a)}}{\alpha^{(s_n+1)}_{n,(0,a-1)}}, \\
	U_{b,n}&=\frac{\alpha^{(s_n+1)}_{n,(b-1,0)}}{\alpha^{(s_n+1)}_{n,(b,0)}}=(1-\delta_{n,0})\beta^{(s_n+1)}_{n,(b,0)}-\beta^{(s_{n+1}+1)}_{n+1,(b-1,0)}.
\end{aligned}\] \end{coro}
\begin{pro}
For $b\in\{1,\dots, q\}$, $a\in\{1,\dots, p\}$, the following identities between Christoffel transformed recurrence matrices hold:
\[	
\begin{aligned}
U_{b}\mathscr T^{[N,N-p]}_{(b-1,0)}
\left(U_b^{[N-p]}\right)^{-1}%
&=\mathscr T^{[N,N-p]}_{(b,0)},&
\left(L^{[N-q]}_{a}\right)^{-1}\mathscr T^{[N-q,N]}_{(0,a-1)}
L_a%
&=\mathscr T^{[N-q,N]}_{(0,a)}.
	\end{aligned}
	\]
\end{pro}
\begin{proof}
It is a direct consequence of
	\begin{align*}
	\mathscr T_{(0,a)}^{[N-q,N]}&=\mathscr{L}_{N-q,(0,a)}	\Lambda^{[N-q,N]}_{[q]}\mathscr L_{N,(0,a)}^{-1},&
	\mathscr T^{[N,N-p]}_{(b,0)}&=
	\mathscr U_{N,(b,0)}^{-1}	\left(\Lambda^{[N-p,N]}_{[p]}\right)^\top \mathscr{U}_{N-p,(b,0)},
\end{align*}
and
\[
	\begin{aligned}
		U_b&\coloneq  \mathscr U_{N,(b,0)}^{-1}\mathscr U_{N,(b-1,0)},&
		L_a&\coloneq  \mathscr L_{N,(0,a-1)}\mathscr L_{N,(0,a)}^{-1}.
	\end{aligned}\]
\end{proof}

\begin{rem}
	In  the truncated scenario the basic Christoffel transformations do not induce anymore a permutation of the bidiagonal factors for the corresponding transformed recurrence  matrix, as it does in the infinite case.  To see this let put $b=1$, and look at 
	\[
U_{1}\mathscr T^{[N,N-p]}_{}
\left(U_1^{[N-p]}\right)^{-1}%
=\mathscr T^{[N,N-p]}_{(1,0)}\]
and, recalling Proposition \ref{pro:structure_T}, we know that
\[	\mathscr	 T^{[N,N-p]}=	\begin{bNiceArray}{cw{c}{2cm}ccw{c}{1cm}}[margin,cell-space-bottom-limit=5pt,]
	\Block[borders=bottom]{5-5}<\large>{\mathscr T_{N-p}}&&&&
	\\
	\\
	\\\\\\
	\Block[borders=right,c]{2-3}<\large>{0_{p\times (N-2p)}} &&&\Block[c]{2-2}<>{T^{[N,N-p]}}&	\\
	\\	\end{bNiceArray}, \]
	and we get
	\[
	\mathscr T_{N-p,(1,0)}=U^{[N-p]}_1L^{[N-p]}_1\cdots L^{[N-p]}_pU^{[N-p]}_q\cdots U^{[N-p]}_2+
\begin{bNiceArray}{cw{c}{2cm}ccw{c}{1.3cm}}[margin,cell-space-bottom-limit=5pt,]
	\Block[borders=bottom]{5-5}<\large>{0_{(N-p-1)\times (N-p)}}&&&&
	\\
	\\
	\\\\\\
	\Block[borders=right,c]{2-3}<\large>{0_{1 \times (N-2p)}} &&&\Block[c]{2-2}<>{e_1^\top T^{[N,N-p]}}&	\\
	\\	\end{bNiceArray}.
	\]
\end{rem}

\section{Christoffel formulas for the bidiagonal 	factors}

We now proceed to use the basic Christoffel transformations and the corresponding  Christoffel formulas in order to derive closed determinantal expressions of the entries in the bidiagonal matrices in terms of determinants of the original polynomials $A^{[N]}(x)$ and $B^{[N]}(X)$ and its truncations evaluated at $x=0$. The required determinants  are given by:

 \begin{defi}
 	[$\tau$-determinants]
 	Let us introduce the following determinants
 \[	\begin{aligned}
 			\tau^B_{b,n}&\coloneq \begin{vNiceArray}{ccc}[cell-space-limits=1pt]
 				B^{(1)}_{n}(0)		& \Cdots	&B^{(b)}_{n}(0) \\
 				\Vdots[shorten-end=-5pt] & &\Vdots[shorten-end=-5pt] 
 				\\\\
 				B^{(1)}_{n+b-1}(0)&\Cdots	& 	B^{(b)}_{n+b-1}(0)
 			\end{vNiceArray}, &
 			b&\in\{1,\dots, q\}, & n&\in\{0,1,\dots,N-b-1\},\\
 			\tau^A_{a,n}&\coloneq \begin{vNiceArray}{ccc}[cell-space-limits=1pt]
 				A^{(1)}_{n+a-1}(0)		& \Cdots	&A^{(1)}_{n}(0)	\\
 				\Vdots[shorten-end=0pt] & &\Vdots[shorten-end=0pt] 
 				\\\\
 				A^{(a)}_{n+a-1}(0)	&\Cdots	& A^{(a)}_{n}(0) 
 			\end{vNiceArray}, &
 			a&\in\{1,\dots, p\}, & n&\in\{0,1,\dots,N-a-1\},
 	\end{aligned}\]
 	and $\tau_{0,n}^A=\tau_{0,n}^B=1$.
 \end{defi}
 These determinants allow to characterize the existence of the orthogonality as was shown in \cite{Manuel-Miguel}, see also \cite{Manuel-Miguel_2} for the Geronimus situation.
 \begin{teo}[Existence of orthogonality (Mañas \& Rojas)]\phantom{hola}
 	\begin{enumerate}
 		\item The existence of the transformed orthogonality for $\d\mu_{(0,a)}$,  $	a\in\{1,\dots, p\}$, up to $N-a-1$ polynomials is equivalent to 
 \[	\begin{aligned}
 	\tau^A_{a,n}&\neq 0, &
 	a&\in\{1,\dots, p\}, & n&\in\{0,1,\dots,N-a-1\}.
 \end{aligned}\]
 \item  The existence of the transformed orthogonality for  $\d\mu_{(b,0)}$, $b\in\{1,\dots, q\}$, up to $N-b-1$ polynomial  is equivalent to 
 \[	\begin{aligned}
 	\tau^B_{b,n}&\neq 0 &
 	b&\in\{1,\dots, q\}, & n&\in\{0,1,\dots,N-b-1\}.
 \end{aligned}\]
 \end{enumerate}
 \end{teo}
 \begin{proof}
 	Is a direct consequence of the existence results described in \cite{Manuel-Miguel}.
 \end{proof}
 
 We proved in  \cite{aim} that
 \begin{teo}\label{teo:equatility determinants}
 	For $n\in\{0,1,\dots,\min(N-p,N-q)\}$ we find
\[
(-1)^{(q-1) n} \begin{vNiceArray}{ccc}[cell-space-limits=1pt]
		B^{(1)}_{n}(x)		& \Cdots	&B^{(q)}_{n}(x) \\
		\Vdots[shorten-end=-5pt] & &\Vdots[shorten-end=-5pt] 
		\\\\
		B^{(1)}_{n+q-1}(x)&\Cdots	& 	B^{(b)}_{n+q-1}(x)
	\end{vNiceArray}=(-1)^{(p-1) (n+1)}\mathscr T_{p, 0} \cdots \mathscr T_{n+p-1, n-1} 
	\begin{vNiceArray}{ccc}[cell-space-limits=1pt]
		A^{(1)}_{n+p-1}(x)		& \Cdots	&A^{(1)}_{n}(x)	\\
		\Vdots[shorten-end=0pt] & &\Vdots[shorten-end=0pt] 
		\\\\
		A^{(p)}_{n+p-1}(x)	&\Cdots	& A^{(p)}_{n}(x) 
	\end{vNiceArray}.
\]
 	\end{teo}
 	That immediately leads to:
 	\begin{coro}
 		The following identities hold
\begin{align}\label{eq:tau_A-B}
	 	(-1)^{(p-1) (n+1)}\mathscr T_{p, 0} \cdots \mathscr T_{n+p-1, n-1} 	\tau^A_{p,n} &=(-1)^{(q-1) n} \tau^B_{q,n},\\\label{eq:tau_quotient_A-B}
	 	(-1)^{p-1} \frac{1}{\mathscr T_{n+p, n} } \frac{\tau^A_{p,n}}{\tau^A_{p,n+1}} =(-1)^{q-1}  \frac{\tau^B_{q,n}}{\tau^B_{q,n+1}}. 
\end{align}
 
 	\end{coro}
 	\begin{proof}
 	Relation \eqref{eq:tau_A-B} follows from the evaluations of the determinants in Theorem \ref{teo:equatility determinants}, Equation \eqref{eq:tau_quotient_A-B} is an trivial consequence derived by taking ratios of the previous relations.
 	\end{proof}
 
 \begin{teo}\label{th:bidiagonal_Christoffel}
 	The entries of the bidiagonal matrices have the following expressions
\[ \begin{aligned}
 		U_{b,n}&=-\frac{\tau^B_{b-1,n}\tau^B_{b,n+1}}{\tau^B_{b-1,n+1}\tau^B_{b,n}}, &b&\in\{1,\dots, q\},  &n\in\{0,1,\dots,N-b-1\},\\
 		L_{a,n+1}&=-\frac{\tau^A_{a-1,n+2}\tau^A_{a,n}}{\tau^A_{a-1,n+1}\tau^A_{a,n+1}}, &a&\in\{1,\dots, p\},  &n\in\{0,1,\dots,N-a-1\}.
 \end{aligned}\]
 \end{teo}
 
 \begin{proof}
 Equation \eqref{eq:bidiagonal_1} imply
 \begin{align*}
 	\begin{bNiceArray}{c}[margin,cell-space-limits=2pt]
 		\mathscr L_{N,(1,0)}^{-1}\\\hline L^{[N+q-1,N]}_{(1,0)}
 	\end{bNiceArray}U_1=	\Lambda^{[N+q-1,N+q]}	\begin{bNiceArray}{c}[margin,cell-space-limits=2pt]
 		I_N\\\hline L^{[N+q,N]}\mathscr L_N
 	\end{bNiceArray} \mathscr L_N^{-1}
 \end{align*}
 and recalling that by definition
 \[
  	\mathscr L_{N+q-1,(1,0)} 
 \begin{bNiceArray}{c}[margin,cell-space-limits=2pt]
	\mathscr L_{N,(1,0)}^{-1}\\\hline L^{[N+q-1,N]}_{(1,0)}
 \end{bNiceArray}=\begin{bNiceArray}{c}[margin,cell-space-limits=2pt]
I_N\\
0_{(q-1)\times N}
 \end{bNiceArray}
 \]
we find that
 \[
 \begin{aligned}
 	\begin{bNiceArray}{c}[margin,cell-space-limits=2pt]
 		I_N\\\hline 0_{(q-1)\times N}
 	\end{bNiceArray}U_1B^{[N]}(x)&=\mathscr L_{N+q-1,(1,0)}	\Lambda^{[N+q-1,N+q]}	\begin{bNiceArray}{c}[margin,cell-space-limits=2pt]
 		I_N\\\hline L^{[N+q,N]}\mathscr L_N
 	\end{bNiceArray} \mathscr L_N^{-1}B^{[N]}(x)\\&=\mathscr L_{N+q-1,(1,0)}	\Lambda^{[N+q-1,N+q]}	\begin{bNiceArray}{c}[margin,cell-space-limits=2pt]
 		I_N\\\hline L^{[N+q,N]}\mathscr L_N
 	\end{bNiceArray} X_{[q]}^{[N]}(x)\\&=\mathscr L_{N+q-1,(1,0)}	\Lambda^{[N+q-1,N+q]}	\begin{bNiceArray}{c}[margin,cell-space-limits=2pt]
 		X_{[q]}^{[N]}(x)\\\hline L^{[N+q,N]}B^{[N]}(x)
 	\end{bNiceArray} \\
 	&=\begin{bNiceMatrix}
 		\mathscr L_{N-1,(1,0)}& 0_{(N-1)\times q}\\
 		\mathscr l^{[N-1+q,N-1]}_{(1,0)}& \mathscr{L}_{[N-1,N-1+q],(1,0)}
 	\end{bNiceMatrix}\begin{bNiceArray}{c}[margin,cell-space-limits=2pt]
 		X_{[q]}^{[N-1]}(x)\mathfrak X_{[q,1]}(x)\\\hline L^{[N+q,N]}B^{[N]}(x)
 	\end{bNiceArray} \
 \end{aligned}
 \]
 Hence, if $U_1^{[N-1,N]}$ is obtained by removing the last row in $U_1$,
  we get 
 \[
 U_1^{[N-1,N]} B^{[N]}(x)=B_{(1,0)}^{[N-1]}(x)\mathfrak X_{[q,1]}(x).
 \]
 The only eigenvalue of $\mathfrak X_{[q,1]}(x)$ is $0$ with right eigenvector give by $e_1\in\R^q$, that zero entries but for 1 in the first entry.
 \[
 U_1^{[N-1,N]} B^{[N]}(0)e_1=0_{(N-1)\times 1}.
 \]
 and, consequently, 
 \[
 U_1^{[N-1,N]} \begin{bNiceMatrix}[cell-space-limits=1pt]
 	B^{(1)}_0(0)\\B^{(1)}_1(0)\\\Vdots\\B^{(1)}_{N-1}(0)
 \end{bNiceMatrix}=0_{(N-1)\times 1}
 \]
 Therefore,
 \[
 \begin{aligned}
 	U_{1,n}&=-\frac{B^{(1)}_{n+1}(0)}{B^{(1)}_{n}(0)}, & n&\in\{0,1,\dots,N-2\}.
 \end{aligned}
 \]

 Now, Equation \eqref{eq:U2} can be written 
   \begin{align*}
\begin{bNiceArray}{c}[margin,cell-space-limits=2pt]
\mathscr L_{N,(2,0)} ^{-1}\\\hline L^{[N+q-2,N]}_{(2,0)}
\end{bNiceArray}U_2U_1=\Lambda^{[N+q-2,N+q-1]}\Lambda^{[N+q-1,N+q]}\begin{bNiceArray}{c}[margin,cell-space-limits=2pt]
 I_N	\\\hline L^{[N+q-1,N]}\mathscr L_{N}
 	\end{bNiceArray} 	\mathscr L_{N}^{-1}
 \end{align*}
 and recalling that
  \[
 \mathscr L_{N+q-1,(2,0)} 
 \begin{bNiceArray}{c}[margin,cell-space-limits=2pt]
 	\mathscr L_{N,(2,0)}^{-1}\\\hline L^{[N+q-2,N]}_{(2,0)}
 \end{bNiceArray}=\begin{bNiceArray}{c}[margin,cell-space-limits=2pt]
 	I_N\\
 	0_{(q-2)\times N}
 \end{bNiceArray}
 \]
 we express it  as
    \begin{align*}
 	\begin{bNiceArray}{c}[margin,cell-space-limits=2pt]
 		I_N\\\hline 	0_{(q-2)\times N}
 	\end{bNiceArray}U_2U_1=\mathscr L_{N+q-2,(2,0)} \Lambda^{[N+q-2,N+q-1]}\Lambda^{[N+q-1,N+q]}\begin{bNiceArray}{c}[margin,cell-space-limits=2pt]
 		I_N	\\\hline L^{[N+q-1,N]}\mathscr L_{N}
 	\end{bNiceArray} 	\mathscr L_{N}^{-1}.
 \end{align*}
 Then, proceeding as we did  for $U_1$:
 \[
 \begin{aligned}
 	\begin{bNiceArray}{c}[margin,cell-space-limits=2pt]
 		I_N\\\hline 0_{(q-2)\times N}
 	\end{bNiceArray}U_2U_1B^{[N]}(x)&=\mathscr L_{N+q-2,(2,0)}	\Lambda^{[N+q-2,N+q-1]}\Lambda^{[N+q-1,N+q]}	\begin{bNiceArray}{c}[margin,cell-space-limits=2pt]
 		I_N\\\hline L^{[N+q-1,N]}\mathscr L_N
 	\end{bNiceArray} \mathscr L_N^{-1}B^{[N]}(x)\\&=\mathscr L_{N+q-2,(2,0)}	\Lambda^{[N+q-2,N+q-1]}\Lambda^{[N+q-1,N+q]}	\begin{bNiceArray}{c}[margin,cell-space-limits=2pt]
 		I_N\\\hline L^{[N+q-1,N]}\mathscr L_N
 	\end{bNiceArray} X_{[q]}^{[N]}(x)\\&=\mathscr L_{N+q-2,(2,0)}	\Lambda^{[N+q-2,N+q-1]}\Lambda^{[N+q-1,N+q]}	\begin{bNiceArray}{c}[margin,cell-space-limits=2pt]
 		X_{[q]}^{[N]}(x)\\\hline L^{[N+q-1,N]}B^{[N]}(x)
 	\end{bNiceArray} \\
 	&=\begin{bNiceMatrix}
 		\mathscr L_{N-2,(2,0)}& 0_{(N-2)\times q}\\
 		\mathscr l^{[N-2+q,N-2]}_{(2,0)}& \mathscr{L}_{[N-2,N-2+q],(2,0)}
 	\end{bNiceMatrix}\begin{bNiceArray}{c}[margin,cell-space-limits=2pt]
 		X_{[q]}^{[N-2]}(x)\mathfrak X^2_{[q,1]}(x)\\\hline L^{[N+q,N]}B^{[N]}(x)
 	\end{bNiceArray} 
 \end{aligned}
 \]
 Thus, if $(U_2U_1)^{[N-2,N]}$ is obtained by removing the last two rows in $U_2U_1$,
 we obtain 
 \[
(U_2U_1)^{[N-2,N]} B^{[N]}(x)=B_{(2,0)}^{[N-2]}(x)\mathfrak X_{[q,1]}^2(x).
 \]
  The only eigenvalue of $\mathfrak X^2_{[q,1]}(x)$ is $0$, which is double, and  with right eigenvector given by $e_1,e_2\in\R^q$, that zero entries but for 1 in the first and second entries, respectively
 \[
(U_2 U_1)^{[N-2,N]} B^{[N]}(0)\begin{bNiceArray}{c|c}
	e_1 &e_2
\end{bNiceArray}=0_{(N-2)\times 2}.
 \]
so that
 \[
(U_2 U_1)^{[N-2,N]} \begin{bNiceMatrix}[cell-space-limits=1pt]
 	B^{(1)}_0(0) &  	B^{(2)}_0(0) \\B^{(1)}_1(0)& 	B^{(2)}_1(0) \\\Vdots&\Vdots\\B^{(1)}_{N-1}(0)& 	B^{(2)}_{N-1}(0) 
 \end{bNiceMatrix}=0_{(N-2)\times 2}.
 \]
 That is, for $n\in\{0,1,\dots N-3\}$, 
 \begin{align*}
 	(U_2U_1)_{n,n} 	B^{(1)}_n(0) +	(U_2U_1)_{n,n+1} 	B^{(1)}_{n+1}(0)+		B^{(1)}_{n+2}(0)&=0,\\  
 	 	(U_2U_1)_{n,n} 	B^{(2)}_n(0) +	(U_2U_1)_{n,n+1} 	B^{(2)}_{n+1}(0)+		B^{(2)}_{n+2}(0)&=0.
 \end{align*}
 This system can be written as follows
 \[
 \begin{bNiceArray}{cc}
 (U_2U_1)_{n,n} &
 (U_2U_1)_{n,n+1}
 \end{bNiceArray}=-\begin{bNiceArray}{cc}
 B^{(1)}_{n+2}(0)&	B^{(2)}_{n+2}(0)
 \end{bNiceArray}\begin{bNiceArray}{cc}[cell-space-limits=1pt]
	B^{(1)}_n(0) & B^{(2)}_n(0) \\
	B^{(1)}_{n+1}(0)		& 	B^{(2)}_{n+1}(0)
 \end{bNiceArray}^{-1}
 \]
 and, consequently,
 \[
\begin{aligned}
	  (U_2U_1)_{n,n}&=-\begin{bNiceArray}{cc}
	  	B^{(1)}_{n+2}(0)&	B^{(2)}_{n+2}(0)
	  \end{bNiceArray}\begin{bNiceArray}{cc}[cell-space-limits=1pt]
	  	B^{(1)}_n(0) & B^{(2)}_n(0) \\
	  	B^{(1)}_{n+1}(0)		& 	B^{(2)}_{n+1}(0)
	  \end{bNiceArray}^{-1}\begin{bNiceArray}{c}
	  	1 \\ 0
	  \end{bNiceArray}=\Theta_*\begin{bNiceArray}{ccc}[cell-space-limits=1pt]
		B^{(1)}_n(0) & 	B^{(2)}_n(0)&1\\
	    B^{(1)}_{n+1}(0)& 	B^{(2)}_{n+1}(0)&0\\
		B^{(1)}_{n+2}(0)& 	B^{(2)}_{n+2}(0)&0
	  \end{bNiceArray},\\
	   (U_2U_1)_{n,n+1}&=-\begin{bNiceArray}{cc}
	   	B^{(1)}_{n+2}(0)&	B^{(2)}_{n+2}(0)
	   \end{bNiceArray}\begin{bNiceArray}{cc}[cell-space-limits=1pt]
	   	B^{(1)}_n(0) & B^{(2)}_n(0) \\
	   	B^{(1)}_{n+1}(0)		& 	B^{(2)}_{n+1}(0)
	   \end{bNiceArray}^{-1}\begin{bNiceArray}{c}
	   	0\\ 1
	   \end{bNiceArray}=\Theta_*\begin{bNiceArray}{ccc}[cell-space-limits=1pt]
	   	B^{(1)}_n(0) & 	B^{(2)}_n(0)&0\\
	   	B^{(1)}_{n+1}(0)& 	B^{(2)}_{n+1}(0)&1\\
	   	B^{(1)}_{n+2}(0)& 	B^{(2)}_{n+2}(0)&0
	   \end{bNiceArray}.
\end{aligned}
 \]
 Where $\Theta_*$ denotes the last quasideterminant \cite{OLVER}, that in the scalar case leads to the following expressions
 \begin{align*}
 	(U_2U_1)_{n,n}=U_{2,n}U_{1,n}&=\frac{\begin{vNiceArray}{cc}[cell-space-limits=1pt]
 	B^{(1)}_{n+1}(0)		& 	B^{(2)}_{n+1}(0) \\
 			B^{(1)}_{n+2}(0)		& 	B^{(2)}_{n+2}(0)
 	\end{vNiceArray}}{\begin{vNiceArray}{cc}[cell-space-limits=1pt]
 			B^{(1)}_{n}(0)		& 	B^{(2)}_{n}(0) \\
 			B^{(1)}_{n+1}(0)		& 	B^{(2)}_{n+1}(0)
 	\end{vNiceArray}},&
 \end{align*}
 for $n\in\{0,1,\dots N-3\}$.
 Hence, we find 
 \[
 \begin{aligned}
 	U_{2,n}&=-\frac{B^{(1)}_{n}(0)}{B^{(1)}_{n+1}(0)}\frac{\begin{vNiceArray}{cc}[cell-space-limits=1pt]
 			B^{(1)}_{n+1}(0)		& 	B^{(2)}_{n+1}(0) \\
 			B^{(1)}_{n+2}(0)		& 	B^{(2)}_{n+2}(0)
 	\end{vNiceArray}}{\begin{vNiceArray}{cc}[cell-space-limits=1pt]
 			B^{(1)}_{n}(0)		& 	B^{(2)}_{n}(0) \\
 			B^{(1)}_{n+1}(0)		& 	B^{(2)}_{n+1}(0)
 	\end{vNiceArray}}
 ,&n&\in\{0,1,\dots N-3\}.
 \end{aligned}
 \]

In general, it holds that 
\[
(U_b\cdots U_1)^{[N-b,N]}B^{[N]}(x)=B^{[N-b]}_{(b,0)}(x)\mathfrak{X}_{[q,1]}^b(x),
\]
 so that  \[
 (U_b\cdots U_1)^{[N-b,N]} \begin{bNiceArray}{ccc}[cell-space-limits=1pt]
 	B^{(1)}_0(0) &\Cdots  	&B^{(b)}_0(0) \\B^{(1)}_1(0)& 	\Cdots&B^{(b)}_1(0) \\\Vdots&&\Vdots\\\\B^{(1)}_{N-1}(0)& 	\Cdots&B^{(b)}_{N-1}(0) 
 \end{bNiceArray}=0_{(N-b)\times b}.
 \]
 This can be rewritten as
  \[ \begin{aligned}
  	(U_b\cdots U_1)_{n,n}&=-\begin{bNiceArray}{ccc}
 	B^{(1)}_{n+b}(0)&	\Cdots & B^{(b)}_{n+b}(0)
 \end{bNiceArray}\begin{bNiceArray}{ccc}[cell-space-limits=1pt]
 	B^{(1)}_n(0) & \Cdots &B^{(b)}_n(0) \\\Vdots &&\Vdots\\\\
 	B^{(1)}_{n+b-1}(0)	&\Cdots	& 	B^{(b)}_{n+b-1}(0)
 \end{bNiceArray}^{-1}\begin{bNiceArray}{c}
 	1 \\ 0\\\Vdots\\0
 \end{bNiceArray}\\&=(-1)^b\frac{\tau^B_{b,n+1}}{\tau^B_{b,n}},
  \end{aligned}\]
for $b\in\{1,\dots, q\},  n\in\{0,1,\dots,N-b-1\}.$ 
Therefore, we finally obtain the stated formula.

We now compute the $L_m$.
 Using \eqref{eq:ML1} we deduce that
  \[
 \begin{aligned}
 	A^{[N]}(x)	 L_1\begin{bNiceArray}{c|c}[margin,cell-space-limits=2pt]
 	I_N&  0_{N\times (p-1)}
 \end{bNiceArray}&=A^{[N]}(x)	\mathscr U_{N}^{-1} \begin{bNiceArray}{c|c}[margin,cell-space-limits=2pt]
 		I_N &  \mathscr U_{N}U^{[N,N+p]}
 	\end{bNiceArray} \left( \Lambda^{[N+p-1,N+p]} \right)^\top\mathscr U_{N+p-1,(0,1)}\\&= 	\left(X_{[p]}^{[N]}(x)\right)^\top	 \begin{bNiceArray}{c|c}[margin,cell-space-limits=2pt]
 	I_N &  \mathscr U_{N}U^{[N,N+p]}
 	\end{bNiceArray} \left( \Lambda^{[N+p-1,N+p]} \right)^\top\mathscr U_{N+p-1,(0,1)}\\
 	&=	 \begin{bNiceArray}{c|c}[margin,cell-space-limits=2pt]
 		\left(X_{[p]}^{[N]}(x)\right)^\top &  A^{[N]}(x)U^{[N,N+p]}
 	\end{bNiceArray} \left( \Lambda^{[N+p-1,N+p]} \right)^\top\mathscr U_{N+p-1,(0,1)}
 	\\&=\begin{bNiceArray}{c|c}[margin,cell-space-limits=2pt]
 	\left(\mathfrak X_{[p,1]}(x)\right)^\top	\left(X_{[p]}^{[N-1]}(x)\right)^\top &  A^{[N]}(x)U^{[N,N+p]}
 	\end{bNiceArray} \mathscr U_{N+p-1,(0,1)}
 		\\&=\begin{multlined}[t][.6\textwidth]
 			\begin{bNiceArray}{c|c}[margin,cell-space-limits=2pt]
 		\left(\mathfrak X_{[p,1]}(x)\right)^\top	\left(X_{[p]}^{[N-1]}(x)\right)^\top &  A^{[N]}(x)U^{[N,N+p]}
 	\end{bNiceArray}  	\\\times \begin{bNiceMatrix}
 	\mathscr U_{N-1,(0,1)} & \mathscr u^{[N+p-1,N-1]}_{(0,1)} \\
 	0_{(p-1)\times (N-1)}	&\mathscr U_{[N+p-1,N-1],(1,1,0)}
 	\end{bNiceMatrix}
 		\end{multlined}
 \end{aligned}
 \]
 Hence, if we denote by $L_1^{[N,N-1]}$ the truncated matrix obtained from $L_1$ by erasing the last column, we get
 \[
 A^{[N]}(x)L_1^{[N,N-1]}=\left(\mathfrak X_{[p,1]}(x)\right)^\top A^{[N-1]}_{(0,1)}(x).
 	\]
Taking into account that for the perturbation polynomial its $\det\left(\mathfrak X_{[p]}(x)\right)^\top $ has only a root  that is located at  $x=0$,  and its left eigenvector is $e_1^\top\in(\R^p)^*$, that has all its components equal to zero but for the first which equals one, we have
 \[
e_1^\top A^{[N]}(0)L_1^{[N,N-1]}=e_1^\top\left(\mathfrak X_{[p]}(x)\right)^\top A^{[N-1]}_{(0,1)}(x)=0_{1\times (N-1)}.
\]
and, consequently we obtain
\begin{align*}
	\begin{bNiceMatrix}
		A^{(1)}_0(0)&A^{(1)}_1(0)&\Cdots &A^{(1)}_{N-1}(0)
	\end{bNiceMatrix}L_1^{[N,N-1]}=0_{1\times (N-1)},
\end{align*}
that leads to
\[
\begin{aligned}
	L_{1,n+1}&=-\frac{A^{(1)}_{n}(0)}{A^{(1)}_{n+1}(0)}, & n&\in\{0,\dots,N-2\}.
\end{aligned}
\]
For the next bidiagonal matrix $L_2$ we use
  \[\begin{multlined}[t][\textwidth]
	A^{[N]}(x)	 L_1L_2\begin{bNiceArray}{c|c}[margin,cell-space-limits=2pt]
		I_N&  0_{N\times (p-2)}
	\end{bNiceArray}\\=A^{[N]}(x)	\mathscr U_{N}^{-1} \begin{bNiceArray}{c|c}[margin,cell-space-limits=2pt]
		I_N &  \mathscr U_{N}U^{[N,N+p-1]}
	\end{bNiceArray} \left( \Lambda^{[N+p-1,N+p]} \right)^\top \left( \Lambda^{[N+p-2,N+p-1]} \right)^\top\mathscr U_{N+p-2,(0,2)}
\end{multlined}
\]
that implies
 \[
A^{[N]}(x)(L_1L_2)^{[N,N-2]}=\left(\mathfrak X_{[p,1]}^2(x)\right)^\top A^{[N-2]}_{(0,2)}(x).
\]
Consequently, 
\[
\begin{aligned}
\begin{bNiceArray}{cccc}[cell-space-limits=1pt]
	A^{(1)}_0(0)&A^{(1)}_1(0)&\Cdots &A^{(1)}_{N-1}\\
	A^{(2)}_0(0)&A^{(2)}_1(0)&\Cdots &A^{(2)}_{N-1}
\end{bNiceArray}(L_1L_2)^{[N,N-2]}=0_{2\times (N-2)}.
\end{aligned}\]
Hence, for $n\in\{0,1,\dots,N-3\}$ we get
\begin{align*}
		A^{(1)}_{n+2}(0)(L_1L_2)_{n+2,n}+	A^{(1)}_{n+1}(0)(L_1L_2)_{n+1,n}+A^{(1)}_{n}(0)&=0,\\
		A^{(2)}_{n+2}(0)(L_1L_2)_{n+2,n}+	A^{(2)}_{n+1}(0)(L_1L_2)_{n+1,n}+A^{(2)}_{n}(0)&=0,
\end{align*}
or
\[
\begin{aligned}
\begin{bNiceArray}{c}
	(L_1L_2)_{n+2,n}\\
	(L_1L_2)_{n+1,n}
	\end{bNiceArray}	=	-\begin{bNiceArray}{cc}[cell-space-limits=2pt]
	A^{(1)}_{n+2}(0) &A^{(1)}_{n+1}(0)\\
	A^{(2)}_{n+2}(0)&A^{(2)}_{n+1}(0)
	\end{bNiceArray}^{-1}\begin{bNiceArray}{cc}
	A^{(1)}_{n}(0) \\
	A^{(2)}_{n}(0)
	\end{bNiceArray}.	
\end{aligned}
\]
The solution is
\[
\begin{aligned}
	(L_1L_2)_{n+2,n}&=	-\begin{bNiceMatrix}
		1&0
	\end{bNiceMatrix}\begin{bNiceArray}{cc}[cell-space-limits=2pt]
		A^{(1)}_{n+2}(0) &A^{(1)}_{n+1}(0)\\
		A^{(2)}_{n+2}(0)&A^{(2)}_{n+1}(0)
	\end{bNiceArray}^{-1}\begin{bNiceArray}{cc}
		A^{(1)}_{n}(0) \\
		A^{(2)}_{n}(0)
	\end{bNiceArray}=\frac{\begin{vNiceArray}{cc}[cell-space-limits=2pt]
			A^{(1)}_{n+1}(0) &A^{(1)}_{n}(0)\\
			A^{(2)}_{n+1}(0)&A^{(2)}_{n}(0)
			\end{vNiceArray}}{\begin{vNiceArray}{cc}[cell-space-limits=2pt]
		A^{(1)}_{n+2}(0) &A^{(1)}_{n+1}(0)\\
		A^{(2)}_{n+2}(0)&A^{(2)}_{n+1}(0)
		\end{vNiceArray}},\\
(L_1L_2)_{n+1,n}&=-\begin{bNiceMatrix}
	0&1
\end{bNiceMatrix}\begin{bNiceArray}{cc}[cell-space-limits=2pt]
A^{(1)}_{n+2}(0) &A^{(1)}_{n+1}(0)\\
A^{(2)}_{n+2}(0)&A^{(2)}_{n+1}(0)
\end{bNiceArray}^{-1}\begin{bNiceArray}{cc}
A^{(1)}_{n}(0) \\
A^{(2)}_{n}(0)
\end{bNiceArray}=-\frac{\begin{vNiceArray}{cc}[cell-space-limits=2pt]
		A^{(1)}_{n+2}(0) &A^{(1)}_{n}(0)\\
		A^{(2)}_{n+2}(0)&A^{(2)}_{n}(0)
		\end{vNiceArray}}{\begin{vNiceArray}{cc}[cell-space-limits=2pt]
	A^{(1)}_{n+2}(0) &A^{(1)}_{n+1}(0)\\
	A^{(2)}_{n+2}(0)&A^{(2)}_{n+1}(0)
	\end{vNiceArray}}.
\end{aligned}\]
As 
\[
(L_1L_2)_{n+2,n}=L_{1,n+2}L_{2,n+1}
\]
we obtain that
\[
\begin{aligned}
	L_{2,n+1}&=-\frac{A^{(1)}_{n+2}(0)}{A^{(1)}_{n+1}(0)}\frac{\begin{vNiceArray}{cc}[cell-space-limits=2pt]
		A^{(1)}_{n+1}(0) &A^{(1)}_{n}(0)\\
		A^{(2)}_{n+1}(0)&A^{(2)}_{n}(0)
\end{vNiceArray}}{\begin{vNiceArray}{cc}[cell-space-limits=2pt]
		A^{(1)}_{n+2}(0) &A^{(1)}_{n+1}(0)\\
		A^{(2)}_{n+2}(0)&A^{(2)}_{n+1}(0)
\end{vNiceArray}},& n&\in\{0,\dots,N-3\}.
\end{aligned}
\]

In general, it holds that 
\[
A^{[N]}(x)(L_1\cdots L_a)^{[N,N-a]}=\left(\mathfrak{X}_{[p,1]}^a(x)\right)^\top A^{[N-a]}_{(a,0)}(x),
\]
so that  \[
\begin{bNiceArray}{cccc}[cell-space-limits=1pt]
A^{(1)}_0(0)&A^{(1)}_1(0)&\Cdots & A^{(1)}_{N-1}(0)\\
\Vdots &\Vdots & &\Vdots\\
A^{(a)}_0(0)&A^{(a)}_1(0)&\Cdots & A^{(a)}_{N-1}(0)
\end{bNiceArray}(L_1\cdots L_a)^{[N,N-a]} =0_{a\times (N-a)}.
\]
that can be rewritten as
\[ \begin{aligned}
	(L_1\cdots L_j)_{n+a,n}&=-\begin{bNiceArray}{cccc}
		1 & 0&\Cdots&0
	\end{bNiceArray}\begin{bNiceArray}{ccc}[cell-space-limits=1pt]
		A^{(1)}_{n+a-1}(0) & \Cdots &A^{(1)}_n(0) \\\Vdots &&\Vdots\\\\
		A^{(a)}_{n+a-1}(0)	&\Cdots	& 	A^{(a)}_{n}(0)
	\end{bNiceArray}^{-1}\begin{bNiceArray}{c}
	A^{(1)}_{n}(0)\\	\Vdots \\ A^{(p)}_{n}(0)
	\end{bNiceArray}\\&=(-1)^a\frac{\tau^A_{a,n}}{\tau^A_{a,n+1}},
\end{aligned}\]
for $a\in\{1,\dots, q\},  n\in\{0,1,\dots,N-a-1\}$. Note that
\[
(L_1\dots L_a)_{n+a,n}=(L_1\dots L_{a-1})_{n+a,n+1}L_{a,n+1}
\]
and we obtain the desired result.
\end{proof}
\begin{rem}
 An expression for $U_{1,N-1}$ is missing in this formula. However, we can get an expression for it. In the one hand, 
		   the $N$-th entry of $U_1B^[N](0)e_1$ is $U_{1,N-1}B^{(1)}_{N-1}(0)$. In the other hand, this last entry must be the first entry of 
		   \[
		 \left(  	\mathscr l^{[N-1+q,N-1]}_{(1,0)}X^{[N-1]}(x)\mathfrak X_{[q,1]}(0)+ \mathscr{L}_{[N-1,N-1+q],(1,0)} L^{[N+q,N]}B^{[N]}(0)\right)e_1,
		   \]
		   and as $\mathfrak X_{[q,1]}(0)e_1$ is the zero-vector we get that we are only interested  in the first entry 
		   \[
		e_1^\top   \mathscr{L}_{[N-1,N-1+q],(1,0)} L^{[N+q,N]}B^{[N]}(0)
		   e_1.\]
		   Hence, as $  \mathscr{L}_{[N-1,N-1+q],(1,0)} $ is a lower unitriangular, it does not intervene in the first entry , so we finally obtain
		   \[
		U_{1,N-1}=\frac{1}{B^{(1)}_{N-1}(0)}e_1^\top\int x^{N_q} \mathfrak{X}_{[q,s_q]}(x)  \d\mu(x) A^{[N]}(x)\begin{bNiceArray}{c}[cell-space-limits=2pt]
				B^{(1)}_0(0)\\B^{(1)}_2(0)\\\Vdots\\B^{(1)}_{N-1}(0)
			\end{bNiceArray},
			   \]
			   were \eqref {eq:L} has been used.
\end{rem}
\begin{rem}
	From
	\begin{align*}
			(U_2U_1)_{n,n+1}=U_{2,n}+U_{1,n+1}&=-\frac{\begin{vNiceArray}{cc}[cell-space-limits=1pt]
				 		B^{(1)}_n(0) & B^{(2)}_n(0) \\
				 		B^{(1)}_{n+2}(0)		& 	B^{(2)}_{n+2}(0)
				 \end{vNiceArray}}{\begin{vNiceArray}{cc}[cell-space-limits=1pt]
				 		B^{(1)}_{n}(0)		& 	B^{(2)}_{n}(0) \\
				 		B^{(1)}_{n+1}(0)		& 	B^{(2)}_{n+1}(0)
				 \end{vNiceArray}},
	\end{align*}
	we derive  alternative expressions
	\[
	\begin{aligned}
		U_{2,n}&
		=-\frac{\begin{vNiceArray}{cc}[cell-space-limits=1pt]
				 	B^{(1)}_n(0) & B^{(2)}_n(0) \\
				 	B^{(1)}_{n+2}(0)		& 	B^{(2)}_{n+2}(0)
				 \end{vNiceArray}}{\begin{vNiceArray}{cc}[cell-space-limits=1pt]
				 B^{(1)}_{n}(0)		& 	B^{(2)}_{n}(0) \\
				 B^{(1)}_{n+1}(0)		& 	B^{(2)}_{n+1}(0)
				\end{vNiceArray}}+\frac{B^{(1)}_{n+2}(0)}{B^{(1)}_{n+1}(0)}
				,&n&\in\{0,1,\dots N-3\}.
			\end{aligned}
			\]
			
\end{rem}
\section{A study case: Hahn multiple orthogonal polynomials}

In this section  we put $q=1$, while $p$ remains arbitrary. That is we are dealing with multiple orthogonality, which is the standard one and not of mixed type,
Now, the step-line recursion relations read
\[	
	\begin{aligned}
		x B_{n}(x) &= B_{n + 1}(x) + b_{n}^0 B_{n}(x) + \sum_{j=1}^p b^j_{n} B_{n -j}(x),& n&\in\{0,1,\dots,N-2\},
		\\
		x A^{(a)}_{n}(x) &= A^{(a)}_{n - 1}(x) + b_{n }^0 A^{(a)}_{n}(x) + \sum_{j=1}^p b^j_{n+j} A^{(a)}_{n+j}(x),& n&\in\{0,1,\dots,N-p-1\}
		,& a&\in\{0,1,\ldots,p\}
	\end{aligned}
\]
with \(B_{-1},\dots,B_{-p}= 0\) and \(A^{(1)}_{-1},\ldots,A^{(p)}_{-1}= 0\).
and the recursion matrix is 
\[	\begin{aligned}
\mathscr	T_{N-p}\coloneq{\begin{bNiceMatrix}
			b^{0}_0 & 1 & 0&\Cdots &&&&0\\
			b^{1}_1 & b^{0}_1 & 1 &\Ddots&&&&\Vdots\\
			&b^{1}_{2}&b^{0}_{2}&\Ddots&&&&\\  \Vdots&&\Ddots[shorten-end=-10pt]&\Ddots&&&&\\
			b^{p}_p&&&&&&&\\
			0&&&&&&&0\\
			\Vdots&\Ddots[shorten-end=-3pt]&\Ddots[shorten-end=-5pt]&&&&&1\\ 0&\Cdots&0&b^{p}_{N-p-1}&\Cdots&&b^{1}_{N-p-1}&b^{0}_{N-p-1}
	\end{bNiceMatrix}}
\end{aligned}\]

Before delving into the multiple Hahn polynomials, it's important to recall that these polynomials can be represented using generalized hypergeometric series \cite{andrews,slater},
\begin{align}
	\label{GHF}
	\pFq{p}{q}{a_1,\ldots,a_p}{\alpha_1,\ldots,\alpha_q}{x}\coloneqq \sum_{l=0}^{\infty}\dfrac{(a_1)_l\cdots(a_p)_l}{(\alpha_1)_l\cdots(\alpha_q)_l}\dfrac{x^l}{l!},
\end{align}
with the Pochhammer symbols given by
\begin{align*}
	(x)_n\coloneqq\dfrac{\Gamma(x+n)}{\Gamma(x)}=\begin{cases}
		x(x+1)\cdots(x+n-1)\;\text{if}\;n\in\N,\\
		1\;\text{if}\;n=0.
	\end{cases}
\end{align*}

The weight functions for the Hahn family are defined as
\begin{equation}
	\label{WeightsHahn}
	\begin{aligned}
		w_{i}(x;\alpha_i,\beta,N)&=\dfrac{\Gamma(\alpha_i+x+1)}{\Gamma(\alpha_i+1)\Gamma(x+1)}\dfrac{\Gamma(\beta+N-x+1)}{\Gamma(\beta+1)\Gamma(N-x+1)}, & i &\in\{1,\ldots,p\},& \Delta &=\{0,\ldots,N\},
	\end{aligned}
\end{equation}
with
\( \alpha_1,\ldots,\alpha_p,\beta>-1\). We need a vector on nonnegative integers indexes $\vec n=(n_1,\dots, n_p)$, with $n_i\in\N_0$ In  order to have an AT system, we require that \( \sz{n} \leqslant N\in\mathbb N_0\) and \( \alpha_i-\alpha_j\not\in\mathbb Z\) for \( i\neq j\). This ensures the existence of the orthogonal polynomial sup to $|\vec n|=N$. We will denote \(\vec{\alpha}\coloneq(\alpha_1,\ldots,\alpha_p)\).

The corresponding type II polynomials were first found for \( p=2\) in \cite{Arvesu} and then generalized for \( p \geqslant 2\) in~\cite{AskeyII}. 
In \cite{BDFMW}, the following alternative representation was proven
\begin{multline}
	\label{HahnTypeIIWeighted}
	B_{\vec{n}}(x;\vec{\alpha},\beta,N)
	=
	{
		(-1)^{|\vec{n}|}}\dfrac{\Gamma(N+\beta+1)}{(N-|\vec{n}|)!}\prod_{i=1}^p\dfrac{(\alpha_i+1)_{n_i}}{(\alpha_i+\beta+|\vec{n}|+1)_{n_i}} \dfrac{\Gamma(N-x+1)}{\Gamma(\beta+N-x+1)} \\
	\times \pFq{p+2}{p+1}{-|\vec{n}|-\beta,-x,\vec{\alpha}+\vec{n}+\vec{1}_p}{-N-\beta,\vec{\alpha}+\vec{1}_p}{1}.
\end{multline}
For $p=2$, the first explicit hypergeometric representation of the Hahn polynomials of type I  were given in \cite{JMAA}. Then, the case for \( p \geqslant 2 \) was presented in  in \cite{BDFMW} and reads as follows:
\begin{multline}
	\label{HahnTypeI}
	A_{\vec n}^{(a)}(x;\vec{\alpha},\beta,N)
	=\dfrac{(-1)^{|\vec{n}|-1}(N+1-|\vec{n}|)!}{(n_a-1)!(\beta+1)_{|\vec{n}|-1}(\alpha_a+\beta+|\vec{n}|)_{N+2-|\vec{n}|}}\dfrac{\prod_{k=1}^p({\alpha}_k+\beta+|\vec{n}|)_{n_k}}{\prod_{k=1,k\neq a}^p({\alpha}_k-{\alpha}_a)_{{n}_k}}\\
	\times\pFq{p+2}{p+1}{-n_a+1,\alpha_a+\beta+|\vec{n}|,(\alpha_a+1)\vec{1}_{p-1}-\vec{\alpha}^{*a}-\vec{n}^{*a},x+\alpha_a+1}{\alpha_a+1,(\alpha_a+1)\vec{1}_{p-1}-\vec{\alpha}^{*a},\alpha_a+\beta+N+2}{1}.
\end{multline}

\begin{rem}
	Notice that we used the following notation. We denoted the unit vector by
	\(\vec{1}_p\coloneq (1,\ldots,1)\in\mathbb R^p\). Given a vector $\vec{v}\in\mathbb R^p$, we wrote $\vec{v}^{\ast q}\in\mathbb R^{p-1}$ for the vector obtained from $\vec{v}$ after removing its $q$-th entry. 
\end{rem}
The evaluations at $x=0$ give
\begin{equation}
	\label{HahnTypeIIWeighted0}
	B_{\vec{n}}(0)
	=
		(-1)^{|\vec{n}|}\dfrac{N!}{(N-|\vec{n}|)!}
		\prod_{i=1}^p\dfrac{(\alpha_i+1)_{n_i}}{(\alpha_i+\beta+|\vec{n}|+1)_{n_i}} 
\end{equation}
and
\begin{multline}
	\label{HahnTypeI}
	A_{\vec n}^{(a)}(0)
	=\dfrac{(-1)^{|\vec{n}|-1}(N+1-|\vec{n}|)!}{(n_a-1)!(\beta+1)_{|\vec{n}|-1}(\alpha_a+\beta+|\vec{n}|)_{N+2-|\vec{n}|}}\dfrac{\prod_{k=1}^p({\alpha}_k+\beta+|\vec{n}|)_{n_k}}{\prod_{k=1,k\neq a}^p({\alpha}_k-{\alpha}_a)_{{n}_k}}\\
	\times\pFq{p+1}{p}{-n_a+1,\alpha_a+\beta+|\vec{n}|,(\alpha_a+1)\vec{1}_{p-1}-\vec{\alpha}^{*a}-\vec{n}^{*a}}{(\alpha_a+1)\vec{1}_{p-1}-\vec{\alpha}^{*a},\alpha_a+\beta+N+2}{1}.
\end{multline}

The indices considered are general. The step-line corresponds to considering multi-index of the form
\[\vec{n}=(\underbrace{m+1,\ldots,m+1}_{s\,\text{times}},\underbrace{m,\ldots,m}_{p-s\,\text{times}}),\quad m\in\N_0,\quad s\in\{0,\ldots,p-1\}\]
This way, these multi-index can be ordered in a sequence such that the modulus \[|\vec{n}|=pm+s
\]
increase by \(1\) as follows
\[\left\lbrace(0,0,\ldots,0),(1,0,\ldots,0),(1,1,\ldots,0),\ldots,(1,1,\ldots,1),(2,1,\ldots,1),\ldots\right\rbrace\]

This is what is known as the step-line and allows to relabel the multiple polynomials 
\[\begin{aligned}
	B_{pm+s}&\coloneq B_{(m+1,\ldots,m+1,m,\ldots,m)}\\
	A^{(a)}_{pm+s-1}&\coloneq A^{(a)}_{(m+1,\ldots,m+1,m,\ldots,m)},& a&\in\{1,\ldots,p\}
\end{aligned}\]

For the recurrence coefficients in the step-line we  introduce the following notation.
	Consider \(\alpha_1,\dots,\alpha_p\in\R\) then
	\begin{equation}
		\label{cicle}
	\begin{aligned}
			\alpha_{i+pn}&\coloneq\alpha_i+n,& i&\in\{1,\ldots,r\},& n&\in\mathbb N_0
	\end{aligned}
	\end{equation}
For example
\[
\begin{aligned}
	\alpha_{p+1}&\coloneq\alpha_1+1, &\alpha_{p+2}&\coloneq\alpha_{2}+1,&\dots,&& \alpha_{2p}&\coloneq\alpha_p+1, &\alpha_{2p+1}&\coloneq\alpha_1+2,\dots
\end{aligned}
\]
This cyclic notation is  useful to get an compact expression for the recurrence coefficients in the step-line 
\begin{multline*}
	b_{pm+k}^j=-(\alpha_{k+1}+m+1)\delta_{j,0}
	+
	\dfrac{(N-pm-k+1)_{j}(\beta+pm+k+1-j)_{j}({\alpha}_{k+1}+\beta+(p+1)m+k+1-j)}{{\prod}_{l=p+k+2-j}^{p+k+1}({\alpha}_l+\beta+(p+1)m+k-j)}\\ \times
\dfrac{\prod_{l=1}^p({\alpha}_l+\beta+pm+k+1-j)_{j}}{\prod_{l=k+1}^{p+k}(\alpha_l+\beta+(p+1)m+k+1-j)_{j}}\sum_{i=k+1}^{p+k+1-j}
	\dfrac{(\alpha_i+m)(\alpha_i+\beta+N+m+1)}{(\alpha_i+\beta+(p+1)m+k-j)_{j+2}}
	\tfrac{\prod_{l=1}^p(\alpha_i-\alpha_l+m)}{\prod_{l=k+1,l\neq i}^{p+k+1-j}(\alpha_i-{\alpha}_l)}
\end{multline*}
for \(j\in\{0,1,\dots,p\}\), \(m\geqslant0\) and \(k\in\{0,\dots,p-1\}\). This appears for the first time in the PhD thesis \cite{tesis}. The symbol ${\prod}_{l=p+k+2-j}^{p+k+1}$ for $j\geq 1$ is  the standard product and for   $j=0$ is $1$.

\subsection{Two weights}

Let us additionally assume that $p=2$.  Therefore our matrices have one superdiagonal and two subdiagonals, and we are looking for a factorization of the form $\mathscr T_{N-2}=L_1L_2U_1$. 

\begin{teo}\label{th:Hahn_now}
	The  Hahn multiple orthogonal polynomials with two weights we have the 
	factorization
	\[
\mathscr T_{N-2}=L_1L_2 U_1
	\] 
	in terms of two lower bidiagonal matrices $L_1$ and $L_2$ and an upper unitriangular matrix $U_1$. The corresponding entries of these matrices are
	\[
	\begin{aligned}
		U_{1,2n}&=\tfrac{(N-2n)(\alpha_1+n+1)(\alpha_1+\beta+2n+1)(\alpha_2+\beta+2n+1)}{
			(\alpha_1+\beta+3n+1)_2(\alpha_2+\beta+3n+1)},\\
		U_{1,2n+1}&=\tfrac{(N-2n-1)(\alpha_2+n+1)(\alpha_1+\beta+2n+2)(\alpha_2+\beta+2n+2)}{(\alpha_1+\beta+3n+3)(\alpha_2+\beta+3n+2)_2},\\
		L_{1,2n}
		&=\tfrac{(N+1-2n)n(\beta+2n)(\alpha_2+\beta+2n)}{(\alpha_1+\beta+3n)_2(\alpha_2+\beta+3n)}
		\tfrac{
			\pFq{3}{2}{-n+1,\alpha_1+\beta+2n,\alpha_1-\alpha_2-n+1}{\alpha_1-\alpha_2+1,\alpha_1+\beta+N+2}{1}}{
			\pFq{3}{2}{-n,\alpha_1+\beta+2n+1,\alpha_1-\alpha_2-n+1}{\alpha_1-\alpha_2+1,\alpha_1+\beta+N+2}{1}}
		,\\
		L_{1,2n+1}&=\tfrac{(N-2n)(\beta+2n+1)(\alpha_2-\alpha_1+n)(\alpha_2+\beta+2n+1)}{(\alpha_1+\beta+3n+2)(\alpha_2+\beta+3n+1)_2}
		\tfrac{\pFq{3}{2}{-n,\alpha_1+\beta+2n+1,\alpha_1-\alpha_2-n+1}{\alpha_1-\alpha_2+1,\alpha_1+\beta+N+2}{1}}{\pFq{3}{2}{-n,\alpha_1+\beta+2n+2,\alpha_1-\alpha_2-n}{\alpha_1-\alpha_2+1,\alpha_1+\beta+N+2}{1}},\\
		L_{2,2n}&=\tfrac{n(\beta+2n)(\alpha_1+\beta+2n+1)(\alpha_2+\beta+N+n+1)}{(\alpha_1+\beta+3n+1)(\alpha_2+\beta+3n)_2}	\tfrac{\pFq{3}{2}{-n,\alpha_1+\beta+2n+2,\alpha_1-\alpha_2-n}{\alpha_1-\alpha_2+1,\alpha_1+\beta+N+2}{1}}{\pFq{3}{2}{-n,\alpha_1+\beta+2n+1,\alpha_1-\alpha_2-n+1}{\alpha_1-\alpha_2+1,\alpha_1+\beta+N+2}{1}},\\
		L_{2,2n+1}&=
			\tfrac{(\beta+2n+1)(\alpha_1+\beta+2n+2)(\alpha_1+\beta+N+n+2)(\alpha_1-\alpha_2+n+1)}{(\alpha_1+\beta+3n+2)_2(\alpha_2+\beta+3n+2)} 
		\tfrac{
			\pFq{3}{2}{-n-1,\alpha_1+\beta+2n+3,\alpha_1-\alpha_2-n}{\alpha_1-\alpha_2+1,\alpha_1+\beta+N+2}{1}}{
			\pFq{3}{2}{-n,\alpha_1+\beta+2n+2,\alpha_1-\alpha_2-n}{\alpha_1-\alpha_2+1,\alpha_1+\beta+N+2}{1}}.
		\end{aligned}\]
\end{teo}
\begin{proof}
	We first note that
\[\begin{aligned}
	B_{2m}&= B_{(m,m)}, & 	B_{2m+1}&= B_{(m+1,m)}&
	A^{(1)}_{2m-1}&= A^{(1)}_{(m,m)},&A^{(1)}_{2m}&= A^{(1)}_{(m+1,m)},&
	A^{(2)}_{2m-1}&= A^{(a)}_{(m,m)},&A^{(2)}_{2m}&= A^{(2)}_{(m+1,m)}.
\end{aligned}\]
From \eqref{HahnTypeIIWeighted0} we get
\[
\begin{aligned}
		B_{2n}(0)
	&=
\tfrac{N!(\alpha_1+1)_{n}(\alpha_2+1)_{n}}{(N-2n)!(\alpha_1+\beta+2n+1)_{n}(\alpha_2+\beta+2n+1)_{n}},&
		B_{2n+1}(0)
	&	=
-\tfrac{N!(\alpha_1+1)_{n+1}(\alpha_2+1)_{n}}{(N-2n-1)!(\alpha_1+\beta+2n+2)_{n+1}(\alpha_2+\beta+2n+2)_{n}}
\end{aligned}
\]
and
\begin{align*}
	A_{2n-1}^{(1)}(0)
&	=\tfrac{-(N+1-2n)!{\alpha}_1+\beta+2n)_{n}({\alpha}_2+\beta+2n)_{n}}{(n-1)!(\beta+1)_{2n-1}(\alpha_1+\beta+2n)_{N+2-2n}({\alpha}_2-{\alpha}_1)_{{n}}}
\scriptstyle\pFq{3}{2}{-n+1,\alpha_1+\beta+2n,\alpha_1-\alpha_2-n+1}{\alpha_1-\alpha_2+1,\alpha_1+\beta+N+2}{1},\\
	A_{2n-1}^{(2)}(0)
	&=\tfrac{-(N+1-2n)!{\alpha}_1+\beta+2n)_{n}({\alpha}_2+\beta+2n)_{n}}{(n-1)!(\beta+1)_{2n-1}(\alpha_2+\beta+2n)_{N+2-2n}({\alpha}_1-{\alpha}_2)_{{n}}}\scriptstyle\pFq{3}{2}{-n+1,\alpha_2+\beta+2n,\alpha_2-\alpha_1-n+1}{\alpha_2-\alpha_1+1,\alpha_2+\beta+N+2}{1},\\
	A_{2n+1}^{(1)}(0)
	&=\tfrac{-(N-1-2n)!({\alpha}_1+\beta+2n+2)_{n+1}({\alpha}_2+\beta+2n+2)_{n+1}}{(n)!(\beta+1)_{2n+1}(\alpha_1+\beta+2n+2)_{N-2n}	({\alpha}_2-{\alpha}_1)_{{n+1}}}
\scriptstyle\pFq{3}{2}{-n,\alpha_1+\beta+2n+2,\alpha_1-\alpha_2-n}{\alpha_1-\alpha_2+1,\alpha_1+\beta+N+2}{1},\\
	A_{2n+1}^{(2)}(0)
	&=\tfrac{-(N-1-2n)!({\alpha}_1+\beta+2n+2)_{n+1}({\alpha}_2+\beta+2n+2)_{n+1}}{(n)!(\beta+1)_{2n+1}(\alpha_2+\beta+2n+2)_{N-2n}({\alpha}_1-{\alpha}_2)_{{n+1}}}\scriptstyle\pFq{3}{2}{-n,\alpha_2+\beta+2n+2,\alpha_2-\alpha_1-n}{\alpha_2-\alpha_1+1,\alpha_2+\beta+N+2}{1},\\
	A_{2n}^{(1)}(0)
	&=\tfrac{(N-2n)!({\alpha}_1+\beta+2n+1)_{n+1}({\alpha}_2+\beta+2n+1)_{n}}{n!(\beta+1)_{2n}(\alpha_1+\beta+2n+1)_{N+1-2n}({\alpha}_2-{\alpha}_1)_{{n}}}
\scriptstyle\pFq{3}{2}{-n,\alpha_1+\beta+2n+1,\alpha_1-\alpha_2-n+1}{\alpha_1-\alpha_2+1,\alpha_1+\beta+N+2}{1},\\
	A_{2n}^{(2)}(0)
	&=\tfrac{(N-2n)!({\alpha}_1+\beta+2n+1)_{n+1}({\alpha}_2+\beta+2n+1)_{n}}{(n-1)!(\beta+1)_{2n}(\alpha_2+\beta+2n+1)_{N+1-2n}({\alpha}_1-{\alpha}_2)_{n+1}}
	\scriptstyle\pFq{3}{2}{-n+1,\alpha_2+\beta+2n+1,\alpha_2-\alpha_1-n}{
		\alpha_2-\alpha_1+1,\alpha_2+\beta+N+2}{1},\\
	A_{2n+2}^{(1)}(0)
	&=\tfrac{(N-2n-2)!({\alpha}_1+\beta+2n+3)_{n+2}({\alpha}_2+\beta+2n+3)_{n+1}}{(n+1)!(\beta+1)_{2n+2}(\alpha_1+\beta+2n+3)_{N-1-2n}({\alpha}_2-{\alpha}_1)_{{n+1}}}
	\scriptstyle\pFq{3}{2}{-n-1,\alpha_1+\beta+2n+3,\alpha_1-\alpha_2-n}{\alpha_1-\alpha_2+1,\alpha_1+\beta+N+2}{1},\\
	A_{2n+2}^{(2)}(0)
	&=\tfrac{(N-2n-2)!({\alpha}_1+\beta+2n+3)_{n+2}({\alpha}_2+\beta+2n+3)_{n+1}}{n!(\beta+1)_{2n+2}(\alpha_2+\beta+2n+3)_{N-1-2n}({\alpha}_1-{\alpha}_2)_{n+2}}
\scriptstyle\pFq{3}{2}{-n,\alpha_2+\beta+2n+3,\alpha_2-\alpha_1-n-1}{
		\alpha_2-\alpha_1+1,\alpha_2+\beta+N+2}{1}.
\end{align*}

For for \(j\in\{0,1,2\}\), the recurrence coefficients are
\[\begin{aligned}
	b_{2m}^j&=\begin{multlined}[t][.8\textwidth]-(\alpha_{1}+m+1)\delta_{j,0}
	+
	\dfrac{(N-2m+1)_{j}(\beta+2m+1-j)_{j}({\alpha}_{1}+\beta+3m+1-j)}{{\prod}_{l=4-j}^{3}({\alpha}_l+\beta+3m-j)}\\ \times
	\dfrac{\prod_{l=1}^2({\alpha}_l+\beta+2m+1-j)_{j}}{\prod_{l=1}^{2}(\alpha_l+\beta+3m+1-j)_{j}}\sum_{i=1}^{3-j}
	\dfrac{(\alpha_i+m)(\alpha_i+\beta+N+m+1)}{(\alpha_i+\beta+3m-j)_{j+2}}
	\dfrac{\prod_{l=1}^2(\alpha_i-\alpha_l+m)}{\prod_{l=1,l\neq i}^{3-j}(\alpha_i-{\alpha}_l)}
\end{multlined}\\
	b_{2m+1}^j&=\begin{multlined}[t][.8\textwidth]-(\alpha_2+m+1)\delta_{j,0}
	+
	\dfrac{(N-2m)_{j}(\beta+2m+2-j)_{j}({\alpha}_2+\beta+3m+2-j)}{{\prod}_{l=5-j}^{4}({\alpha}_l+\beta+3m+{1}-j)}\\ \times
	\dfrac{\prod_{l=1}^2({\alpha}_l+\beta+2m+2-j)_{j}}{\prod_{l=2}^{3}(\alpha_l+\beta+3m+2-j)_{j}}\sum_{i=2}^{4-j}
	\dfrac{(\alpha_i+m)(\alpha_i+\beta+N+m+1)}{(\alpha_i+\beta+3m+1-j)_{j+2}}
	\dfrac{\prod_{l=1}^2(\alpha_i-\alpha_l+m)}{\prod_{l=2,l\neq i}^{4-j}(\alpha_i-{\alpha}_l)}.
\end{multlined}
\end{aligned}\]

The corresponding $\tau$-determinants are given by
\[
\begin{aligned}
	\tau^B_{1,n}&=B_n(0), & n&\in\{0,1,\dots, N-2\},&
		\tau^A_{1,n}&=A^{(1)}_n(0),& n&\in\{0,1,\dots, N-3\},
\end{aligned}
\]
and 	\[	\begin{aligned}
	\tau^A_{2,n}&=\begin{vNiceArray}{cc}[cell-space-limits=2pt]
	A^{(1)}_{n+1}(0) &A^{(1)}_n(0)\\
	A^{(2)}_{n+1}(0) &A^{(2)}_n(0)
\end{vNiceArray}, & n&\in\{0,1,\dots, N-3\}, &	(-1)^{(n+1)}b^2_2 \cdots b^2_{n+1} 	\tau^A_{2,n} &=B_n(0).
\end{aligned}
\]

 Thus, following Theorem \ref{th:bidiagonal_Christoffel}, for $n\in\{0,1,\dots,N-2\}$, we get
 	\[ \begin{aligned}
 		U_{1,n}&=-\frac{B_{n+1}(0)}{B_{n}(0)},   &
 		L_{1,n+1}&=-\frac{A^{(1)}_{n}(0)}{A^{(1)}_{n+1}(0)},
 			\end{aligned}\]
 			and 
 				\[ \begin{aligned}
 				L_{2,n+1}&=-\frac{A^{(1)}_{n+2}(0)}{A^{(1)}_{n+1}(0)}\frac{\begin{vNiceArray}{cc}[cell-space-limits=2pt]
 						A^{(1)}_{n+1}(0) &A^{(1)}_n(0)\\
 						A^{(2)}_{n+1}(0) &A^{(2)}_n(0)
 				\end{vNiceArray}}{\begin{vNiceArray}{cc}[cell-space-limits=2pt]
 						A^{(1)}_{n+2}(0) &A^{(1)}_{n+1}(0)\\
 						A^{(2)}_{n+2}(0) &A^{(2)}_{n+1}(0)
 				\end{vNiceArray}}=b^2_{n+2}\frac{A^{(1)}_{n+2}(0)}{A^{(1)}_{n+1}(0)}\frac{B_{n}(0)}{B_{n+1}(0)},  &n\in\{0,1,\dots,N-3\}.
 			\end{aligned}\]

Therefore,
\[
\begin{aligned}
	U_{1,2n}&=-\frac{-\frac{N!}{(N-2n-1)!}\frac{(\alpha_1+1)_{n+1}(\alpha_2+1)_{n}}{(\alpha_1+\beta+2n+2)_{n+1}(\alpha_2+\beta+2n+2)_{n}}}{\frac{N!}{(N-2n)!}\frac{(\alpha_1+1)_{n}(\alpha_2+1)_{n}}{(\alpha_1+\beta+2n+1)_{n}(\alpha_2+\beta+2n+1)_{n}}}=\tfrac{(N-2n)(\alpha_1+n+1)(\alpha_1+\beta+2n+1)(\alpha_2+\beta+2n+1)}{(\alpha_1+\beta+3n+1)(\alpha_1+\beta+3n+2)(\alpha_2+\beta+3n+1)},\\
		U_{1,2n+1}&=-\frac{\frac{N!}{(N-2n-2)!}\frac{(\alpha_1+1)_{n+1}(\alpha_2+1)_{n+1}}{(\alpha_1+\beta+2n+3)_{n+1}(\alpha_2+\beta+2n+3)_{n+1}}}{-\frac{N!}{(N-2n-1)!}\frac{(\alpha_1+1)_{n+1}(\alpha_2+1)_{n}}{(\alpha_1+\beta+2n+2)_{n+1}(\alpha_2+\beta+2n+2)_{n}} }
		=\tfrac{(N-2n-1)(\alpha_2+n+1)(\alpha_1+\beta+2n+2)(\alpha_2+\beta+2n+2)}{(\alpha_1+\beta+3n+3)(\alpha_2+\beta+3n+3)(\alpha_2+\beta+3n+2)},
\end{aligned}\]
and also, for the lower bidiagonal we find
\[
\begin{aligned}
L_{1,2n}&=-\scriptstyle\frac{\frac{-(N+1-2n)!}{(n-1)!(\beta+1)_{2n-1}(\alpha_1+\beta+2n)_{N+2-2n}}
	\frac{
		({\alpha}_1+\beta+2n)_{n}({\alpha}_2+\beta+2n)_{n}}{
		({\alpha}_2-{\alpha}_1)_{{n}}}
	\pFq{3}{2}{-n+1,\alpha_1+\beta+2n,\alpha_1-\alpha_2-n+1}{\alpha_1-\alpha_2+1,\alpha_1+\beta+N+2}{1}}{\frac{(N-2n)!}{n!(\beta+1)_{2n}(\alpha_1+\beta+2n+1)_{N+1-2n}}
	\frac{({\alpha}_1+\beta+2n+1)_{n+1}({\alpha}_2+\beta+2n+1)_{n}}{
		({\alpha}_2-{\alpha}_1)_{{n}}}
	\pFq{3}{2}{-n,\alpha_1+\beta+2n+1,\alpha_1-\alpha_2-n+1}{\alpha_1-\alpha_2+1,\alpha_1+\beta+N+2}{1}}\\
	&=\scriptstyle\frac{(N+1-2n)n(\beta+2n)(\alpha_2+\beta+2n)}{(\alpha_1+\beta+3n)(\alpha_1+\beta+3n+1)(\alpha_2+\beta+3n)}
		\frac{
		\pFq{3}{2}{-n+1,\alpha_1+\beta+2n,\alpha_1-\alpha_2-n+1}{\alpha_1-\alpha_2+1,\alpha_1+\beta+N+2}{1}}{
		\pFq{3}{2}{-n,\alpha_1+\beta+2n+1,\alpha_1-\alpha_2-n+1}{\alpha_1-\alpha_2+1,\alpha_1+\beta+N+2}{1}}
	,\\
	L_{1,2n+1}&=\scriptstyle-\frac{\frac{(N-2n)!}{n!(\beta+1)_{2n}(\alpha_1+\beta+2n+1)_{N+1-2n}}
		\frac{({\alpha}_1+\beta+2n+1)_{n+1}({\alpha}_2+\beta+2n+1)_{n}}{
			({\alpha}_2-{\alpha}_1)_{{n}}}
		\pFq{3}{2}{-n,\alpha_1+\beta+2n+1,\alpha_1-\alpha_2-n+1}{\alpha_1-\alpha_2+1,\alpha_1+\beta+N+2}{1}}{\frac{-(N-2n-1)!}{n!(\beta+1)_{2n+1}(\alpha_1+\beta+2n+2)_{N-2n}}
		\frac{
			({\alpha}_1+\beta+2n+2)_{n+1}({\alpha}_2+\beta+2n+2)_{n+1}}{
			({\alpha}_2-{\alpha}_1)_{{n+1}}}\pFq{3}{2}{-n,\alpha_1+\beta+2n+2,\alpha_1-\alpha_2-n}{\alpha_1-\alpha_2+1,\alpha_1+\beta+N+2}{1}}\\&=\scriptstyle\frac{(N-2n)(\beta+2n+1)(\alpha_2-\alpha_1+n)(\alpha_2+\beta+2n+1)}{(\alpha_1+\beta+3n+2)(\alpha_2+\beta+3n+1)(\alpha_2+\beta+3n+2)}
	\frac{\pFq{3}{2}{-n,\alpha_1+\beta+2n+1,\alpha_1-\alpha_2-n+1}{\alpha_1-\alpha_2+1,\alpha_1+\beta+N+2}{1}}{\pFq{3}{2}{-n,\alpha_1+\beta+2n+2,\alpha_1-\alpha_2-n}{\alpha_1-\alpha_2+1,\alpha_1+\beta+N+2}{1}}
\end{aligned}
\]

%
%
%
%
Let us continue by  considering
$L_{2,2n}=\frac{b^2_{2n+1}}{L_{1,2n+1}U_{1,2n-1}}$.
Substituting
\begin{multline*}
b_{2n+1}^2=
	\dfrac{(N-2n)_{2}(\beta+2n)_{2}({\alpha}_2+\beta+3n)}{\prod_{l=3}^{4}({\alpha}_l+
	\beta+3n-1)}
	\dfrac{\prod_{l=1}^2({\alpha}_l+\beta+2n)_{2}(\alpha_2-\alpha_l+n)}{\prod_{l=2}^{3}(\alpha_l+\beta+3n)_{2}}\\
\times 	\dfrac{(\alpha_2+n)(\alpha_2+\beta+N+n+1)}{(\alpha_2+\beta+3n-1)_{4}}.
\end{multline*}
where \( \alpha_3 = \alpha_1+1\) and  \( \alpha_4 = \alpha_2+1\) and 
\[
\begin{aligned}
	L_{1,2n+1}&=\scriptstyle\frac{(N-2n)(\beta+2n+1)(\alpha_2-\alpha_1+n)(\alpha_2+\beta+2n+1)}{(\alpha_1+\beta+3n+2)(\alpha_2+\beta+3n+1)(\alpha_2+\beta+3n+2)}
	\frac{\pFq{3}{2}{-n,\alpha_1+\beta+2n+1,\alpha_1-\alpha_2-n+1}{\alpha_1-\alpha_2+1,\alpha_1+\beta+N+2}{1}}{\pFq{3}{2}{-n,\alpha_1+\beta+2n+2,\alpha_1-\alpha_2-n}{\alpha_1-\alpha_2+1,\alpha_1+\beta+N+2}{1}},\\
U_{1,2n-1}&=\scriptstyle\frac{(N-2n+1)(\alpha_2+n)(\alpha_1+\beta+2n)(\alpha_2+\beta+2n)}{(\alpha_1+\beta+3n)(\alpha_2+\beta+3n)(\alpha_2+\beta+3n-1)}
\end{aligned}
\]
we arrive to the identity
\begin{align*}
L_{2,2n}=\scriptstyle\frac{n(\beta+2n)(\alpha_1+\beta+2n+1)(\alpha_2+\beta+N+n+1)}{(\alpha_1+\beta+3n+1)(\alpha_2+\beta+3n)_2}	\frac{\pFq{3}{2}{-n,\alpha_1+\beta+2n+2,\alpha_1-\alpha_2-n}{\alpha_1-\alpha_2+1,\alpha_1+\beta+N+2}{1}}{\pFq{3}{2}{-n,\alpha_1+\beta+2n+1,\alpha_1-\alpha_2-n+1}{\alpha_1-\alpha_2+1,\alpha_1+\beta+N+2}{1}}.
\end{align*}

Let us finally  discuss 
$L_{2,2n+1}=\frac{b^2_{2n+2}}{L_{1,2n+2}U_{1,2n}}$.
With  the substitutions
\[\begin{aligned}
	b_{2n+2}^2&=\begin{multlined}[t][.8\textwidth]
\scriptstyle	\frac{(N-2n-1)_{2}(\beta+2n+1)_{2}({\alpha}_{1}+\beta+3n+2)}{\prod_{l=2}^{3}({\alpha}_l+\beta+3n+1)}
	\frac{\prod_{l=1}^2({\alpha}_l+\beta+2n+1)_{2}}{\prod_{l=1}^{2}(\alpha_l+\beta+3n+2)_{2}}
	\frac{(\alpha_1+n+1)(\alpha_1+\beta+N+n+2)}{(\alpha_1+\beta+3n+1)_{4}}
	\frac{\prod_{l=1}^2(\alpha_1-\alpha_l+n+1)}{1},
\end{multlined}\\
L_{1,2n+2}&=\scriptstyle\frac{(N-1-2n)(n+1)(\beta+2n+2)(\alpha_2+\beta+2n+2)}{(\alpha_1+\beta+3n+3)
	(\alpha_1+\beta+3n+4)(\alpha_2+\beta+3n+3)}
\frac{
	\pFq{3}{2}{-n,\alpha_1+\beta+2n+2,\alpha_1-\alpha_2-n}{\alpha_1-\alpha_2+1,\alpha_1+\beta+N+2}{1}}{
	\pFq{3}{2}{-n-1,\alpha_1+\beta+2n+3,\alpha_1-\alpha_2-n}{\alpha_1-\alpha_2+1,\alpha_1+\beta+N+2}{1}},\\
	U_{1,2n}&=\scriptstyle\frac{(N-2n)(\alpha_1+n+1)(\alpha_1+\beta+2n+1)(\alpha_2+\beta+2n+1)}{(\alpha_1+\beta+3n+1)(\alpha_1+\beta+3n+2)(\alpha_2+\beta+3n+1)},
\end{aligned}\]


We obtain that 
\begin{align*}
L_{2,2n+1}=\scriptstyle\frac{(\beta+2n+1)(\alpha_1+\beta+2n+2)(\alpha_1+\beta+N+n+2)(\alpha_1-\alpha_2+n+1)}{(\alpha_1+\beta+3n+2)_2(\alpha_2+\beta+3n+2)}  \frac{
		\pFq{3}{2}{-n-1,\alpha_1+\beta+2n+3,\alpha_1-\alpha_2-n}{\alpha_1-\alpha_2+1,\alpha_1+\beta+N+2}{1}}{
		\pFq{3}{2}{-n,\alpha_1+\beta+2n+2,\alpha_1-\alpha_2-n}{\alpha_1-\alpha_2+1,\alpha_1+\beta+N+2}{1}}.
\end{align*}
\end{proof}

\begin{lemma}\label{lemma:hypergeom}
	The following hypergeometrical identities hold: 
\begin{align*}
\scriptstyle\frac{(-1)^{n}(2n-2)!(\alpha_2+\beta+n+1)_{2n-1}}{(n-1)! (\alpha_1-\alpha_2-n+1)_{2n-1} }\, \pFq{3}{2}{-n+1,-N,\alpha_2-\alpha_1-n+1}{-2n+2,\alpha_2+\beta+n+1}{1}&=\scriptstyle
	\frac{
		({\alpha}_2+\beta+2n)_{n}}{(\alpha_1+\beta+N+1)
		({\alpha}_2-{\alpha}_1)_{{n}}}\,\pFq{3}{2}{-n+1,\alpha_1+\beta+2n,\alpha_1-\alpha_2-n+1}{\alpha_1-\alpha_2+1,\alpha_1+\beta+N+2}{1},
\\
\scriptstyle	\frac{(-1)^n(2n-1)!(\alpha_2+\beta+n+1)_{2n}}{(n-1)! (\alpha_1-\alpha_2-n+1)_{2n} } \pFq{3}{2}{-n,-N,\alpha_2-\alpha_1-n}{-2n+1,\alpha_2+\beta+n+1}{1}&=\scriptstyle
	\frac{({\alpha}_2+\beta+2n+1)_{n}}{
	(\alpha_1+\beta+N+1)	({\alpha}_2-{\alpha}_1)_{{n}}}\pFq{3}{2}{-n,\alpha_1+\beta+2n+1,\alpha_1-\alpha_2-n+1}{\alpha_1-\alpha_2+1,\alpha_1+\beta+N+2}{1}.
\end{align*}
\end{lemma}

\begin{proof}
	Using our previous results in \cite{JMAA} we obtain the alternative expressions for the type I Hahn  multiple orthogonal polynomials at $x=0$ 
	\[\begin{aligned}
		A^{(1)}_{2n-1}(0)&=\scriptstyle
			\frac{(-1)^{n-1}(N-2n+1)!(2n-2)!(\alpha_2+\beta+n+1)_{2n-1}}{(n-1)!(n-1)!(\beta+1)_{2n-1} (\alpha_1+\beta+3n)_{N-2n+1}(\alpha_1-\alpha_2-n+1)_{2n-1} }  \pFq{3}{2}{-n+1,-N,\alpha_2-\alpha_1-n+1}{-2n+2,\alpha_2+\beta+n+1}{1},\\
		A^{(1)}_{2n}(0)&=\scriptstyle
			\frac{(-1)^n(N-2n)!(2n-1)!(\alpha_2+\beta+n+1)_{2n}}{n!(n-1)!(\beta+1)_{2n} (\alpha_1+\beta+3n+2)_{N-2n}(\alpha_1-\alpha_2-n+1)_{2n} } \pFq{3}{2}{-n,-N,\alpha_2-\alpha_1-n}{-2n+1,\alpha_2+\beta+n+1}{1}
	\end{aligned}\]
	Comparing these expressions with the one used previously, deduced from \cite{BDFMW}, we get the result.
\end{proof}

In our \cite{LAA}, with the notation 
\[
\begin{aligned}
	a_{6n+1}&=U_{1,2n}, &a_{6n+4}&=U_{1,2n+1},&a_{6n+2}&=L_{1,2n},  &a_{6n+5}&=L_{1,2n+1}, &a_{6n+3}&=L_{2,2n}, &a_{6n+6}&=L_{1,2n+1}, 
\end{aligned}
\]
we proved that:
\begin{teo}\label{th:Hahn_before}
	The  bidiagonal factorization of the recursion matrix for the Hahn multiple orthogonal polynomials with respect two weights is:
	\begin{align*}
		\begin{aligned}
			a_{6n+1}
			&=\scriptstyle
			\frac{(N-2n)(\alpha_1+1+n)(\alpha_1+\beta+2n+1)(\alpha_2+\beta+2n+1)}{(\alpha_1+\beta+3n+1)_2(\alpha_2+\beta+3n+1)}\\
			a_{6n+4}&=\scriptstyle\frac{(N-2n-1)(\alpha_2+1+n)(\alpha_1+\beta+2n+2)(\alpha_2+\beta+2n+2)}{(\alpha_1+\beta+3n+3)(\alpha_2+\beta+3n+2)_2},\\
			a_{6n+2}&=\scriptstyle\frac{(N-2n)(n)_n(\beta+2n+1)(\alpha_2-\alpha_1+n)(\alpha_2+\beta+n+1)}{(n+1)_n(\alpha_1+\beta+3n+2)(\alpha_2+\beta+3n+1)_2}
				\frac{\pFq{3}{2}{-n,-N,\alpha_2-\alpha_1-n}{-2n+1,\alpha_2+\beta+n+1}{1}}{\pFq{3}{2}{-n,-N,\alpha_2-\alpha_1-n}{-2n,\alpha_2+\beta+n+2}{1}},\\
			a_{6n+5}&=\scriptstyle\frac{(n+1)(N-2n-1)(\beta+2n+2)(\alpha_1-\alpha_2+n+1)(\alpha_1+\beta+2+n+N)}{(2n+1)(\alpha_1+\beta+3n+3)_2(\alpha_2+\beta+3n+3)}				\frac{\pFq{3}{2}{-n,-N,\alpha_2-\alpha_1-n}{-2n,\alpha_2+\beta+n+2}{1}}{\pFq{3}{2}{-n-1,-N,\alpha_2-\alpha_1-n-1}{-2n-1,\alpha_2+\beta+n+2}{1}},\\
			a_{6n+3}&=\scriptstyle\frac{(2n+1)(\beta+2n+1)(\alpha_1+\beta+2n+2)(\alpha_2+\beta+2n+2)}{(\alpha_1+\beta+3n+2)_2(\alpha_2+\beta+3n+2)}
				\frac{\pFq{3}{2}{-n-1,-N,\alpha_2-\alpha_1-n-1}{-2n-1,\alpha_2+\beta+n+2}{1}}{\pFq{3}{2}{-n,-N,\alpha_2-\alpha_1-n}{-2n,\alpha_2+\beta+n+2}{1}},\\
			a_{6n+6}&=\scriptstyle\frac{2(n+1)(\beta+2n+2)(\alpha_1+\beta+2n+3)(\alpha_2+\beta+2n+3)(\alpha_2+\beta+2+n+N)}{(\alpha_1+\beta+3n+4)(\alpha_2+\beta+3n+3)_2(\alpha_2+\beta+n+2)}
				\frac{\pFq{3}{2}{-n-1,-N,\alpha_2-\alpha_1-n-1}{-2n-2,\alpha_2+\beta+n+3}{1}}{\pFq{3}{2}{-n-1,-N,\alpha_2-\alpha_1-n-1}{-2n-1,\alpha_2+\beta+n+2}{1}}.
		\end{aligned}
	\end{align*}
\end{teo}

\begin{teo}
The bidiagonal factorization in Theorems \ref{th:Hahn_now} and \ref{th:Hahn_before} coincide.
\end{teo}
\begin{proof}
	Use Lemma \ref{lemma:hypergeom}.
\end{proof}

\subsection{Three weights}
Let us now assume that $p=3$.  Therefore our matrices have one superdiagonal and two subdiagonals, and we are looking for a factorization of the form $\mathscr T_{N-3}=L_1L_2L_3U_1$. 

We first note that
\[\begin{aligned}
	B_{3m}&= B_{(m,m,m)}, & 	B_{3m+1}&= B_{(m+1,m,m)}, & 	B_{3m+2}&= B_{(m+1,m+1,m)},\\
	A^{(1)}_{3m-1}&= A^{(1)}_{(m,m,m)},&A^{(1)}_{3m}&= A^{(1)}_{(m+1,m,m)},&A^{(1)}_{3m+1}&= A^{(1)}_{(m+1,m+1,m)},\\
	A^{(2)}_{3m-1}&= A^{(2)}_{(m,m,m)},&A^{(2)}_{3m}&= A^{(2)}_{(m+1,m,m)},&A^{(2)}_{3m+1}&= A^{(2)}_{(m+1,m+1,m)},\\
		A^{(3)}_{3m-1}&= A^{(3)}_{(m,m,m)},&A^{(3)}_{3m}&= A^{(3)}_{(m+1,m,m)},&A^{(3)}_{3m+1}&= A^{(3)}_{(m+1,m+1,m)},\\
\end{aligned}\]

\begin{teo}\label{th:Hahn_3}
	The  Hahn multiple orthogonal polynomials with two weights we have the 
	factorization
	\[
	\mathscr T_{N-3}=L_1L_2 L_3U_1
	\] 
	in terms of three lower bidiagonal matrices $L_1,L_2$ and $L_3$ and an upper unitriangular matrix $U_1$. The corresponding entries of these matrices are
\[\hspace{-1cm}
	\begin{aligned}
			U_{1,3n}&=\tfrac{(N-3n)(\alpha_{1}+n+1)(\alpha_{1}+\beta+3n+1)(\alpha_{2}+\beta+3n+1)(\alpha_{3}+\beta+3n+1)}
			{(\alpha_{1}+\beta+4n+1)_2(\alpha_{2}+\beta+4n+1)(\alpha_{3}+\beta+4n+1)},\\
			U_{1,3n+1}&=-\tfrac{(N-3n-1)(\alpha_{2}+n+1)\,(\alpha_{1}+\beta+3n+2)(\alpha_{2}+\beta+3n+2)(\alpha_{3}+\beta+3n+2)}
			{(\alpha_{1}+\beta+4n+3)(\alpha_{2}+\beta+4n+2)_2(\alpha_{3}+\beta+4n+2)},\\
			U_{1,3n+2}&=\tfrac{(N-3n-2)(\alpha_{3}+n+1)\,
				(\alpha_{1}+\beta+3n+3)(\alpha_{2}+\beta+3n+3)(\alpha_{3}+\beta+3n+3)}
			{(\alpha_{1}+\beta+4n+4)(\alpha_{2}+\beta+4n+4)
				(\alpha_{3}+\beta+4n+3)_2},\\
					L_{1,3n+1}&=\tfrac{(N-3n)\,(\beta+3n+1)\,(\alpha_2-\alpha_1+n)(\alpha_2+\beta+3n+1)\,(\alpha_3+\beta+3n+1)}
					{(\alpha_1+\beta+4n+2)\,(\alpha_2+\beta+4n+1)_2(\alpha_3+\beta+4n+1)}
					\tfrac{\pFq{4}{3}{-n,\alpha_1+\beta+3n+1,\alpha_1-\alpha_2-n+1,\alpha_1-\alpha_3-n+1}{
							\alpha_1-\alpha_2+1,		\alpha_1-\alpha_3+1,\alpha_1+\beta+N+2}{1}}{\pFq{4}{3}{-n,\alpha_1+\beta+3n+2,\alpha_1-\alpha_2-n,\alpha_1-\alpha_3-n+1}{
							\alpha_1-\alpha_2+1,		\alpha_1-\alpha_3+1,\alpha_1+\beta+N+2}{1}},\\
					L_{1,3n+2}&=\tfrac{(N-3n-1)\,(\beta+3n+2)\,(\alpha_3-\alpha_1+n)
						(\alpha_2+\beta+3n+2)(\alpha_3+\beta+3n+2)}
					{(\alpha_1+\beta+4n+3)(\alpha_2+\beta+4n+3)(\alpha_3+\beta+4n+2)_2}
					\tfrac{\pFq{4}{3}{-n,\alpha_1+\beta+3n+2,\alpha_1-\alpha_2-n,\alpha_1-\alpha_3-n+1}{
							\alpha_1-\alpha_2+1,		\alpha_1-\alpha_3+1,\alpha_1+\beta+N+2}{1}}{\pFq{4}{3}{-n,\alpha_1+\beta+3n+3,\alpha_1-\alpha_2-n,\alpha_1-\alpha_3-n}{
							\alpha_1-\alpha_2+1,		\alpha_1-\alpha_3+1,\alpha_1+\beta+N+2}{1}},\\
					L_{1,3n+3}&=\tfrac{(N-3n-2)(n+1)(\beta+3n+3)(\alpha_2+\beta+3n+3)(\alpha_3+\beta+3n+3)}
					{(\alpha_1+\beta+4n+4)_2(\alpha_2+\beta+4n+4)(\alpha_3+\beta+4n+4)}
					\tfrac{\pFq{4}{3}{-n,\alpha_1+\beta+3n+3,\alpha_1-\alpha_2-n,\alpha_1-\alpha_3-n}{
							\alpha_1-\alpha_2+1,		\alpha_1-\alpha_3+1,\alpha_1+\beta+N+2}{1}}{\pFq{4}{3}{-n-1,\alpha_1+\beta+3n+4,\alpha_1-\alpha_2-n,\alpha_1-\alpha_3-n}{
							\alpha_1-\alpha_2+1,		\alpha_1-\alpha_3+1,\alpha_1+\beta+N+2}{1}},\\
		L_{2, 3n+1}&=\begin{multlined}[t][\textwidth]
			-\tfrac{n(N-3n)_2(\beta+3n)_2(\alpha_1+\beta+3n)(\alpha_1+\beta+4n+3)(\alpha_2+\beta+3n)(\alpha_2+\beta+4n+3)(\alpha_3+\beta+3n)_2(\alpha_3+\beta+4n+2)_2 (\alpha_1-\alpha_2+n)}%
			{(N-3n-1)(\beta+3n+2)(\alpha_1+\beta+4n)_3(\alpha_2+\beta+3n+2)
				(\alpha_2+\beta+4n)_3(\alpha_3+\beta+3n+2)(\alpha_3+\beta+4n)_2(\alpha_3-\alpha_1+n)}\\
			\times 
			\tfrac{\pFq{4}{3}{-n,\alpha_1+\beta+3n+3,\alpha_1-\alpha_2-n,\alpha_1-\alpha_3-n}{
					\alpha_1-\alpha_2+1,		\alpha_1-\alpha_3+1,\alpha_1+\beta+N+2}{1}}{\pFq{4}{3}{-n,\alpha_1+\beta+3n+2,\alpha_1-\alpha_2-n,\alpha_1-\alpha_3-n+1}{
					\alpha_1-\alpha_2+1,		\alpha_1-\alpha_3+1,\alpha_1+\beta+N+2}{1}} \tfrac{\begin{vNiceArray}{cc}[cell-space-limits=2pt]
					\frac{ \pFq{4}{3}{-n,\alpha_1+\beta+3n+1,\alpha_1-\alpha_2-n+1,\alpha_1-\alpha_3-n+1}{
							\alpha_1-\alpha_2+1,		\alpha_1-\alpha_3+1,\alpha_1+\beta+N+2}{1}}{n}&\frac{\pFq{4}{3}{-n+1,\alpha_1+\beta+3n,\alpha_1-\alpha_2-n+1,\alpha_1-\alpha_3-n+1}{
							\alpha_1-\alpha_2+1,		\alpha_1-\alpha_3+1,\alpha_1+\beta+N+2}{1}}{\alpha_1+\beta+3n}\\
					\frac{\pFq{4}{3}{-n+1,\alpha_2+\beta+3n+1,\alpha_2-\alpha_1-n,\alpha_2-\alpha_3-n+1}{
							\alpha_2-\alpha_1+1,		\alpha_2-\alpha_3+1,\alpha_2+\beta+N+2}{1}}{\alpha_1-\alpha_2+n}&\frac{\pFq{4}{3}{-n+1,\alpha_2+\beta+3n,\alpha_2-\alpha_1-n+1,\alpha_2-\alpha_3-n+1}{
							\alpha_2-\alpha_1+1,		\alpha_2-\alpha_3+1,\alpha_2+\beta+N+2}{1}}{\alpha_2+\beta+3n}
				\end{vNiceArray}
			}{\begin{vNiceArray}{cc}[cell-space-limits=2pt,small]
					\frac{\pFq{4}{3}{-n,\alpha_1+\beta+3n+2,\alpha_1-\alpha_2-n,\alpha_1-\alpha_3-n+1}{
							\alpha_1-\alpha_2+1,		\alpha_1-\alpha_3+1,\alpha_1+\beta+N+2}{1}}{ \alpha_2-\alpha_1+n}&\frac{\pFq{4}{3}{-n,\alpha_1+\beta+3n+1,\alpha_1-\alpha_2-n+1,\alpha_1-\alpha_3-n+1}{
							\alpha_1-\alpha_2+1,		\alpha_1-\alpha_3+1,\alpha_1+\beta+N+2}{1}}{n(\alpha_1+\beta+3n+1)	}\\
					\pFq{4}{3}{-n,\alpha_2+\beta+3n+2,\alpha_2-\alpha_1-n,\alpha_2-\alpha_3-n+1}{
						\alpha_2-\alpha_1+1,		\alpha_2-\alpha_3+1,\alpha_2+\beta+N+2}{1}&\frac{\pFq{4}{3}{-n+1,\alpha_2+\beta+3n+1,\alpha_2-\alpha_1-n,\alpha_2-\alpha_3-n+1}{
							\alpha_2-\alpha_1+1,		\alpha_2-\alpha_3+1,\alpha_2+\beta+N+2}{1}}{\alpha_2+\beta+3n+1}
			\end{vNiceArray}},		\end{multlined}\\
		L_{2, 3n+2}&=\begin{multlined}[t][\textwidth]-
			\tfrac{n (N-3n-1)_2(\beta+3n+1)_2(\alpha_1+\beta+3n+1)(\alpha_1+\beta+4n+4)_2(\alpha_2+\beta+3n+1)(\alpha_2+\beta+4n+4)(\alpha_3+\beta+3n+1)_2(\alpha_3+\beta+4n+4)  (\alpha_2-\alpha_1+n)}%
			{(n+1)(N-3n-2)(\beta+3n+3)(\alpha_1+\beta+4n+2)_2(\alpha_2+\beta+3n+3)(\alpha_2+\beta+4n+1)_3(\alpha_3+\beta+3n+3)(\alpha_3+\beta+4n+1)_3}\\
			\times\tfrac{\pFq{4}{3}{-n-1,\alpha_1+\beta+3n+4,\alpha_1-\alpha_2-n,\alpha_1-\alpha_3-n}{
					\alpha_1-\alpha_2+1,		\alpha_1-\alpha_3+1,\alpha_1+\beta+N+2}{1}}{\pFq{4}{3}{-n,\alpha_1+\beta+3n+,\alpha_1-\alpha_2-n,\alpha_1-\alpha_3-n}{
					\alpha_1-\alpha_2+1,		\alpha_1-\alpha_3+1,\alpha_1+\beta+N+2}{1}}
			\tfrac{\begin{vNiceArray}{cc}[cell-space-limits=2pt,small]
					\frac{\pFq{4}{3}{-n,\alpha_1+\beta+3n+2,\alpha_1-\alpha_2-n,\alpha_1-\alpha_3-n+1}{
							\alpha_1-\alpha_2+1,		\alpha_1-\alpha_3+1,\alpha_1+\beta+N+2}{1}}{ \alpha_2-\alpha_1+n}&\frac{\pFq{4}{3}{-n,\alpha_1+\beta+3n+1,\alpha_1-\alpha_2-n+1,\alpha_1-\alpha_3-n+1}{
							\alpha_1-\alpha_2+1,		\alpha_1-\alpha_3+1,\alpha_1+\beta+N+2}{1}}{n(\alpha_1+\beta+3n+1)	}\\
					\pFq{4}{3}{-n,\alpha_2+\beta+3n+2,\alpha_2-\alpha_1-n,\alpha_2-\alpha_3-n+1}{
						\alpha_2-\alpha_1+1,		\alpha_2-\alpha_3+1,\alpha_2+\beta+N+2}{1}&\frac{\pFq{4}{3}{-n+1,\alpha_2+\beta+3n+1,\alpha_2-\alpha_1-n,\alpha_2-\alpha_3-n+1}{
							\alpha_2-\alpha_1+1,		\alpha_2-\alpha_3+1,\alpha_2+\beta+N+2}{1}}{\alpha_2+\beta+3n+1}
				\end{vNiceArray}
			}
			{\begin{vNiceArray}{cc}[cell-space-limits=2pt]
					\frac{ \pFq{4}{3}{-n,\alpha_1+\beta+3n+3,\alpha_1-\alpha_2-n,\alpha_1-\alpha_3-n}{
							\alpha_1-\alpha_2+1,		\alpha_1-\alpha_3+1,\alpha_1+\beta+N+2}{1}}{\alpha_3-\alpha_1+n}&\frac{\pFq{4}{3}{-n,\alpha_1+\beta+3n+2,\alpha_1-\alpha_2-n,\alpha_1-\alpha_3-n+1}{
							\alpha_1-\alpha_2+1,		\alpha_1-\alpha_3+1,\alpha_1+\beta+N+2}{1}}{(\alpha_1+\beta+3n+2)	}\\
					\frac{\pFq{4}{3}{-n,\alpha_2+\beta+3n+3,\alpha_2-\alpha_1-n,\alpha_2-\alpha_3-n}{
							\alpha_2-\alpha_1+1,		\alpha_2-\alpha_3+1,\alpha_2+\beta+N+2}{1}}{\alpha_3-\alpha_2+n}&\frac{\pFq{4}{3}{-n,\alpha_2+\beta+3n+2,\alpha_2-\alpha_1-n,\alpha_2-\alpha_3-n+1}{
							\alpha_2-\alpha_1+1,		\alpha_2-\alpha_3+1,\alpha_2+\beta+N+2}{1}}{\alpha_2+\beta+3n+2}
			\end{vNiceArray}},
				\end{multlined}\\
		L_{2, 3n+3}&
=\begin{multlined}[t][.6\textwidth]
			-\tfrac{(N-3n-2)_2 (\beta+3n+2)_2(\alpha_1+\beta+3n+2)(\alpha_1+\beta+4n+6)(\alpha_2+\beta+3n+2)(\alpha_2+\beta+4n+5)_2(\alpha_3+\beta+3n+2)_2(\alpha_3+\beta+4n+5) (\alpha_3-\alpha_1+n)(\alpha_3-\alpha_2+n)}%
			{(N-3n-3)(\beta+3n+4)(\alpha_1+\beta+4n+3)_3(\alpha_2+\beta+3n+4)
				(\alpha_2+\beta+4n+3)_2(\alpha_3+\beta+3n+4)(\alpha_3+\beta+4n+2)_3(\alpha_2-\alpha_1+n+1)}
			\\\times 
			\tfrac{\pFq{4}{3}{-n-1,\alpha_1+\beta+3n+5,\alpha_1-\alpha_2-n-1,\alpha_1-\alpha_3-n}{
					\alpha_1-\alpha_2+1,		\alpha_1-\alpha_3+1,\alpha_1+\beta+N+2}{1}}{\pFq{4}{3}{-n-1,\alpha_1+\beta+3n+4,\alpha_1-\alpha_2-n,\alpha_1-\alpha_3-n}{
					\alpha_1-\alpha_2+1,		\alpha_1-\alpha_3+1,\alpha_1+\beta+N+2}{1}}
			\tfrac{\begin{vNiceArray}{cc}[cell-space-limits=2pt]
					\frac{ \pFq{4}{3}{-n,\alpha_1+\beta+3n+3,\alpha_1-\alpha_2-n,\alpha_1-\alpha_3-n}{
							\alpha_1-\alpha_2+1,		\alpha_1-\alpha_3+1,\alpha_1+\beta+N+2}{1}}{\alpha_3-\alpha_1+n}&\frac{\pFq{4}{3}{-n,\alpha_1+\beta+3n+2,\alpha_1-\alpha_2-n,\alpha_1-\alpha_3-n+1}{
							\alpha_1-\alpha_2+1,		\alpha_1-\alpha_3+1,\alpha_1+\beta+N+2}{1}}{(\alpha_1+\beta+3n+2)	}\\
					\frac{\pFq{4}{3}{-n,\alpha_2+\beta+3n+3,\alpha_2-\alpha_1-n,\alpha_2-\alpha_3-n}{
							\alpha_2-\alpha_1+1,		\alpha_2-\alpha_3+1,\alpha_2+\beta+N+2}{1}}{\alpha_3-\alpha_2+n}&\frac{\pFq{4}{3}{-n,\alpha_2+\beta+3n+2,\alpha_2-\alpha_1-n,\alpha_2-\alpha_3-n+1}{
							\alpha_2-\alpha_1+1,		\alpha_2-\alpha_3+1,\alpha_2+\beta+N+2}{1}}{\alpha_2+\beta+3n+2}
				\end{vNiceArray}
			}
			{\begin{vNiceArray}{cc}[cell-space-limits=2pt]
					\frac{ \pFq{4}{3}{-n-1,\alpha_1+\beta+3n+4,\alpha_1-\alpha_2-n,\alpha_1-\alpha_3-n}{
							\alpha_1-\alpha_2+1,		\alpha_1-\alpha_3+1,\alpha_1+\beta+N+2}{1}}{n+1}&\frac{\pFq{4}{3}{-n,\alpha_1+\beta+3n+3,\alpha_1-\alpha_2-n,\alpha_1-\alpha_3-n}{
							\alpha_1-\alpha_2+1,		\alpha_1-\alpha_3+1,\alpha_1+\beta+N+2}{1}}{\alpha_1+\beta+3n+3	}\\
					\frac{\pFq{4}{3}{-n,\alpha_2+\beta+3n+4,\alpha_2-\alpha_1-n-1,\alpha_2-\alpha_3-n}{
							\alpha_2-\alpha_1+1,		\alpha_2-\alpha_3+1,\alpha_2+\beta+N+2}{1}}{\alpha_1-\alpha_2+n+1}&\frac{\pFq{4}{3}{-n,\alpha_2+\beta+3n+3,\alpha_2-\alpha_1-n,\alpha_2-\alpha_3-n}{
							\alpha_2-\alpha_1+1,		\alpha_2-\alpha_3+1,\alpha_2+\beta+N+2}{1}}{\alpha_2+\beta+3n+3}
			\end{vNiceArray}}
		\end{multlined}
	\end{aligned}\]
	\begin{align*}
			L_{3,3n+1}&=\begin{multlined}[t][.9\textwidth]
		-\tfrac{(N-3n-2)(\beta+3n+3)(\alpha_1+\beta+3n+2)_2(\alpha_2+\beta+3n+2)_2(\alpha_3+\beta+3n+2)_2}{(\alpha_1+\beta+4n+2)_3(\alpha_2+\beta+4n+1)_4(\alpha_3+\beta+4n+1)_4}\\\times \tfrac{(\alpha_1+n+1)(\alpha_1+\beta+N+n+2)(n+1)(\alpha_1-\alpha_2+n+1) (\alpha_1-\alpha_3+n+1)}{(\alpha_1+\beta+4n+1)_5}\\
			\times \tfrac{(\alpha_1+\beta+4n)(\alpha_2+\beta+4n)_4(\alpha_3+\beta+4n-1)_5(\alpha_1+\beta+4n+2)_2}{(N-3n+1)(\alpha_3+n)(\alpha_1+\beta+3n)(\alpha_2+\beta+3n)(\alpha_3+\beta+3n)n(\alpha_2-\alpha_1+n)}\\
			\times \frac{\begin{vNiceArray}{cc}[cell-space-limits=2pt]
					\frac{ \pFq{4}{3}{-n,\alpha_1+\beta+3n+3,\alpha_1-\alpha_2-n,\alpha_1-\alpha_3-n}{
							\alpha_1-\alpha_2+1,		\alpha_1-\alpha_3+1,\alpha_1+\beta+N+2}{1}}{\alpha_3-\alpha_1+n}&\frac{\pFq{4}{3}{-n,\alpha_1+\beta+3n+2,\alpha_1-\alpha_2-n,\alpha_1-\alpha_3-n+1}{
							\alpha_1-\alpha_2+1,		\alpha_1-\alpha_3+1,\alpha_1+\beta+N+2}{1}}{(\alpha_1+\beta+3n+2)	}\\
					\frac{\pFq{4}{3}{-n,\alpha_2+\beta+3n+3,\alpha_2-\alpha_1-n,\alpha_2-\alpha_3-n}{
							\alpha_2-\alpha_1+1,		\alpha_2-\alpha_3+1,\alpha_2+\beta+N+2}{1}}{\alpha_3-\alpha_2+n}&\frac{\pFq{4}{3}{-n,\alpha_2+\beta+3n+2,\alpha_2-\alpha_1-n,\alpha_2-\alpha_3-n+1}{
							\alpha_2-\alpha_1+1,		\alpha_2-\alpha_3+1,\alpha_2+\beta+N+2}{1}}{\alpha_2+\beta+3n+2}
			\end{vNiceArray}}{\begin{vNiceArray}{cc}[cell-space-limits=2pt]
					\frac{\pFq{4}{3}{-n,\alpha_1+\beta+3n+2,\alpha_1-\alpha_2-n,\alpha_1-\alpha_3-n+1}{
							\alpha_1-\alpha_2+1,		\alpha_1-\alpha_3+1,\alpha_1+\beta+N+2}{1}}{ \alpha_2-\alpha_1+n}&\frac{\pFq{4}{3}{-n,\alpha_1+\beta+3n+1,\alpha_1-\alpha_2-n+1,\alpha_1-\alpha_3-n+1}{
							\alpha_1-\alpha_2+1,		\alpha_1-\alpha_3+1,\alpha_1+\beta+N+2}{1}}{n(\alpha_1+\beta+3n+1)	}\\
					\pFq{4}{3}{-n,\alpha_2+\beta+3n+2,\alpha_2-\alpha_1-n,\alpha_2-\alpha_3-n+1}{
						\alpha_2-\alpha_1+1,		\alpha_2-\alpha_3+1,\alpha_2+\beta+N+2}{1}&\frac{\pFq{4}{3}{-n+1,\alpha_2+\beta+3n+1,\alpha_2-\alpha_1-n,\alpha_2-\alpha_3-n+1}{
							\alpha_2-\alpha_1+1,		\alpha_2-\alpha_3+1,\alpha_2+\beta+N+2}{1}}{\alpha_2+\beta+3n+1}
				\end{vNiceArray}
			},\end{multlined} \\
			L_{3,3n+2}&			= \begin{multlined}[t][.9\textwidth]-\tfrac{(N-3n-3)(\beta+3n+4)(\alpha_1+\beta+3n+3)_2(\alpha_2+\beta+3n+3)_2(\alpha_3+\beta+3n+4) (\alpha_2+\beta+N+n+2)}{(N-3n)(\alpha_1+\beta+3n+1)(\alpha_2+\beta+3n+1)(\alpha_3+\beta+3n+1)(\alpha_2+\beta+4n+2)_4}\\
			\times \tfrac{(n+1)(\alpha_2+n+1)(\alpha_2-\alpha_1+n+1)(\alpha_2-\alpha_3+n+1)(\alpha_1+\beta+4n+1)_2 (\alpha_2+\beta+4n+1) (\alpha_3+\beta+4n+1)}{(\alpha_1+n+1)(\alpha_3-\alpha_2+n)(\alpha_3-\alpha_1+n)(\alpha_1+\beta+4n+6)(\alpha_2+\beta+4n+5)_2(\alpha_3+\beta+4n+5)}\\
			\times \tfrac{\begin{vNiceArray}{cc}[cell-space-limits=2pt]
					\frac{ \pFq{4}{3}{-n-1,\alpha_1+\beta+3n+4,\alpha_1-\alpha_2-n,\alpha_1-\alpha_3-n}{
							\alpha_1-\alpha_2+1,		\alpha_1-\alpha_3+1,\alpha_1+\beta+N+2}{1}}{n+1}&\frac{\pFq{4}{3}{-n,\alpha_1+\beta+3n+3,\alpha_1-\alpha_2-n,\alpha_1-\alpha_3-n}{
							\alpha_1-\alpha_2+1,		\alpha_1-\alpha_3+1,\alpha_1+\beta+N+2}{1}}{\alpha_1+\beta+3n+3	}\\
					\frac{\pFq{4}{3}{-n,\alpha_2+\beta+3n+4,\alpha_2-\alpha_1-n-1,\alpha_2-\alpha_3-n}{
							\alpha_2-\alpha_1+1,		\alpha_2-\alpha_3+1,\alpha_2+\beta+N+2}{1}}{\alpha_1-\alpha_2+n+1}&\frac{\pFq{4}{3}{-n,\alpha_2+\beta+3n+3,\alpha_2-\alpha_1-n,\alpha_2-\alpha_3-n}{
							\alpha_2-\alpha_1+1,		\alpha_2-\alpha_3+1,\alpha_2+\beta+N+2}{1}}{\alpha_2+\beta+3n+3}
			\end{vNiceArray}}{\begin{vNiceArray}{cc}[cell-space-limits=2pt]
					\frac{ \pFq{4}{3}{-n,\alpha_1+\beta+3n+3,\alpha_1-\alpha_2-n,\alpha_1-\alpha_3-n}{
							\alpha_1-\alpha_2+1,		\alpha_1-\alpha_3+1,\alpha_1+\beta+N+2}{1}}{\alpha_3-\alpha_1+n}&\frac{\pFq{4}{3}{-n,\alpha_1+\beta+3n+2,\alpha_1-\alpha_2-n,\alpha_1-\alpha_3-n+1}{
							\alpha_1-\alpha_2+1,		\alpha_1-\alpha_3+1,\alpha_1+\beta+N+2}{1}}{(\alpha_1+\beta+3n+2)	}\\
					\frac{\pFq{4}{3}{-n,\alpha_2+\beta+3n+3,\alpha_2-\alpha_1-n,\alpha_2-\alpha_3-n}{
							\alpha_2-\alpha_1+1,		\alpha_2-\alpha_3+1,\alpha_2+\beta+N+2}{1}}{\alpha_3-\alpha_2+n}&\frac{\pFq{4}{3}{-n,\alpha_2+\beta+3n+2,\alpha_2-\alpha_1-n,\alpha_2-\alpha_3-n+1}{
							\alpha_2-\alpha_1+1,		\alpha_2-\alpha_3+1,\alpha_2+\beta+N+2}{1}}{\alpha_2+\beta+3n+2}
				\end{vNiceArray}
			},\end{multlined}\\
			L_{3,3n+3}&= \begin{multlined}[t][.9\textwidth]-\tfrac{(N-3n-4)(\beta+3n+5)(\alpha_1+\beta+3n+4)_2(\alpha_2+\beta+3n+4)_2(\alpha_3+\beta+3n+5) }{(N-3n-1)(\alpha_1+\beta+3n+2)(\alpha_2+\beta+3n+2)(\alpha_3+\beta+3n+2)(\alpha_3+\beta+4n+3)_5}\\
			\times \tfrac{(n+1)(\alpha_3+n+1)(\alpha_3-\alpha_1+n+1)(\alpha_3-\alpha_2+n+1)(\alpha_3+\beta+N+n+2) }{(\alpha_2+n+1)(\alpha_1-\alpha_2+n+1)}\\
			\times \tfrac{(\alpha_1+\beta+4n+3)(\alpha_2+\beta+4n+2)_2(\alpha_3+\beta+4n+2)}{(\alpha_1+\beta+4n+7)(\alpha_2+\beta+4n+7)(\alpha_3+\beta+4n+6)}\\
			\times
			\frac{\begin{vNiceArray}{cc}[cell-space-limits=2pt]
					\frac{\pFq{4}{3}{-n-1,\alpha_1+\beta+3n+5,\alpha_1-\alpha_2-n-1,\alpha_1-\alpha_3-n}{
							\alpha_1-\alpha_2+1,		\alpha_1-\alpha_3+1,\alpha_1+\beta+N+2}{1}}{ \alpha_2-\alpha_1+n+1}&\frac{\pFq{4}{3}{-n-1,\alpha_1+\beta+3n+4,\alpha_1-\alpha_2-n,\alpha_1-\alpha_3-n}{
							\alpha_1-\alpha_2+1,		\alpha_1-\alpha_3+1,\alpha_1+\beta+N+2}{1}}{(n+1)(\alpha_1+\beta+3n+4)	}\\
					\pFq{4}{3}{-n-1,\alpha_2+\beta+3n+5,\alpha_2-\alpha_1-n-1,\alpha_2-\alpha_3-n}{
						\alpha_2-\alpha_1+1,		\alpha_2-\alpha_3+1,\alpha_2+\beta+N+2}{1}&\frac{\pFq{4}{3}{-n,\alpha_2+\beta+3n+4,\alpha_2-\alpha_1-n-1,\alpha_2-\alpha_3-n}{
							\alpha_2-\alpha_1+1,		\alpha_2-\alpha_3+1,\alpha_2+\beta+N+2}{1}}{\alpha_2+\beta+3n+4}
			\end{vNiceArray}}{\begin{vNiceArray}{cc}[cell-space-limits=2pt]
					\frac{ \pFq{4}{3}{-n-1,\alpha_1+\beta+3n+4,\alpha_1-\alpha_2-n,\alpha_1-\alpha_3-n}{
							\alpha_1-\alpha_2+1,		\alpha_1-\alpha_3+1,\alpha_1+\beta+N+2}{1}}{n+1}&\frac{\pFq{4}{3}{-n,\alpha_1+\beta+3n+3,\alpha_1-\alpha_2-n,\alpha_1-\alpha_3-n}{
							\alpha_1-\alpha_2+1,		\alpha_1-\alpha_3+1,\alpha_1+\beta+N+2}{1}}{\alpha_1+\beta+3n+3	}\\
					\frac{\pFq{4}{3}{-n,\alpha_2+\beta+3n+4,\alpha_2-\alpha_1-n-1,\alpha_2-\alpha_3-n}{
							\alpha_2-\alpha_1+1,		\alpha_2-\alpha_3+1,\alpha_2+\beta+N+2}{1}}{\alpha_1-\alpha_2+n+1}&\frac{\pFq{4}{3}{-n,\alpha_2+\beta+3n+3,\alpha_2-\alpha_1-n,\alpha_2-\alpha_3-n}{
							\alpha_2-\alpha_1+1,		\alpha_2-\alpha_3+1,\alpha_2+\beta+N+2}{1}}{\alpha_2+\beta+3n+3}
				\end{vNiceArray}
		}
		\end{multlined}
	\end{align*}
\end{teo}
\begin{proof}

From \eqref{HahnTypeIIWeighted0} we get
\[
\begin{aligned}
	B_{3n}(0)
	&=(-1)^{n}\scriptstyle
	\frac{N!}{(N-3n)!}\frac{(\alpha_1+1)_{n}(\alpha_2+1)_{n}(\alpha_3+1)_{n}}{(\alpha_1+\beta+3n+1)_{n}(\alpha_2+\beta+3n+1)_{n}(\alpha_3+\beta+3n+1)_{n}},\\
	B_{3n+1}(0)
&=(-1)^{n+1}\scriptstyle
\frac{N!}{(N-3n-1)!}\frac{(\alpha_1+1)_{n+1}(\alpha_2+1)_{n}(\alpha_3+1)_{n}}{(\alpha_1+\beta+3n+2)_{n+1}(\alpha_2+\beta+3n+2)_{n}(\alpha_3+\beta+3n+2)_{n}},\\
B_{3n+2}(0)
&=(-1)^{n}\scriptstyle
\frac{N!}{(N-3n-2)!}\frac{(\alpha_1+1)_{n+1}(\alpha_2+1)_{n+1}(\alpha_3+1)_{n}}{(\alpha_1+\beta+3n+3)_{n+1}(\alpha_2+\beta+3n+3)_{n+1}(\alpha_3+\beta+3n+3)_{n}},\\
\end{aligned}
\]
and
\begin{align*}
	A_{3n-1}^{(1)}(0)
	&=\scriptstyle
	\frac{(-1)^{n-1}(N+1-3n)!
		({\alpha}_1+\beta+3n)_{n}({\alpha}_2+\beta+3n)_{n}({\alpha}_3+\beta+3n)_{n}}{
	(n-1)!(\beta+1)_{3n-1}(\alpha_1+\beta+3n)_{N+2-3n}	({\alpha}_2-{\alpha}_1)_{{n}}({\alpha}_3-{\alpha}_1)_{{n}}}\pFq{4}{3}{-n+1,\alpha_1+\beta+3n,\alpha_1-\alpha_2-n+1,\alpha_1-\alpha_3-n+1}{
		\alpha_1-\alpha_2+1,		\alpha_1-\alpha_3+1,\alpha_1+\beta+N+2}{1},\\
	A_{3n-1}^{(2)}(0)
&	=\scriptstyle
	\frac{(-1)^{n-1}(N+1-3n)!
		({\alpha}_1+\beta+3n)_{n}({\alpha}_2+\beta+3n)_{n}({\alpha}_3+\beta+3n)_{n}}{
	(n-1)!(\beta+1)_{3n-1}(\alpha_2+\beta+3n)_{N+2-3n}	({\alpha}_1-{\alpha}_2)_{{n}}({\alpha}_3-{\alpha}_2)_{{n}}}\pFq{4}{3}{-n+1,\alpha_2+\beta+3n,\alpha_2-\alpha_1-n+1,\alpha_2-\alpha_3-n+1}{
		\alpha_2-\alpha_1+1,		\alpha_2-\alpha_3+1,\alpha_2+\beta+N+2}{1},\\
	A_{3n-1}^{(3)}(0)
&	=\scriptstyle
	\frac{(-1)^{n-1}(N+1-3n)!
		({\alpha}_1+\beta+3n)_{n}({\alpha}_2+\beta+3n)_{n}({\alpha}_3+\beta+3n)_{n}}{
		(n-1)!(\beta+1)_{3n-1}(\alpha_3+\beta+3n)_{N+2-3n}({\alpha}_1-{\alpha}_3)_{{n}}({\alpha}_2-{\alpha}_3)_{{n}}}\pFq{4}{3}{-n+1,\alpha_3+\beta+3n,\alpha_3-\alpha_1-n+1,\alpha_3-\alpha_2-n+1}{
		\alpha_3-\alpha_1+1,		\alpha_3-\alpha_2+1,\alpha_3+\beta+N+2}{1},\\
	A_{3n}^{(1)}(0)
&	=\scriptstyle
	\frac{(-1)^{n}(N-3n)!
		({\alpha}_1+\beta+3n+1)_{n+1}({\alpha}_2+\beta+3n+1)_{n}({\alpha}_3+\beta+3n+1)_{n}}{
	n!(\beta+1)_{3n}(\alpha_1+\beta+3n+1)_{N+1-3n}	({\alpha}_2-{\alpha}_1)_{{n}}({\alpha}_3-{\alpha}_1)_{{n}}}\pFq{4}{3}{-n,\alpha_1+\beta+3n+1,\alpha_1-\alpha_2-n+1,\alpha_1-\alpha_3-n+1}{
		\alpha_1-\alpha_2+1,		\alpha_1-\alpha_3+1,\alpha_1+\beta+N+2}{1},\\
	A_{3n}^{(2)}(0)
&	=\scriptstyle
	\frac{(-1)^{n}(N-3n)!
		({\alpha}_1+\beta+3n+1)_{n+1}({\alpha}_2+\beta+3n+1)_{n}({\alpha}_3+\beta+3n+1)_{n}}{
		(n-1)!(\beta+1)_{3n}(\alpha_2+\beta+3n+1)_{N+1-3n}({\alpha}_1-{\alpha}_2)_{{n+1}}({\alpha}_3-{\alpha}_2)_{{n}}}\pFq{4}{3}{-n+1,\alpha_2+\beta+3n+1,\alpha_2-\alpha_1-n,\alpha_2-\alpha_3-n+1}{
		\alpha_2-\alpha_1+1,		\alpha_2-\alpha_3+1,\alpha_2+\beta+N+2}{1},\\
	A_{3n}^{(3)}(0)
&	=
\scriptstyle	\frac{(-1)^{n}(N-3n)!
		({\alpha}_1+\beta+3n+1)_{n+1}({\alpha}_2+\beta+3n+1)_{n}({\alpha}_3+\beta+3n+1)_{n}}{
	(n-1)!(\beta+1)_{3n}(\alpha_3+\beta+3n+1)_{N+1-3n}	({\alpha}_1-{\alpha}_3)_{{n+1}}({\alpha}_2-{\alpha}_3)_{{n}}}\pFq{4}{3}{-n+1,\alpha_3+\beta+3n+1,\alpha_3-\alpha_1-n,\alpha_3-\alpha_2-n+1}{
		\alpha_3-\alpha_1+1,		\alpha_3-\alpha_2+1,\alpha_3+\beta+N+2}{1},\\
	A_{3n+1}^{(1)}(0)
&	=\scriptstyle
	\frac{(-1)^{n+1}(N-3n-1)!
		({\alpha}_1+\beta+3n+2)_{n+1}({\alpha}_2+\beta+3n+2)_{n+1}({\alpha}_3+\beta+3n+2)_{n}}{
		n!(\beta+1)_{3n+1}(\alpha_1+\beta+3n+2)_{N-3n}	({\alpha}_2-{\alpha}_1)_{{n+1}}({\alpha}_3-{\alpha}_1)_{{n}}}\pFq{4}{3}{-n,\alpha_1+\beta+3n+2,\alpha_1-\alpha_2-n,\alpha_1-\alpha_3-n+1}{
		\alpha_1-\alpha_2+1,		\alpha_1-\alpha_3+1,\alpha_1+\beta+N+2}{1},\\
	A_{3n+1}^{(2)}(0)
	&=\scriptstyle
	\frac{(-1)^{n+1}(N-3n-1)!
		({\alpha}_1+\beta+3n+2)_{n+1}({\alpha}_2+\beta+3n+2)_{n+1}({\alpha}_3+\beta+3n+2)_{n}}{
		n!(\beta+1)_{3n+1}(\alpha_2+\beta+3n+2)_{N-3n}({\alpha}_1-{\alpha}_2)_{{n+1}}({\alpha}_3-{\alpha}_2)_{{n}}}\pFq{4}{3}{-n,\alpha_2+\beta+3n+2,\alpha_2-\alpha_1-n,\alpha_2-\alpha_3-n+1}{
		\alpha_2-\alpha_1+1,		\alpha_2-\alpha_3+1,\alpha_2+\beta+N+2}{1},\\
	A_{3n+1}^{(3)}(0)
	&=\scriptstyle
	\frac{(-1)^{n+1}(N-3n-1)!
		({\alpha}_1+\beta+3n+2)_{n+1}({\alpha}_2+\beta+3n+2)_{n+1}({\alpha}_3+\beta+3n+2)_{n}}{
		(n-1)!(\beta+1)_{3n+1}(\alpha_3+\beta+3n+2)_{N-3n}	({\alpha}_1-{\alpha}_3)_{{n+1}}({\alpha}_2-{\alpha}_3)_{{n+1}}}\pFq{4}{3}{-n+1,\alpha_3+\beta+3n+2,\alpha_3-\alpha_1-n,\alpha_3-\alpha_2-n}{
		\alpha_3-\alpha_1+1,		\alpha_3-\alpha_2+1,\alpha_3+\beta+N+2}{1}.
\end{align*}
For for \(j\in\{0,1,2,3\}\), the recurrence coefficients are
\[\begin{aligned}
	b_{3m}^j&=\scriptstyle-(\alpha_{1}+m+1)\delta_{j,0}
		+
		\frac{(N-3m+1)_{j}(\beta+3m+1-j)_{j}({\alpha}_{1}+\beta+4m+1-j)}{{\prod}_{l=5-j}^{4}({\alpha}_l+\beta+4m-j)}
		\frac{\prod_{l=1}^3({\alpha}_l+\beta+3m+1-j)_{j}}{\prod_{l=1}^{3}(\alpha_l+\beta+4m+1-j)_{j}}\sum_{i=1}^{4-j}
		\frac{(\alpha_i+m)(\alpha_i+\beta+N+m+1)}{(\alpha_i+\beta+4m-j)_{j+2}}
		\frac{\prod_{l=1}^3(\alpha_i-\alpha_l+m)}{\prod_{l=1,l\neq i}^{4-j}(\alpha_i-{\alpha}_l)}
\\
	b_{3m+1}^j&=\scriptstyle-(\alpha_2+m+1)\delta_{j,0}
		+
		\frac{(N-3m)_{j}(\beta+3m+2-j)_{j}({\alpha}_2+\beta+4m+2-j)}{{\prod}_{l=6-j}^{5}({\alpha}_l+\beta+4m+{1}-j)}
		\frac{\prod_{l=1}^3({\alpha}_l+\beta+3m+2-j)_{j}}{\prod_{l=2}^{4}(\alpha_l+\beta+4m+2-j)_{j}}\sum_{i=2}^{5-j}
		\frac{(\alpha_i+m)(\alpha_i+\beta+N+m+1)}{(\alpha_i+\beta+4m+1-j)_{j+2}}
		\frac{\prod_{l=1}^3(\alpha_i-\alpha_l+m)}{\prod_{l=2,l\neq i}^{5-j}(\alpha_i-{\alpha}_l)},
\\
		b_{3m+2}^j&=\scriptstyle-(\alpha_3+m+1)\delta_{j,0}
		+
		\frac{(N-3m-1)_{j}(\beta+3m+3-j)_{j}({\alpha}_3+\beta+4m+3-j)}{{\prod}_{l=7-j}^{6}({\alpha}_l+\beta+4m+{2}-j)}
		\frac{\prod_{l=1}^3({\alpha}_l+\beta+3m+3-j)_{j}}{\prod_{l=3}^{5}(\alpha_l+\beta+4m+3-j)_{j}}\sum_{i=3}^{6-j}
		\frac{(\alpha_i+m)(\alpha_i+\beta+N+m+1)}{(\alpha_i+\beta+4m+2-j)_{j+2}}
		\frac{\prod_{l=1}^3(\alpha_i-\alpha_l+m)}{\prod_{l=3,l\neq i}^{6-j}(\alpha_i-{\alpha}_l)}.
\end{aligned}\]
In particular
\[\begin{aligned}
	b_{3m}^3&=\scriptstyle
		\frac{(N-3m+1)_{3}(\beta+3m-2)_{3}({\alpha}_{1}+\beta+4m-2)}{{\prod}_{l=2}^{4}({\alpha}_l+\beta+4m-3)}
		\frac{\prod_{l=1}^3({\alpha}_l+\beta+3m-2)_{3}}{\prod_{l=1}^{3}(\alpha_l+\beta+4m-2)_{3}}
		\frac{(\alpha_1+m)(\alpha_1+\beta+N+m+1)}{(\alpha_1+\beta+4m-3)_{5}}
	\prod_{l=1}^3(\alpha_1-\alpha_l+m)
\\
	b_{3m+1}^3&=\scriptstyle
		\frac{(N-3m)_{3}(\beta+3m-1)_{3}({\alpha}_2+\beta+4m-1)}{{\prod}_{l=3}^{5}({\alpha}_l+\beta+4m-2)}
		\frac{\prod_{l=1}^3({\alpha}_l+\beta+3m-1)_{3}}{\prod_{l=2}^{4}(\alpha_l+\beta+4m-1)_{3}}
		\frac{(\alpha_2+m)(\alpha_2+\beta+N+m+1)}{(\alpha_2+\beta+4m-2)_{5}}
		\prod_{l=1}^3(\alpha_2-\alpha_l+m),
\\
		b_{3m+2}^3&=\scriptstyle
		\frac{(N-3m-1)_{3}(\beta+3m)_{3}({\alpha}_3+\beta+4m)}{{\prod}_{l=4}^{6}({\alpha}_l+\beta+4m-1)}
		\frac{\prod_{l=1}^3({\alpha}_l+\beta+3m)_{3}}{\prod_{l=3}^{5}(\alpha_l+\beta+4m)_{3}}
		\frac{(\alpha_3+m)(\alpha_3+\beta+N+m+1)}{(\alpha_3+\beta+4m-1)_{5}}
		\prod_{l=1}^3(\alpha_3-\alpha_l+m).
\end{aligned}\]

The corresponding $\tau$-determinants are given by
\[
\begin{aligned}
	\tau^B_{1,n}&=B_n(0), & n&\in\{0,1,\dots, N-2\},&
	\tau^A_{1,n}&=A^{(1)}_n(0),& n&\in\{0,1,\dots, N-4\},\\
\end{aligned}
\]
and 
\[
\begin{aligned}
	\tau^A_{2,n}&=\begin{vNiceArray}{cc}[cell-space-limits=2pt,small]
		A^{(1)}_{n+1}(0) &A^{(1)}_n(0)\\
		A^{(2)}_{n+1}(0) &A^{(2)}_n(0)
	\end{vNiceArray}, & n&\in\{0,1,\dots, N-5\},&
	\tau^A_{3,n}&=\begin{vNiceArray}{ccc}[cell-space-limits=2pt,small]
		A^{(1)}_{n+2}(0) &	A^{(1)}_{n+1}(0) &A^{(1)}_n(0)\\
		A^{(2)}_{n+2}(0) &	A^{(2)}_{n+1}(0) &A^{(2)}_n(0)\\
		A^{(3)}_{n+2}(0) &	A^{(3)}_{n+1}(0) &A^{(3)}_n(0)
	\end{vNiceArray}, & n&\in\{0,1,\dots, N-6\}.
\end{aligned}
\]
Thus, following Theorem \ref{th:bidiagonal_Christoffel} we get
\[ \begin{aligned}
	U_{1,n}&=-\frac{B_{n+1}(0)}{B_{n}(0)},   &n&\in\{0,1,\dots,N-2\},&
	L_{1,n+1}&=-\frac{A^{(1)}_{n}(0)}{A^{(1)}_{n+1}(0)}, &n&\in\{0,1,\dots,N-4\},
\end{aligned}\]
and 
\[ \begin{aligned}
\scriptstyle	L_{2,n+1}&=-\scriptstyle\frac{A^{(1)}_{n+2}(0)}{A^{(1)}_{n+1}(0)}\frac{\begin{vNiceArray}{cc}[cell-space-limits=2pt,small]
			A^{(1)}_{n+1}(0) &A^{(1)}_n(0)\\
			A^{(2)}_{n+1}(0) &A^{(2)}_n(0)
	\end{vNiceArray}}{\begin{vNiceArray}{cc}[cell-space-limits=2pt,small]
			A^{(1)}_{n+2}(0) &A^{(1)}_{n+1}(0)\\
			A^{(2)}_{n+2}(0) &A^{(2)}_{n+1}(0)
	\end{vNiceArray}}&&\scriptstyle n\in\{0,1,\dots,N-5\},&
\scriptstyle	L_{3,n+1}&=-\frac{\begin{vNiceArray}{cc}[cell-space-limits=2pt,small]
			A^{(1)}_{n+3}(0) &A^{(1)}_{n+2}(0)\\
			A^{(2)}_{n+3}(0) &A^{(2)}_{n+2}(0)
	\end{vNiceArray}}{\begin{vNiceArray}{cc}[cell-space-limits=2pt,small]
			A^{(1)}_{n+2}(0) &A^{(1)}_{n+1}(0)\\
			A^{(2)}_{n+2}(0) &A^{(2)}_{n+1}(0)
	\end{vNiceArray}}\frac{\begin{vNiceArray}{ccc}[cell-space-limits=2pt,small]
			A^{(1)}_{n+2}(0) &	A^{(1)}_{n+1}(0) &A^{(1)}_n(0)\\
			A^{(2)}_{n+2}(0) &	A^{(2)}_{n+1}(0) &A^{(2)}_n(0)\\
			A^{(3)}_{n+2}(0) &	A^{(3)}_{n+1}(0) &A^{(3)}_n(0)
	\end{vNiceArray}}{\begin{vNiceArray}{ccc}[cell-space-limits=2pt,small]
			A^{(1)}_{n+3}(0) &	A^{(1)}_{n+2}(0) &A^{(1)}_{n+1}(0)\\
			A^{(2)}_{n+3}(0) &	A^{(2)}_{n+2}(0) &A^{(2)}_{n+1}(0)\\
			A^{(3)}_{n+3}(0) &	A^{(3)}_{n+2}(0) &A^{(3)}_{n+1}(0)
	\end{vNiceArray}},  &&\scriptstyle n\in\{0,1,\dots,N-6\}.
\end{aligned}\]
In particular
\[ \begin{aligned}
	L_{3,n+1}&=-b^3_{n+3}\frac{\begin{vNiceArray}{cc}[cell-space-limits=2pt]
			A^{(1)}_{n+3}(0) &A^{(1)}_{n+2}(0)\\
			A^{(2)}_{n+3}(0) &A^{(2)}_{n+2}(0)
	\end{vNiceArray}}{\begin{vNiceArray}{cc}[cell-space-limits=2pt]
			A^{(1)}_{n+2}(0) &A^{(1)}_{n+1}(0)\\
			A^{(2)}_{n+2}(0) &A^{(2)}_{n+1}(0)
	\end{vNiceArray}}\frac{B_{n}(0)}{B_{n+1}(0)},  &n\in\{0,1,\dots,N-6\}.
\end{aligned}\]

The expressions for the $U$'s are obtained form 
\[\begin{aligned}
	U_{1,3n}&=-\frac{B_{3n+1}(0)}{B_{3n}(0)}, &U_{1,3n+1}&=-\frac{B_{3n+2}(0)}{B_{3n+1}(0)}, & U_{1,3n+2}&=-\frac{B_{3n+3}(0)}{B_{3n+2}(0)}.
\end{aligned}
\]
For the $L_1$'s we obtain:
\begin{multline*}
	L_{1,3n+1}=-\frac{A^{(1)}_{3n}(0)}{A^{(1)}_{3n+1}(0)}=-\frac{\frac{(-1)^{n}(N-3n)!
			({\alpha}_1+\beta+3n+1)_{n+1}({\alpha}_2+\beta+3n+1)_{n}({\alpha}_3+\beta+3n+1)_{n}}{
			n!(\beta+1)_{3n}(\alpha_1+\beta+3n+1)_{N+1-3n}	({\alpha}_2-{\alpha}_1)_{{n}}({\alpha}_3-{\alpha}_1)_{{n}}}}{\frac{(-1)^{n+1}(N-3n-1)!
			({\alpha}_1+\beta+3n+2)_{n+1}({\alpha}_2+\beta+3n+2)_{n+1}({\alpha}_3+\beta+3n+2)_{n}}{
			n!(\beta+1)_{3n+1}(\alpha_1+\beta+3n+2)_{N-3n}	({\alpha}_2-{\alpha}_1)_{{n+1}}({\alpha}_3-{\alpha}_1)_{{n}}}}\\\times 
		\frac{\pFq{4}{3}{-n,\alpha_1+\beta+3n+1,\alpha_1-\alpha_2-n+1,\alpha_1-\alpha_3-n+1}{
				\alpha_1-\alpha_2+1,		\alpha_1-\alpha_3+1,\alpha_1+\beta+N+2}{1}}{\pFq{4}{3}{-n,\alpha_1+\beta+3n+2,\alpha_1-\alpha_2-n,\alpha_1-\alpha_3-n+1}{
				\alpha_1-\alpha_2+1,		\alpha_1-\alpha_3+1,\alpha_1+\beta+N+2}{1}}
\end{multline*}
\begin{multline*}
	L_{1,3n+2}=-\frac{A^{(1)}_{3n+1}(0)}{A^{(1)}_{3n+2}(0)}=-\frac{	\frac{(-1)^{n+1}(N-3n-1)!
			({\alpha}_1+\beta+3n+2)_{n+1}({\alpha}_2+\beta+3n+2)_{n+1}({\alpha}_3+\beta+3n+2)_{n}}{
			n!(\beta+1)_{3n+1}(\alpha_1+\beta+3n+2)_{N-3n}	({\alpha}_2-{\alpha}_1)_{{n+1}}({\alpha}_3-{\alpha}_1)_{{n}}}}{	\frac{(-1)^{n}(N+1-3(n+1))!
			({\alpha}_1+\beta+3(n+1))_{n+1}({\alpha}_2+\beta+3(n+1))_{n+1}({\alpha}_3+\beta+3(n+1))_{n+1}}{
			n!(\beta+1)_{3(n+1)-1}(\alpha_1+\beta+3(n+1))_{N+2-3(n+1)}	({\alpha}_2-{\alpha}_1)_{{n+1}}({\alpha}_3-{\alpha}_1)_{{n+1}}}}\\\times
		\frac{\pFq{4}{3}{-n,\alpha_1+\beta+3n+2,\alpha_1-\alpha_2-n,\alpha_1-\alpha_3-n+1}{
				\alpha_1-\alpha_2+1,		\alpha_1-\alpha_3+1,\alpha_1+\beta+N+2}{1}}{\pFq{4}{3}{-(n+1)+1,\alpha_1+\beta+3(n+1),\alpha_1-\alpha_2-(n+1)+1,\alpha_1-\alpha_3-(n+1)+1}{
				\alpha_1-\alpha_2+1,		\alpha_1-\alpha_3+1,\alpha_1+\beta+N+2}{1}}
\end{multline*}
\begin{multline*}
	L_{1,3n+3}=-\frac{A^{(1)}_{3n+2}(0)}{A^{(1)}_{3n+3}(0)}=-\frac{\frac{(-1)^{n}(N+1-3(n+1))!
			({\alpha}_1+\beta+3(n+1))_{n+1}({\alpha}_2+\beta+3(n+1))_{n+1}({\alpha}_3+\beta+3(n+1))_{n+1}}{
			n!(\beta+1)_{3(n+1)-1}(\alpha_1+\beta+3(n+1))_{N+2-3(n+1)}	({\alpha}_2-{\alpha}_1)_{{n+1}}({\alpha}_3-{\alpha}_1)_{{n+1}}}}{\frac{(-1)^{n+1}(N-3(n+1))!
			({\alpha}_1+\beta+3(n+1)+1)_{n+2}({\alpha}_2+\beta+3(n+1)+1)_{n+1}({\alpha}_3+\beta+3(n+1)+1)_{n+1}}{
			(n+1)!(\beta+1)_{3(n+1)}(\alpha_1+\beta+3(n+1)+1)_{N+1-3(n+1)}	({\alpha}_2-{\alpha}_1)_{{n+1}}({\alpha}_3-{\alpha}_1)_{{n+1}}}}\\
		\times\frac{\pFq{4}{3}{-(n+1)+1,\alpha_1+\beta+3(n+1),\alpha_1-\alpha_2-(n+1)+1,\alpha_1-\alpha_3-(n+1)+1}{
				\alpha_1-\alpha_2+1,		\alpha_1-\alpha_3+1,\alpha_1+\beta+N+2}{1}}{\pFq{4}{3}{-(n+1),\alpha_1+\beta+3(n+1)+1,\alpha_1-\alpha_2-(n+1)+1,\alpha_1-\alpha_3-(n+1)+1}{
				\alpha_1-\alpha_2+1,		\alpha_1-\alpha_3+1,\alpha_1+\beta+N+2}{1}}
\end{multline*}
That simplifies to  the exspressions for the $L_1$'s in the theorem.

Now, we need to deal with $2\times 2$ determinants
\[
\begin{aligned}
	\tau^A_{2,3n\textcolor{black}{+1}}&=\begin{vNiceArray}{cc}[cell-space-limits=2pt]
	A^{(1)}_{3n+1}(0) &A^{(1)}_{3n}(0)\\
	A^{(2)}_{3n+1}(0) &A^{(2)}_{3n}(0)
\end{vNiceArray},&	\tau^A_{2,3n+2}&=\begin{vNiceArray}{cc}[cell-space-limits=2pt]
A^{(1)}_{3n+2}(0) &A^{(1)}_{3n+1}(0)\\
A^{(2)}_{3n+2}(0) &A^{(2)}_{3n+1}(0)
\end{vNiceArray},&	\tau^A_{2,3n+3}&=\begin{vNiceArray}{cc}[cell-space-limits=2pt]
A^{(1)}_{3n+3}(0) &A^{(1)}_{3n+2}(0)\\
A^{(2)}_{3n+3}(0) &A^{(2)}_{3n+2}(0)
\end{vNiceArray},
\end{aligned}
\]
Hence, 
\begin{multline*}
	\tau^A_{2,3n\textcolor{black}{+1}}=\tfrac{(-1)^{n}(N-3n)!
		({\alpha}_1+\beta+3n+1)_{n+1}({\alpha}_2+\beta+3n+1)_{n}({\alpha}_3+\beta+3n+1)_{n}}{
		(n-1)!(\beta+1)_{3n}}\\\tfrac{(-1)^{n+1}(N-3n-1)!
		({\alpha}_1+\beta+3n+2)_{n+1}({\alpha}_2+\beta+3n+2)_{n+1}({\alpha}_3+\beta+3n+2)_{n}}{
		n!(\beta+1)_{3n+1}}\\\begin{vNiceArray}{cc}[cell-space-limits=2pt]
		\frac{ \pFq{4}{3}{-n,\alpha_1+\beta+3n+2,\alpha_1-\alpha_2-n,\alpha_1-\alpha_3-n+1}{
				\alpha_1-\alpha_2+1,		\alpha_1-\alpha_3+1,\alpha_1+\beta+N+2}{1}}{(\alpha_1+\beta+3n+2)_{N-3n}	({\alpha}_2-{\alpha}_1)_{{n+1}}({\alpha}_3-{\alpha}_1)_{{n}}}&\frac{\pFq{4}{3}{-n,\alpha_1+\beta+3n+1,\alpha_1-\alpha_2-n+1,\alpha_1-\alpha_3-n+1}{
				\alpha_1-\alpha_2+1,		\alpha_1-\alpha_3+1,\alpha_1+\beta+N+2}{1}}{n(\alpha_1+\beta+3n+1)_{N+1-3n}	({\alpha}_2-{\alpha}_1)_{{n}}({\alpha}_3-{\alpha}_1)_{{n}}}\\
	\frac{ \pFq{4}{3}{-n,\alpha_2+\beta+3n+2,\alpha_2-\alpha_1-n,\alpha_2-\alpha_3-n+1}{
			\alpha_2-\alpha_1+1,		\alpha_2-\alpha_3+1,\alpha_2+\beta+N+2}{1}}{(\alpha_2+\beta+3n+2)_{N-3n}({\alpha}_1-{\alpha}_2)_{{n+1}}({\alpha}_3-{\alpha}_2)_{{n}}}&\frac{\pFq{4}{3}{-n+1,\alpha_2+\beta+3n+1,\alpha_2-\alpha_1-n,\alpha_2-\alpha_3-n+1}{
				\alpha_2-\alpha_1+1,		\alpha_2-\alpha_3+1,\alpha_2+\beta+N+2}{1}}{(\alpha_2+\beta+3n+1)_{N+1-3n}({\alpha}_1-{\alpha}_2)_{{n +1}}({\alpha}_3-{\alpha}_2)_{{n}}}
		\end{vNiceArray}\\=-\tfrac{(N-3n)!
		({\alpha}_1+\beta+3n+1)_{n+1}({\alpha}_2+\beta+3n+1)_{n}({\alpha}_3+\beta+3n+1)_{n}}{
		(n-1)!(\beta+1)_{3n}({\alpha}_2-{\alpha}_1)_{{\textcolor{black}{  n } }}({\alpha}_3-{\alpha}_1)_{{n}}(\alpha_1+\beta+3n+2)_{N-3n}}\\\tfrac{(N-3n-1)!
		({\alpha}_1+\beta+3n+2)_{n+1}({\alpha}_2+\beta+3n+2)_{n+1}({\alpha}_3+\beta+3n+2)_{n}}{
		n!(\beta+1)_{3n+1}({\alpha}_1-{\alpha}_2)_{{\textcolor{black}{  n+1 }  }}({\alpha}_3-{\alpha}_2)_{{n}}(\alpha_2+\beta+3n+2)_{N-3n}}\\\begin{vNiceArray}{cc}[cell-space-limits=2pt]
	 \frac{\pFq{4}{3}{-n,\alpha_1+\beta+3n+2,\alpha_1-\alpha_2-n,\alpha_1-\alpha_3-n+1}{
				\alpha_1-\alpha_2+1,		\alpha_1-\alpha_3+1,\alpha_1+\beta+N+2}{1}}{ \alpha_2-\alpha_1+n}&\frac{\pFq{4}{3}{-n,\alpha_1+\beta+3n+1,\alpha_1-\alpha_2-n+1,\alpha_1-\alpha_3-n+1}{
				\alpha_1-\alpha_2+1,		\alpha_1-\alpha_3+1,\alpha_1+\beta+N+2}{1}}{n(\alpha_1+\beta+3n+1)	}\\
		\pFq{4}{3}{-n,\alpha_2+\beta+3n+2,\alpha_2-\alpha_1-n,\alpha_2-\alpha_3-n+1}{
				\alpha_2-\alpha_1+1,		\alpha_2-\alpha_3+1,\alpha_2+\beta+N+2}{1}&\frac{\pFq{4}{3}{-n+1,\alpha_2+\beta+3n+1,\alpha_2-\alpha_1-n,\alpha_2-\alpha_3-n+1}{
				\alpha_2-\alpha_1+1,		\alpha_2-\alpha_3+1,\alpha_2+\beta+N+2}{1}}{\alpha_2+\beta+3n+1}
				\end{vNiceArray}
\end{multline*}

\begin{multline*}
	\tau^A_{2,3n+2}=-\tfrac{(N-3n-2)!
		({\alpha}_1+\beta+3n+3)_{n+1}({\alpha}_2+\beta+3n+3)_{n+1}({\alpha}_3+\beta+3n+3)_{n+1}}{
		(n)!(\beta+1)_{3n+2} ({\alpha}_1+\beta+3n+3)_{N-3n-1} (\alpha_2-\alpha_1)_{n+1}(\alpha_3-\alpha_1)_n}\\
		\tfrac{(N-3n-1)!
		({\alpha}_1+\beta+3n+2)_{n+1}({\alpha}_2+\beta+3n+2)_{n+1}({\alpha}_3+\beta+3n+2)_{n}}{
		n!(\beta+1)_{3n+1}({\alpha}_2+\beta+3n+3)_{N-3n-1} (\alpha_1-\alpha_2)_{n+1}(\alpha_3-\alpha_2)_n}
		\\\begin{vNiceArray}{cc}[cell-space-limits=2pt]
	\frac{ \pFq{4}{3}{-n,\alpha_1+\beta+3n+3,\alpha_1-\alpha_2-n,\alpha_1-\alpha_3-n}{
				\alpha_1-\alpha_2+1,		\alpha_1-\alpha_3+1,\alpha_1+\beta+N+2}{1}}{\alpha_3-\alpha_1+n}&\frac{\pFq{4}{3}{-n,\alpha_1+\beta+3n+2,\alpha_1-\alpha_2-n,\alpha_1-\alpha_3-n+1}{
				\alpha_1-\alpha_2+1,		\alpha_1-\alpha_3+1,\alpha_1+\beta+N+2}{1}}{(\alpha_1+\beta+3n+2)	}\\
		\frac{\pFq{4}{3}{-n,\alpha_2+\beta+3n+3,\alpha_2-\alpha_1-n,\alpha_2-\alpha_3-n}{
				\alpha_2-\alpha_1+1,		\alpha_2-\alpha_3+1,\alpha_2+\beta+N+2}{1}}{\alpha_3-\alpha_2+n}&\frac{\pFq{4}{3}{-n,\alpha_2+\beta+3n+2,\alpha_2-\alpha_1-n,\alpha_2-\alpha_3-n+1}{
				\alpha_2-\alpha_1+1,		\alpha_2-\alpha_3+1,\alpha_2+\beta+N+2}{1}}{\alpha_2+\beta+3n+2}
				\end{vNiceArray}
\end{multline*}
\begin{multline*}
	\tau^A_{2,3n}=-\tfrac{(N-3n)!
		({\alpha}_1+\beta+3n+1)_{n+1}({\alpha}_2+\beta+3n+1)_{n}({\alpha}_3+\beta+3n+1)_{n}}{
		(n-1)!(\beta+1)_{3n} ({\alpha}_1+\beta+3n+1)_{N-3n+1} (\alpha_2-\alpha_1)_{n}(\alpha_3-\alpha_1)_{n}}\\
		\tfrac{(N-3n+1)!
		({\alpha}_1+\beta+3n)_{n}({\alpha}_2+\beta+3n)_{n}({\alpha}_3+\beta+3n)_{n}}{
		(n-1)!(\beta+1)_{3n-1}({\alpha}_2+\beta+3n+1)_{N-3n+1} (\alpha_1-\alpha_2)_{n}(\alpha_3-\alpha_2)_{n}}
		\\\begin{vNiceArray}{cc}[cell-space-limits=2pt]
	\frac{ \pFq{4}{3}{-n,\alpha_1+\beta+3n+1,\alpha_1-\alpha_2-n+1,\alpha_1-\alpha_3-n+1}{
				\alpha_1-\alpha_2+1,		\alpha_1-\alpha_3+1,\alpha_1+\beta+N+2}{1}}{n}&\frac{\pFq{4}{3}{-n+1,\alpha_1+\beta+3n,\alpha_1-\alpha_2-n+1,\alpha_1-\alpha_3-n+1}{
				\alpha_1-\alpha_2+1,		\alpha_1-\alpha_3+1,\alpha_1+\beta+N+2}{1}}{\alpha_1+\beta+3n	}\\
		\frac{\pFq{4}{3}{-n+1,\alpha_2+\beta+3n+1,\alpha_2-\alpha_1-n,\alpha_2-\alpha_3-n+1}{
				\alpha_2-\alpha_1+1,		\alpha_2-\alpha_3+1,\alpha_2+\beta+N+2}{1}}{\alpha_1-\alpha_2+n}&\frac{\pFq{4}{3}{-n+1,\alpha_2+\beta+3n,\alpha_2-\alpha_1-n+1,\alpha_2-\alpha_3-n+1}{
				\alpha_2-\alpha_1+1,		\alpha_2-\alpha_3+1,\alpha_2+\beta+N+2}{1}}{\alpha_2+\beta+3n}
				\end{vNiceArray}
\end{multline*}
\begin{multline*}
	\tau^A_{2,3n+3}=-\tfrac{(N-3n-3)!
		({\alpha}_1+\beta+3n+4)_{n+2}({\alpha}_2+\beta+3n+4)_{n+1}({\alpha}_3+\beta+3n+4)_{n+1}}{
		(n)!(\beta+1)_{3n+3} ({\alpha}_1+\beta+3n+4)_{N-3n-2} (\alpha_2-\alpha_1)_{n+1}(\alpha_3-\alpha_1)_{n+1}}\\
		\tfrac{(N-3n-2)!
		({\alpha}_1+\beta+3n+3)_{n+1}({\alpha}_2+\beta+3n+3)_{n+1}({\alpha}_3+\beta+3n+3)_{n+1}}{
		n!(\beta+1)_{3n+2}({\alpha}_2+\beta+3n+4)_{N-3n-2} (\alpha_1-\alpha_2)_{n+1}(\alpha_3-\alpha_2)_{n+1}}
		\\\begin{vNiceArray}{cc}[cell-space-limits=2pt]
	\tfrac{ \pFq{4}{3}{-n-1,\alpha_1+\beta+3n+4,\alpha_1-\alpha_2-n,\alpha_1-\alpha_3-n}{
				\alpha_1-\alpha_2+1,		\alpha_1-\alpha_3+1,\alpha_1+\beta+N+2}{1}}{n+1}&\frac{\pFq{4}{3}{-n,\alpha_1+\beta+3n+3,\alpha_1-\alpha_2-n,\alpha_1-\alpha_3-n}{
				\alpha_1-\alpha_2+1,		\alpha_1-\alpha_3+1,\alpha_1+\beta+N+2}{1}}{\alpha_1+\beta+3n+3	}\\
		\tfrac{\pFq{4}{3}{-n,\alpha_2+\beta+3n+4,\alpha_2-\alpha_1-n-1,\alpha_2-\alpha_3-n}{
				\alpha_2-\alpha_1+1,		\alpha_2-\alpha_3+1,\alpha_2+\beta+N+2}{1}}{\alpha_1-\alpha_2+n+1}&\frac{\pFq{4}{3}{-n,\alpha_2+\beta+3n+3,\alpha_2-\alpha_1-n,\alpha_2-\alpha_3-n}{
				\alpha_2-\alpha_1+1,		\alpha_2-\alpha_3+1,\alpha_2+\beta+N+2}{1}}{\alpha_2+\beta+3n+3}
				\end{vNiceArray}
\end{multline*}
\begin{multline*}
\frac{\tau^A_{2,3n}}{\tau^A_{2,3n+1}}=\tfrac{ n (N-3n)_2 (\beta+3n)_2 (\alpha_1+\beta+3n)(\alpha_2+\beta+3n)(\alpha_3+\beta+3n)_2 (\alpha_1-\alpha_2+n)}{(\alpha_1+\beta+4n)_3(\alpha_2+\beta+4n)_3(\alpha_3+\beta+4n)_2}\\
\times 
 \tfrac{\begin{vNiceArray}{cc}[cell-space-limits=2pt]
	\frac{ \pFq{4}{3}{-n,\alpha_1+\beta+3n+1,\alpha_1-\alpha_2-n+1,\alpha_1-\alpha_3-n+1}{
				\alpha_1-\alpha_2+1,		\alpha_1-\alpha_3+1,\alpha_1+\beta+N+2}{1}}{n}&\frac{\pFq{4}{3}{-n+1,\alpha_1+\beta+3n,\alpha_1-\alpha_2-n+1,\alpha_1-\alpha_3-n+1}{
				\alpha_1-\alpha_2+1,		\alpha_1-\alpha_3+1,\alpha_1+\beta+N+2}{1}}{\alpha_1+\beta+3n	}\\
		\frac{\pFq{4}{3}{-n+1,\alpha_2+\beta+3n+1,\alpha_2-\alpha_1-n,\alpha_2-\alpha_3-n+1}{
				\alpha_2-\alpha_1+1,		\alpha_2-\alpha_3+1,\alpha_2+\beta+N+2}{1}}{\alpha_1-\alpha_2+n}&\frac{\pFq{4}{3}{-n+1,\alpha_2+\beta+3n,\alpha_2-\alpha_1-n+1,\alpha_2-\alpha_3-n+1}{
				\alpha_2-\alpha_1+1,		\alpha_2-\alpha_3+1,\alpha_2+\beta+N+2}{1}}{\alpha_2+\beta+3n}
				\end{vNiceArray}
}{\begin{vNiceArray}{cc}[cell-space-limits=2pt]
	 \frac{\pFq{4}{3}{-n,\alpha_1+\beta+3n+2,\alpha_1-\alpha_2-n,\alpha_1-\alpha_3-n+1}{
				\alpha_1-\alpha_2+1,		\alpha_1-\alpha_3+1,\alpha_1+\beta+N+2}{1}}{ \alpha_2-\alpha_1+n}&\frac{\pFq{4}{3}{-n,\alpha_1+\beta+3n+1,\alpha_1-\alpha_2-n+1,\alpha_1-\alpha_3-n+1}{
				\alpha_1-\alpha_2+1,		\alpha_1-\alpha_3+1,\alpha_1+\beta+N+2}{1}}{n(\alpha_1+\beta+3n+1)	}\\
		\pFq{4}{3}{-n,\alpha_2+\beta+3n+2,\alpha_2-\alpha_1-n,\alpha_2-\alpha_3-n+1}{
				\alpha_2-\alpha_1+1,		\alpha_2-\alpha_3+1,\alpha_2+\beta+N+2}{1}&\frac{\pFq{4}{3}{-n+1,\alpha_2+\beta+3n+1,\alpha_2-\alpha_1-n,\alpha_2-\alpha_3-n+1}{
				\alpha_2-\alpha_1+1,		\alpha_2-\alpha_3+1,\alpha_2+\beta+N+2}{1}}{\alpha_2+\beta+3n+1}
				\end{vNiceArray}}
\end{multline*}
\begin{multline*}
\frac{\tau^A_{2,3n+1}}{\tau^A_{2,3n+2}} =\tfrac{ n (N-3n-1)_2 (\beta+1+3n)_2 (\alpha_1+\beta+3n+1)(\alpha_2+\beta+3n+1)(\alpha_3+\beta+3n+1)_2}{(\alpha_1+\beta+4n+2)_2(\alpha_2+\beta+4n+1)_3(\alpha_3+\beta+4n+1)_3}\\
\times 
(\alpha_2-\alpha_1+n)
\frac{\begin{vNiceArray}{cc}[cell-space-limits=2pt]
	 \frac{\pFq{4}{3}{-n,\alpha_1+\beta+3n+2,\alpha_1-\alpha_2-n,\alpha_1-\alpha_3-n+1}{
				\alpha_1-\alpha_2+1,		\alpha_1-\alpha_3+1,\alpha_1+\beta+N+2}{1}}{ \alpha_2-\alpha_1+n}&\frac{\pFq{4}{3}{-n,\alpha_1+\beta+3n+1,\alpha_1-\alpha_2-n+1,\alpha_1-\alpha_3-n+1}{
				\alpha_1-\alpha_2+1,		\alpha_1-\alpha_3+1,\alpha_1+\beta+N+2}{1}}{n(\alpha_1+\beta+3n+1)	}\\
		\pFq{4}{3}{-n,\alpha_2+\beta+3n+2,\alpha_2-\alpha_1-n,\alpha_2-\alpha_3-n+1}{
				\alpha_2-\alpha_1+1,		\alpha_2-\alpha_3+1,\alpha_2+\beta+N+2}{1}&\frac{\pFq{4}{3}{-n+1,\alpha_2+\beta+3n+1,\alpha_2-\alpha_1-n,\alpha_2-\alpha_3-n+1}{
				\alpha_2-\alpha_1+1,		\alpha_2-\alpha_3+1,\alpha_2+\beta+N+2}{1}}{\alpha_2+\beta+3n+1}
				\end{vNiceArray}
}
{\begin{vNiceArray}{cc}[cell-space-limits=2pt]
	\frac{ \pFq{4}{3}{-n,\alpha_1+\beta+3n+3,\alpha_1-\alpha_2-n,\alpha_1-\alpha_3-n}{
				\alpha_1-\alpha_2+1,		\alpha_1-\alpha_3+1,\alpha_1+\beta+N+2}{1}}{\alpha_3-\alpha_1+n}&\frac{\pFq{4}{3}{-n,\alpha_1+\beta+3n+2,\alpha_1-\alpha_2-n,\alpha_1-\alpha_3-n+1}{
				\alpha_1-\alpha_2+1,		\alpha_1-\alpha_3+1,\alpha_1+\beta+N+2}{1}}{(\alpha_1+\beta+3n+2)	}\\
		\frac{\pFq{4}{3}{-n,\alpha_2+\beta+3n+3,\alpha_2-\alpha_1-n,\alpha_2-\alpha_3-n}{
				\alpha_2-\alpha_1+1,		\alpha_2-\alpha_3+1,\alpha_2+\beta+N+2}{1}}{\alpha_3-\alpha_2+n}&\frac{\pFq{4}{3}{-n,\alpha_2+\beta+3n+2,\alpha_2-\alpha_1-n,\alpha_2-\alpha_3-n+1}{
				\alpha_2-\alpha_1+1,		\alpha_2-\alpha_3+1,\alpha_2+\beta+N+2}{1}}{\alpha_2+\beta+3n+2}
				\end{vNiceArray}}
\end{multline*}
\begin{multline*}
\frac{\tau^A_{2,3n+2}}{\tau^A_{2,3n+3}} =\tfrac{ (N-3n-2)_2 (\beta+3n+2)_2 (\alpha_1+\beta+3n+2)(\alpha_2+\beta+3n+2)(\alpha_3+\beta+3n+2)_2}{(\alpha_1+\beta+4n+3)_3(\alpha_2+\beta+4n+3)_2(\alpha_3+\beta+4n+2)_3}\\
\times
(\alpha_3-\alpha_1+n)(\alpha_3-\alpha_2+n)
\frac{\begin{vNiceArray}{cc}[cell-space-limits=2pt]
	\frac{ \pFq{4}{3}{-n,\alpha_1+\beta+3n+3,\alpha_1-\alpha_2-n,\alpha_1-\alpha_3-n}{
				\alpha_1-\alpha_2+1,		\alpha_1-\alpha_3+1,\alpha_1+\beta+N+2}{1}}{\alpha_3-\alpha_1+n}&\frac{\pFq{4}{3}{-n,\alpha_1+\beta+3n+2,\alpha_1-\alpha_2-n,\alpha_1-\alpha_3-n+1}{
				\alpha_1-\alpha_2+1,		\alpha_1-\alpha_3+1,\alpha_1+\beta+N+2}{1}}{(\alpha_1+\beta+3n+2)	}\\
		\frac{\pFq{4}{3}{-n,\alpha_2+\beta+3n+3,\alpha_2-\alpha_1-n,\alpha_2-\alpha_3-n}{
				\alpha_2-\alpha_1+1,		\alpha_2-\alpha_3+1,\alpha_2+\beta+N+2}{1}}{\alpha_3-\alpha_2+n}&\frac{\pFq{4}{3}{-n,\alpha_2+\beta+3n+2,\alpha_2-\alpha_1-n,\alpha_2-\alpha_3-n+1}{
				\alpha_2-\alpha_1+1,		\alpha_2-\alpha_3+1,\alpha_2+\beta+N+2}{1}}{\alpha_2+\beta+3n+2}
				\end{vNiceArray}
}
{\begin{vNiceArray}{cc}[cell-space-limits=2pt]
	\frac{ \pFq{4}{3}{-n-1,\alpha_1+\beta+3n+4,\alpha_1-\alpha_2-n,\alpha_1-\alpha_3-n}{
				\alpha_1-\alpha_2+1,		\alpha_1-\alpha_3+1,\alpha_1+\beta+N+2}{1}}{n+1}&\frac{\pFq{4}{3}{-n,\alpha_1+\beta+3n+3,\alpha_1-\alpha_2-n,\alpha_1-\alpha_3-n}{
				\alpha_1-\alpha_2+1,		\alpha_1-\alpha_3+1,\alpha_1+\beta+N+2}{1}}{\alpha_1+\beta+3n+3	}\\
		\frac{\pFq{4}{3}{-n,\alpha_2+\beta+3n+4,\alpha_2-\alpha_1-n-1,\alpha_2-\alpha_3-n}{
				\alpha_2-\alpha_1+1,		\alpha_2-\alpha_3+1,\alpha_2+\beta+N+2}{1}}{\alpha_1-\alpha_2+n+1}&\frac{\pFq{4}{3}{-n,\alpha_2+\beta+3n+3,\alpha_2-\alpha_1-n,\alpha_2-\alpha_3-n}{
				\alpha_2-\alpha_1+1,		\alpha_2-\alpha_3+1,\alpha_2+\beta+N+2}{1}}{\alpha_2+\beta+3n+3}
				\end{vNiceArray}}
\end{multline*}
Now we can write the other lower bidiagonal entries as follows. Note that
\[
\begin{aligned}
	L_{1,n+1}&=-\frac{A^{(1)}_{n}(0)}{A^{(1)}_{n+1}(0)}, &n\in\{0,1,\dots,N-4\},\\
L_{2,n+1}&=-\frac{1}{	L_{1,n+2}}\frac{\tau^A_{2,n}}{\tau^A_{2,n+1}}&n\in\{0,1,\dots,N-5\}\\
L_{3,n+1}&=-\frac{b^3_{n+3}}{L_{2,n+2}L_{1,n+3}U_{1,n-1}}=-\frac{b^3_{n+3}}{U_{1,n-1}}\frac{\tau^A_{2,n+2}}{\tau^A_{2,n+1}}&n\in\{0,1,\dots,N-6\}.
\end{aligned}\]
Hence, using 
\[\begin{aligned}
	L_{2, 3n+1}&=-\frac{1}{	L_{1,3n+2}}\frac{\tau^A_{2,3n}}{\tau^A_{2,3n+1}}, &
	L_{2, 3n+2}&=-\frac{1}{	L_{1,3n+3}}\frac{\tau^A_{2,3n+1}}{\tau^A_{2,3n+2}}, &
	L_{2, 3n+3}&=-\frac{1}{	L_{1,3n+4}}\frac{\tau^A_{2,3n+2}}{\tau^A_{2,3n+3}}
\end{aligned}\]
we obtain the $L_2$'s. Finally, from
\[
\begin{aligned}
	L_{3,3n+1}&=-\frac{b^3_{3n+3}}{U_{1,3n-1}}\frac{\tau^A_{2,3n+2}}{\tau^A_{2,3n+1}}, &
	L_{3,3n+2}&=-\frac{b^3_{3n+4}}{U_{1,3n}}\frac{\tau^A_{2,3n+3}}{\tau^A_{2,3n+2}}, &
	L_{3,3n+3}&=-\frac{b^3_{3n+5}}{U_{1,3n+1}}\frac{\tau^A_{2,3n+4}}{\tau^A_{2,3n+3}}.
\end{aligned}\]
we obtain the $L_3$'s. 
\end{proof}

\section*{Declarations}

\subsection*{Competing interests}
The authors declare that they have no competing interests.

\subsection*{Author contributions}
All authors contributed to the conception of the work, the development of the theoretical results, and the writing of the manuscript. All authors read and approved the final manuscript.

\subsection*{Use of artificial intelligence}
Artificial intelligence tools were used solely to improve the presentation (e.g.\ language and readability). All AI-assisted edits were reviewed and validated by the authors, who remain fully responsible for the content of the manuscript.

\subsection*{Data availability}
No datasets were generated or analysed during the current study.

\subsection*{Code availability}
No code was used or generated for the current study.

\subsection*{Ethics approval}
Not applicable.

\subsection*{Consent to participate}
Not applicable.

\subsection*{Consent for publication}
Not applicable.

\section*{Acknowledgments}

AB was financially supported by the Funda\c c\~ao para a Ci\^encia e a Tecnologia (Portuguese Foundation for Science and Technology) under the scope of the projects
\hyperref{https://doi.org/10.54499/UID/00324/2025}{}{}{\texttt{UID/00324/2025}}
(Centre for Mathematics of the University of Coimbra).

AF acknowledges CIDMA Center for Research and Development in Mathematics and Applications (University of Aveiro) and the Portuguese Foundation for Science and Technology (FCT) within project 
\hyperref{https://doi.org/10.54499/UID/04106/2025}{}{}{\texttt{UID/04106/2025}}.

MM acknowledges Spanish ``Agencia Estatal de Investigación'' research projects [PID2021- 122154NB-I00], \emph{Ortogonalidad y Aproximación con Aplicaciones en Machine Learning y Teoría de la Probabilidad} and [PID2024-155133NB-I00], \emph{Ortogonalidad, aproximación e integrabilidad: aplicaciones en procesos estocásticos clásicos y cuánticos}.

\end{document}